\documentclass[reqno]{amsart}
\usepackage{amssymb}
\usepackage{hyperref}
\usepackage{latexsym}
\usepackage{amsmath}

\usepackage[usenames,dvipsnames,svgnames,table]{xcolor}

\usepackage{amsfonts, amsmath, bbm, amssymb, amsthm}
\usepackage{enumerate}
\usepackage{todonotes}

\newcommand{\Z}{\mathbb{Z}}
\newcommand{\R}{\mathbb{R}}
\newcommand{\C}{\mathbb{C}}
\newcommand{\N}{\mathbb{N}}

\newcommand{\tr}{\boldsymbol{t}_\Sigma}

\newcommand{\dom}{\textup{dom }}
\newcommand{\ran}{\textup{ran }}
\newcommand{\sign}{\textup{sign}}

\newcommand{\supp}{ \textup{supp }}

\newcommand{\abs}[1]{\lvert{#1}\rvert}

\newcommand{\norm}[1]{{\lVert{#1}\rVert}}

\theoremstyle{plain}
\newtheorem{theo}{Theorem}[section]
\newtheorem{lemma}[theo]{Lemma}
\newtheorem{cor}[theo]{Corollary}
\newtheorem{proposition}[theo]{Proposition}

\theoremstyle{definition}

\theoremstyle{remark}
\newtheorem{remark}[theo]{Remark}

\numberwithin{equation}{section}

\begin{document}

	\title[Approximation of Dirac operators with $\delta$-shell potentials]{Approximation of Dirac operators with $\boldsymbol{\delta}$-shell potentials in the norm resolvent sense, II.~Quantitative results}

	\author[J. Behrndt]{Jussi Behrndt}
	\address{Technische Universit\"{a}t Graz\\
		Institut f\"ur Angewandte Mathematik\\
		Steyrergasse 30\\
		8010 Graz, Austria}
	\email{behrndt@tugraz.at}
	
	\author[M. Holzmann]{Markus Holzmann}
	\address{Technische Universit\"{a}t Graz\\
		Institut f\"ur Angewandte Mathematik\\
		Steyrergasse 30\\
		8010 Graz, Austria}
	\email{holzmann@math.tugraz.at}

	\author[C. Stelzer-Landauer]{Christian Stelzer-Landauer}
	\address{Technische Universit\"{a}t Graz\\
		Institut f\"ur Angewandte Mathematik\\
		Steyrergasse 30\\
		8010 Graz, Austria}
	\email{christian.stelzer09@gmail.com}
	
	\maketitle
	\begin{abstract}
		This paper is devoted to the approximation of two and three-dimensional Dirac operators $H_{\widetilde{V} \delta_\Sigma}$ with combinations of electrostatic and Lorentz scalar $\delta$-shell interactions in the norm resolvent sense. Relying on results from \cite{BHS23} an explicit smallness condition on the coupling parameters is derived so that $H_{\widetilde{V} \delta_\Sigma}$ is the limit of Dirac operators with scaled electrostatic and Lorentz scalar potentials. Via counterexamples it is shown that this condition is sharp. The approximation of $H_{\widetilde{V} \delta_\Sigma}$ for larger coupling constants is achieved by adding an additional scaled magnetic term.
	\end{abstract}
	
	\section{Introduction}
	Dirac operators  describe relativistic spin 1/2 particles in a quantum mechanical framework \cite{T92} and they play an important role in the mathematical description of graphene \cite{AB08, NGM04}. In order to model interactions  that are supported in a small neighbourhood of a curve in $\R^2$ or a surface in $\R^3$ Dirac operators with $\delta$-shell potentials are used. These operators  have been studied  intensively in the recent years, see for instance \cite{AMV14, AMV15, BEHL18, BEHL19, BHOP20, BHSS22, BHT23, Ben21, BP24, CLMT21,R22a}, and  are formally given by 
	\begin{equation}\label{eq_H_V_tilde_formal}
			H_{\widetilde{V} \delta_\Sigma} =
			-i(\alpha \cdot \nabla) 
			+ m \beta + \widetilde{V} \delta_\Sigma,
	\end{equation}
	where  $\alpha \cdot \nabla ,\beta$ are as in Section~\ref{sec_not}~\eqref{it_Dirac_matrices},
	$m \in \R$ is the mass of the underlying particle,  $\delta_\Sigma$ is the $\delta$-distribution supported on a $C^2$-smooth curve $\Sigma$ in $\R^2$ or surface in $\R^3$, and $\widetilde{V}$ is a symmetric matrix-valued function which models the  interaction on $\Sigma$. 

	The main objective of the present paper -- which is a continuation of our investigations in \cite{BHS23} -- is the approximation  of $H_{\widetilde{V} \delta_\Sigma}$ by Dirac operators with strongly localized potentials;
    here we complement the qualitative analysis in \cite{BHS23} with sharp quantitative results.
	We also refer the reader to \cite{BHS23} for a detailed introduction, overview and references on the topic.  The approximation of $H_{\widetilde{V} \delta_\Sigma}$ serves to justify the viewpoint that the expression in \eqref{eq_H_V_tilde_formal} is an idealized replacement for more realistic Dirac operators describing strongly localized interactions and was first considered in one dimension in \cite{S89}, see also \cite{HT22,HT23,Hu95,H97,H99,T20}. Furthermore, in the multidimensional setting strong resolvent convergence was shown in \cite{BHT23,CLMT21,MP18,Z23} under various assumptions on $\widetilde{V}$ and $\Sigma$. Finally, in our recent paper \cite{BHS23} it was shown that $H_{\widetilde{V} \delta_\Sigma}$ can be approximated in the norm resolvent sense, if the matrix 
	function $\widetilde{V}$ 
	satisfies an implicit smallness condition. In the present paper we replace this implicit condition by an  explicit and sharp condition for a specific class of interaction matrices, namely those that model electrostatic and Lorentz scalar interactions, see~\eqref{eq_V} and~\eqref{eq_scaling} below.

	Let us recall the problem setting and the main result from \cite{BHS23}. The space dimension is denoted by $\theta \in \{ 2, 3 \}$ and we set $N=2$ for $\theta = 2$ and $N=4$ for $\theta = 3$. Assume that $\Sigma$ is the boundary of a $C^2$-domain  $\Omega_+ \subset \R^{\theta}$ as in Section~\ref{sec_not}~\eqref{it_def_Sigma} and that $\nu$ is the unit normal vector field on $\Sigma$ which points outwards of $\Omega_+$. Note that $\Omega_+$ can be any bounded $C^2$-domain, but also a wide class of unbounded domains is allowed in our assumptions. Then, define 
	\begin{equation}\label{eq_iota}
		\iota: \Sigma \times \R \to \R^\theta, \quad \iota(x_\Sigma,t):= x_\Sigma + t \nu(x_\Sigma), \quad (x_\Sigma,t) \in \Sigma \times \R,
	\end{equation}
	 and for  $\varepsilon \in (0,\infty)$  set $\Omega_\varepsilon := \iota(\Sigma \times (-\varepsilon,\varepsilon))$, which is the so-called \textit{tubular neighbourhood} of $\Sigma$. It follows from the assumptions on $\Sigma$ that one can fix an $\varepsilon_1>0$  such that   $\iota|_{ \Sigma \times (-\varepsilon_1,\varepsilon_1)}$ is injective; cf.  \cite[Proposition~2.4]{BHS23}. Next, we fix a
	 symmetric matrix-valued function
	\begin{equation}\label{eq_q1}
		V = V^*\in W^1_\infty(\Sigma;\C^{N \times N}),
		\end{equation}
		where $W^1_\infty$ is the $L^\infty$-based Sobolev space of once weakly differentiable functions, and 
		\begin{equation}\label{eq_q2}
	q \in L^{\infty}((-1,1);[0,\infty))  \,\,\text{ such that }\,\,   \int_{-1}^{1} q(s)  \, ds = 1,
	\end{equation}
	and we introduce for $\varepsilon \in (0,\varepsilon_1)$ the strongly localized potentials
	\begin{equation}\label{eq_V_eps}
		V_\varepsilon(x) := 
		\begin{cases}
			\frac{1}{\varepsilon}V(x_\Sigma)q\left(\frac{t}{\varepsilon}\right),& \text{if } x = \iota(x_\Sigma,t) \in \Omega_\varepsilon,\\
			0,& \text{if } x \notin \Omega_\varepsilon.
		\end{cases}
	\end{equation} 
	For $m \in \R$ and $\varepsilon \in (0, \varepsilon_1)$ consider the operator
	\begin{equation}\label{eq_H_V_eps}
		\begin{split}\
			H_{V_\varepsilon} u &:= -i(\alpha \cdot \nabla) u + m \beta u + V_\varepsilon  u, \quad  \dom H_{V_\varepsilon} :=  H^1(\R^\theta;\C^N),
		\end{split}
	\end{equation}
	where $H^k$ denotes the $L^2$-based Sobolev space of $k$-times weakly differentiable functions.
	Note that $ H_{V_\varepsilon} = H + V_\varepsilon$ is self-adjoint in $L^2(\mathbb{R}^\theta; \mathbb{C}^N)$ as $V_\varepsilon \in L^\infty(\mathbb{R}^\theta; \mathbb{C}^{N \times N})$ is symmetric and bounded, and the free Dirac operator $H=-i(\alpha \cdot \nabla)  + m \beta $ is self-adjoint 
	in $L^2(\mathbb{R}^\theta; \mathbb{C}^N)$; cf. Section~\ref{sec_free_Dirac}.	
	
	Next, the rigorous mathematical definition of $H_{\widetilde{V}\delta_\Sigma}$ as an operator in $L^2(\mathbb{R}^\theta; \mathbb{C}^N)$ is recalled. We
	set $\Omega_- = \R^{\theta} \setminus \overline{\Omega_+}$, write $u_\pm := u|_{\Omega_\pm}$ 
	for
	$u: \mathbb{R}^\theta \rightarrow \mathbb{C}^N$, and the Dirichlet trace operator is denoted by $\mathbf{t}_\Sigma^\pm: H^1(\Omega_\pm; \mathbb{C}^N) \rightarrow H^{1/2}(\Sigma; \mathbb{C}^N)$.
Then, for $\widetilde{V} = \widetilde{V}^* \in L^\infty(\Sigma;\C^{N \times N})$ the differential operator $H_{\widetilde{V} \delta_\Sigma}$  in $L^2(\mathbb{R}^\theta; \mathbb{C}^N)$  is defined by 
\begin{equation}\label{eq_def_H_V_tilde}
	\begin{split}
		H_{\widetilde{V}\delta_\Sigma}u & := (-i(\alpha \cdot \nabla) + m \beta) u_+ \oplus (-i(\alpha \cdot \nabla ) + m \beta) u_-, \\
		\dom H_{\widetilde{V}\delta_\Sigma}&:= \biggl\{ u \in H^1 (\Omega_+;\C^N) \oplus H^1(\Omega_-; \mathbb{C}^N): \\
		& \quad \qquad i(\alpha \cdot \nu )(\tr^+ u_+  - \tr^- u_-) +  \frac{\widetilde{V}}{2}(\tr^+ u_+  + \tr^- u_-) =0 \biggr\}.
	\end{split}
\end{equation}
	Eventually, the application of $\cos(x)$ and $\textup{sinc}(x) = \frac{1}{x} \sin(x)$ to matrices is understood via power series.
	Now, we are ready to recall the main result from \cite{BHS23} in a compressed  form; note that $H_{V_\varepsilon}$ and $H_{\widetilde{V} \delta_\Sigma}$ are denoted  in  \cite{BHS23} by $H_\varepsilon$ and $H_{\widetilde{V}}$, respectively.

	\begin{theo}[{\cite[Theorem~1.1]{BHS23}}]\label{THEO_MAIN}
		Let  $V$ and $q$ be as in \eqref{eq_q1}--\eqref{eq_q2}. Furthermore, assume that
		\begin{equation}\label{cond234}
			\cos\bigl(\tfrac{1}{2}(\alpha \cdot \nu)V\bigr)^{-1} \in W^1_\infty(\Sigma;\C^{N\times N})
		\end{equation}
		and set $\widetilde V =  V \textup{sinc}\bigl(\tfrac{1}{2}(\alpha \cdot \nu)V\bigr)	\cos\bigl(\tfrac{1}{2}(\alpha \cdot \nu)V\bigr)^{-1}$. If
		 \begin{equation*}
			\norm{V}_{W^1_\infty(\Sigma;\C^{N\times N})}   \norm{q}_{L^\infty(\R)} 
		\end{equation*}
		is sufficiently small, 	then 	$H_{\widetilde{V}\delta_\Sigma}$  is self-adjoint in $L^2(\R^{\theta};\C^N)$ and $H_{V_\varepsilon}$ converges to $H_{\widetilde{V}\delta_\Sigma}$ in the norm resolvent sense as $\varepsilon \to 0$. 
	\end{theo}
	
	Similar to previous approximation results a nonlinear rescaling of $V$ to $\widetilde{V}$ is observed in the limit, which is often related to Klein's paradox; cf. \cite{CLMT21,MP18,S89}.
	
	A direct application of Theorem~\ref{THEO_MAIN} to check the convergence of $H_{V_\varepsilon}$ is difficult, since it contains a nonexplicit smallness assumption on $V$. 
	The main motivation  in the present paper is to make this smallness condition explicit, if $V$ is of the specific form
	\begin{equation}\label{eq_V}
			V = \eta I_N + \tau \beta,\qquad \eta,\tau \in C^1_b(\Sigma;\R);
	\end{equation}
	here $C^1_b(\Sigma;\R)$ is the space of continuously differentiable functions with bounded derivatives 
	(see Section~\ref{sec_not}~\eqref{gehtdas}).
	In fact, such potentials are of particular physical interest since  $\eta I_N$ and $\tau \beta$ model electrostatic and Lorentz scalar interactions, respectively, and they are, in the context of $\delta$-shell potentials, studied in \cite{AMV14, AMV15,BEHL19,BHOP20,BHSS22,Ben21, BP24, R22a}. 
	Moreover, for $V$ as in~\eqref{eq_V} the condition \eqref{cond234} and the rescaling $V \mapsto \widetilde{V}$ simplify substantially, as then the anti-commutation rules for the Dirac matrices yield $((\alpha \cdot \nu) V)^2 = d I_N$, with $d = \eta^2 - \tau^2$, and thus,  the power series representations of $\cos$ and $\textup{sinc}$ lead
	to
	\begin{equation*}
		\cos\bigl(\tfrac{1}{2}(\alpha \cdot \nu)V\bigr) = \cos\bigl(\tfrac{\sqrt{d}}{2}\bigr) I_N \quad \textup{and} \quad \textup{sinc}\bigl(\tfrac{1}{2}(\alpha \cdot \nu)V\bigr) = \textup{sinc}\bigl(\tfrac{\sqrt{d}}{2}\bigr) I_N.
	\end{equation*}
	Hence, \eqref{cond234} simplifies to
	\begin{equation}\label{eq_scaling_cond}
		\inf_{x_\Sigma \in \Sigma, k \in \N_0} \abs{d(x_\Sigma)- (2k+1)^2 \pi^2} >0
	\end{equation}
	and we have
	\begin{equation}\label{eq_scaling}
		\widetilde{V} = \widetilde{\eta} I_N + \widetilde{\tau} \beta, \qquad (\widetilde{\eta},\widetilde{\tau}) = \textup{tanc}\bigl(\tfrac{\sqrt{d}}{2}\bigr)(\eta,\tau),
	\end{equation}
	where $\textup{tanc}(x) = \frac{\textup{tan}(x)}{x}$, see also Section~\ref{sec_not}~\eqref{it_tanc}.  Moreover, in the present situation $H_{\widetilde{V}\delta_\Sigma}$ is self-adjoint in $L^2(\R^\theta;\C^N)$ if 
	\begin{equation}\label{eq_d_tilde_non_crit}
		\inf_{x_\Sigma \in \Sigma} |\widetilde{d}(x_\Sigma)-4| > 0, \qquad  \widetilde{d} = \widetilde{\eta}^2 - \widetilde{\tau}^2,
	\end{equation}
	see  \cite[Section~6]{R22a}, which is called the \textit{noncritical} case.

In the main result of this paper, Theorem~\ref{theo_main_1}, we prove that if $V$ is chosen as in \eqref{eq_V}, then the  explicit smallness condition 
\begin{equation}\label{eq_d_cond_int}
	\sup_{x_\Sigma \in \Sigma} d(x_\Sigma) < \frac{\pi^2}{4}, \qquad d = \eta^2 - \tau^2,
\end{equation}
already guarantees norm resolvent convergence of $H_{V_\varepsilon}$ to $H_{\widetilde{V} \delta_\Sigma}$. 
Note that~\eqref{eq_d_cond_int} via \eqref{eq_scaling} implies 
\begin{equation}\label{supjussi}
\sup_{x_\Sigma \in \Sigma}\vert \widetilde{d}(x_\Sigma)\vert <4
\end{equation}
and hence
\eqref{eq_d_tilde_non_crit} is satisfied. This is already a first indication that it is not possible to relax the condition~\eqref{eq_d_cond_int}, as otherwise Dirac operators with critical $\delta$-shell potentials could be approximated, which have very different spectral properties as in the noncritical case.
However, to approximate $H_{\widetilde{V}\delta_\Sigma}$ also for $\widetilde{V}$ with 
\begin{equation}\label{infjussi}
\inf_{x_\Sigma \in \Sigma} |\widetilde{d}(x_\Sigma)| > 4
\end{equation}
we make use of a well-known unitary transformation.
Its approximation  
results, in a similar way as it was also described in \cite[Section~8]{CLMT21}, in an additional strongly localized magnetic potential $V_{m,\varepsilon}$ that is added to  $H_{V_\varepsilon}$. In Theorem~\ref{theo_main_2} we then show that $H_{V_\varepsilon} + V_{m,\varepsilon}$  also converges in the norm resolvent sense to a Dirac operator with a combination of electrostatic and Lorentz scalar $\delta$-shell potentials, where the rescaling of $\eta$ and $\tau$ is different.
In fact, as a consequence of this result and Theorem~\ref{theo_main_1} we obtain Corollary~\ref{cor_main_1}, which states that every Dirac operator with a given $\delta$-shell potential $\widetilde{V} \delta_\Sigma$, $\widetilde{V} = \widetilde{\eta} I_N + \widetilde{\tau} \beta$, $\widetilde{\eta},\widetilde{\tau} \in C_b^1(\Sigma;\R)$, and $\widetilde{d} = \widetilde{\eta}^2 -\widetilde{\tau}^2$ fulfilling \eqref{supjussi} or \eqref{infjussi}
can  be approximated by a sequence of Dirac operators with strongly localized potentials in the norm resolvent sense.

In Section~\ref{sec_counter} we show that the condition \eqref{eq_d_cond_int} is sharp, i.e. that $H_{V_\varepsilon}$ does, in general, not converge to $H_{\widetilde{V} \delta_\Sigma}$, if \eqref{eq_d_cond_int} is not satisfied. To do so, we provide suitable counterexamples under the assumption $\eta,\tau \in \R$. In Theorem~\ref{theo_main_counter_critical} we consider the case $\widetilde{d} =4 $ for a compact surface $\Sigma \subset \R^\theta$ and in Theorem~\ref{theo_main_3} the situation $\widetilde{d} > 4$ for $\Sigma = \{0\} \times \R$. In both cases we are able to show that if \eqref{eq_d_cond_int} is not fulfilled, then norm resolvent convergence would yield contradictory spectral implications for the limit operator $H_{\widetilde{V}\delta_\Sigma}$.

Finally, let us describe the structure of the paper. We start by introducing various notations and conventions in Section~\ref{sec_not}. Then, in Section~\ref{sec_expl_cond} we prove the main results of this paper on the approximation of $H_{\widetilde{V} \delta_\Sigma}$ by Dirac operators with squeezed potentials in the norm resolvent sense, and we provide explicit conditions on the coefficients such that the associated convergence results hold. The proof of Theorem~\ref{theo_main_1}  relies on the same statement for the special case where $\Sigma$ is a rotated graph and the technically complicated long proof of this special case is outsourced to Section~\ref{sec_C^2_b_graphs}.   After Section~\ref{sec_expl_cond}, we complement in  Section~\ref{sec_counter} the results of Section~\ref{sec_expl_cond} by providing counterexamples which show that \eqref{eq_d_cond_int} is sharp.

\subsection{Notations and assumptions}\label{sec_not}
	\newcounter{not_list}
	In this section we introduce frequently used notations and assumptions.
	Let us start with a few general conventions.
	\begin{enumerate}[(i)]
		\item The letter $C>0$ always denotes a generic constant which may change in-between lines.
		\item The branch of the square root is fixed by $\text{Im}\, \sqrt{w} > 0$ for $w \in \C \setminus [0, \infty)$.
		\item \label{it_tanc} For $w\in \C \setminus \{k\pi + \tfrac{\pi}{2} : k \in \Z \}$ we define the function 	
		\begin{equation*}
			\textup{tanc}(w) :=
			\begin{cases}
				\tfrac{\tan(w)}{w}, &  w \in \C \setminus \big(\{0\} \cup \{ k\pi + \tfrac{\pi}{2} : k \in \Z \}\big),\\
				1, & w = 0.
			\end{cases}
		\end{equation*}
		For $ x \in \R \setminus \{ 0 \}$ the equation $\textup{tanc}(i x) = \frac{\tanh(x)}{x}$ is valid, which can be extended by continuity to $x=0$.
		\item The symbol $| \cdot|$ is used for the absolute value, the Euclidean vector norm, or the Frobenius norm of a number, vector, or matrix, respectively.
		We write $\langle \cdot, \cdot \rangle$ for the Euclidean scalar product  in $\C^n$, $n \in \N$, which is anti-linear in the second argument. 
		\item  Let $\mathcal{H}$ and $\mathcal{G}$ be Hilbert spaces and  $A$ be a linear operator from $\mathcal{H}$ to $\mathcal{G}$. The domain, kernel, and range of $A$ are denoted by $\dom A$, $\ker A$, and $\ran A$, respectively. If $A$ is bounded and everywhere defined, then we write $\| A \|_{\mathcal{H} \rightarrow \mathcal{G}}$ for its operator norm.  The expression $[\cdot,\cdot]$ denotes the commutator of two operators. If $\mathcal{H} = \mathcal{G}$ and $A$ is a closed operator, then the resolvent set,  the spectrum, and the point spectrum of $A$ are denoted by $\rho(A)$, $\sigma(A)$, and $\sigma_{\textup{p}}(A)$, respectively. In the case that $A$ is self-adjoint, we denote the essential and discrete spectrum of $A$ by $\sigma_{\textup{ess}}(A)$ and $\sigma_{\textup{disc}}(A)$, respectively.
		\setcounter{not_list}{\value{enumi}}
	\end{enumerate}
	Next, we fix the space dimension, introduce Dirac matrices and define related notations.
	\begin{enumerate}[(i)]
		\setcounter{enumi}{\value{not_list}}
		\item\label{it_Dirac_matrices} By $\theta \in \{ 2, 3 \}$ we denote the space dimension and we set $N=2$ for $\theta=2$ and $N=4$ for $\theta = 3$. 
		With the help of the  Pauli spin matrices
		\begin{equation*}
			\sigma_1 = \begin{pmatrix} 0 & 1\\ 1 & 0 \end{pmatrix}, \quad\sigma_2 = \begin{pmatrix} 0 & -i\\ i & 0 \end{pmatrix}, \quad \text{and} \quad \sigma_3 =\begin{pmatrix} 1 & 0\\ 0 & -1 \end{pmatrix}
		\end{equation*} 
		we define the Dirac matrices $\alpha_1, \dots, \alpha_\theta, \beta \in \mathbb{C}^{N \times N}$  for $\theta=2$ by
		\begin{equation} \label{def_Dirac_matrices_2d}
			\alpha_1 := \sigma_1, \quad \alpha_2 := \sigma_2, \quad \text{and} \quad \beta := \sigma_3,
		\end{equation}
		and  for $\theta = 3$ by
		\begin{equation} \label{def_Dirac_matrices_3d}
			\alpha_j := \begin{pmatrix} 0 &
				\sigma_j \\ \sigma_j & 0
			\end{pmatrix}\text{ for } j=1,2,3  \quad \text{and} \quad \beta := \begin{pmatrix}
				I_2 & 0 \\ 0 & -I_2
			\end{pmatrix},
		\end{equation}
		where $I_n$ is the $n \times n$-identity matrix, $n \in \mathbb{N}$. The Dirac matrices satisfy
		\begin{equation}\label{eq_dirac_matrices}
			\alpha_j \alpha_k + \alpha_k \alpha_j = 2 I_N \delta_{jk} \quad \text {and} \quad \alpha_j \beta + \beta \alpha_j=0 \quad j,k \in \{1,\dots,\theta\},
		\end{equation}
		where $\delta_{jk}$ denotes the Kronecker delta.
		Vectors in $\C^\theta$ are denoted by  $x = (x_1,\dots,x_\theta)$ and	we will often make use of the notations
		\begin{equation*}
			\alpha \cdot \nabla := \sum_{j = 1}^\theta \alpha_j \partial_j \quad \text{and} \quad \alpha \cdot x := \sum_{j = 1}^\theta \alpha_j x_j, \quad x = (x_1, \dots, x_\theta) \in \C^\theta.
		\end{equation*}
		Finally, we use the notation $x = (x',x_\theta)$ with $x' \in \C^{\theta-1}$ and $x_\theta \in \C$.
		\setcounter{not_list}{\value{enumi}}
	\end{enumerate}
	In the upcoming item we introduce a class of $C^2$-hypersurfaces which is convenient for the definition of trace and extension operators, as well as tubular neighbourhoods.
	\begin{enumerate}[(i)]
		\setcounter{enumi}{\value{not_list}}
		\item  \label{it_def_Sigma} We assume that $\Sigma$ is  the boundary of an open set $\Omega_+ \subset \R^{\theta}$ which satisfies the following: There exist open sets  $W_1, \dots ,W_p \subset \mathbb{R}^\theta$, mappings $\zeta_1,\dots,\zeta_p \in C^{2}_b(\R^{\theta-1}; \mathbb{R})$ (see \eqref{gehtdas} below), rotation matrices $\kappa_1, \dots ,\kappa_p \in \R^{\theta \times \theta}$, and $\varepsilon_0 >0$ such that
		\begin{itemize}
			\item[(i)] $\Sigma \subset \bigcup_{l=1}^p W_l$;
			\item[(ii)] if $x \in \partial \Omega_+ = \Sigma$, then there exists $l \in \{1, \dots ,p\}$ such that $B(x,\varepsilon_0) \subset W_l$;
			\item[(iii)] $W_l \cap \Omega_+
			= W_l \cap \Omega_l$, where $\Omega_l = \{ \kappa_l (x', x_\theta): x_\theta < \zeta_l(x'), \, (x',x_\theta) \in \R^\theta \}$, for $l \in \{1,\dots ,p\}$.
		\end{itemize}
		Furthermore, we set $\Sigma_l := \partial \Omega_l = \{ \kappa_l (x', \zeta_l(x')): x' \in \mathbb{R}^{\theta-1} \}$,  $\Omega_- := \R^\theta \setminus \overline{\Omega_+}$, and denote the unit normal vector field at $\Sigma$ that is pointing outwards of $\Omega_+$ by $\nu$. Moreover, for $u: \mathbb{R}^\theta \rightarrow \mathbb{C}^N$ we define $u_\pm := u|_{\Omega_\pm}$.
		\setcounter{not_list}{\value{enumi}}
	\end{enumerate}
	Now, we turn to the introduction of various function spaces.
	\begin{enumerate}[(i)]
		\setcounter{enumi}{\value{not_list}}
		\item\label{gehtdas} If $n \in \N$ and $U \subset \R^n$ is open, then  $H^r(U)$ and $W^r_\infty(U)$ denote the $L^2$ and $L^\infty$ based Sobolev spaces of order $r$, respectively; cf. \cite[Chaper~3]{M00}. Moreover, if $k \in \N \cup \{\infty\}$, then we write $C_b^k(U)$ for the space which contains all $f \in C^k(U)$ such that $f$ and all partial derivatives of $f$ up to order $k$ are bounded. Vector or matrix valued function spaces are defined in the natural way, i.e.  component-wise. Moreover, for function spaces with $\C^{n \times l}$-valued functions, $n,l \in \N$, we use the notations $H^r(U;\C^{n\times l})$, $W^r_\infty(U;\C^{n\times l})$, and $C_b^k(U;\C^{n \times l}).$

		\item The trace spaces $H^r(\Sigma)$, $r \in [-2,2]$,  $W^1_\infty(\Sigma)$, and $C^k_b(\Sigma)$, $k \in \{1,2\}$, are defined via local coordinates, a partition of unity and the corresponding function spaces on open sets; cf. \cite[Section~2]{BHS23}.
		Moreover, the well-defined and bounded Dirichlet trace operator is denoted by
		\begin{equation*}
			\tr^{\pm}:  H^r(\Omega_\pm) \to H^{r-1/2}(\Sigma)  \quad \textup{and} \quad  \tr : H^{r}(\R^{\theta}) \to H^{r-1/2}(\Sigma),
		\end{equation*}
		$r \in (\tfrac{1}{2},\tfrac{5}{2})$; cf. \cite[Theorem~2]{M87}, where we use for 
		$$u=u_+\oplus u_- \in H^r(\R^{\theta} \setminus \Sigma) = H^r(\Omega_+;\C^N) \oplus H^r(\Omega_-;\C^N)$$ 
		the shortened notation $\tr^\pm u$ for $\tr^\pm u_\pm$. Vector or matrix valued trace spaces and trace operators are defined  component-wise.

		\item\label{it_Bochner} For a Hilbert space $\mathcal{H}$ the usual $L^2((-1,1))$ based Bochner Lebesgue space of $\mathcal{H}$-valued
		functions is denoted by $L^2((-1,1);\mathcal{H})$; cf. \cite[Section~2.2]{BHS23}. In the case $\mathcal{H}=H^r(S;\C^N)$ with $S \in \{\Sigma,\R^{\theta-1}\}$, we write $\mathcal{B}^r(S)$  instead of $L^2((-1,1);H^r(S;\C^N))$, respectively. We also write
		$\norm{\cdot}_r$ for the norm in $\mathcal{B}^r(S)$. In a similar way, we define
		\begin{equation*}\begin{aligned} 
				\norm{\cdot}_{r \to r'} &:=\norm{\cdot}_{\mathcal{B}^r(S) \to \mathcal{B}^{r'}(S)},\\
				\norm{\cdot}_{r \to \mathcal{H}} &:= \norm{\cdot}_{\mathcal{B}^r(S) \to \mathcal{H}},\\
				\norm{\cdot}_{\mathcal{H} \to r'} &:= \norm{\cdot}_{\mathcal{H} \to \mathcal{B}^{r'}(S) }.
			\end{aligned}
		\end{equation*}
		We  will  use  the bounded embedding
		\begin{equation*}
			\mathfrak{J}:H^r(S;\C^N) \to \mathcal{B}^{r}(S),\quad
			\mathfrak{J}\varphi(t) := \varphi,
		\end{equation*}
		and its adjoint
		\begin{equation*}
			\mathfrak{J}^* :\mathcal{B}^{r}(S) \to H^r(S;\C^N),\quad
			\mathfrak{J}^*f = \int_{-1}^1 f(t) \,dt.
		\end{equation*} 
		
		We also use the following convenient identification: Let
		$\mathcal{Q} \in L^\infty((-1,1))$ and  $\mathcal{A}$ be bounded operator in $H^r(S;\C^N)$, $r \in [-2,2]$. Then, we identify 
		\begin{equation*}
			\mathcal{M}_{\mathcal{Q}}: \mathcal{B}^r(S)\to \mathcal{B}^r(S),\quad \!\!
			(\mathcal{M}_{\mathcal{Q}}f)(t) := \mathcal{Q}(t) f(t), 
		\end{equation*}
		and 
		\begin{equation*}
			\mathcal{M}_{\mathcal{A}} : \mathcal{B}^r(S)\to \mathcal{B}^r(S) ,\quad
			(\mathcal{M}_{\mathcal{A}}f)(t) :=  \mathcal{A}(f(t)), 
		\end{equation*}
		with $\mathcal{Q}$ and $\mathcal{A}$, respectively. Note  that the norms  $\norm{\mathcal{M}_{\mathcal{Q}}}_{r \to r}$ and $\norm{\mathcal{M}_{\mathcal{A}}}_{r \to r}$ are equal  to $\norm{\mathcal{Q}}_{L^\infty((-1,1))}$ and $\norm{\mathcal{A}}_{H^r(S;\C^N)\to H^r(S;\C^N)}$, respectively. 
		\setcounter{not_list}{\value{enumi}}
	\end{enumerate}
	Finally, we fix notations and state simple properties regarding Fourier transforms.
	\begin{enumerate}[(i)]
		\setcounter{enumi}{\value{not_list}}
		\item \label{it_Fourier}  The expression $\mathcal{F}$ denotes the Fourier transform in $\R^{\theta-1}$. Moreover, $\mathcal{F}_1$ and $\mathcal{F}_2$ denote the partial Fourier transforms in $\R^{\theta}$  with respect to the first $\theta-1$ variables and the $\theta$-th variable, respectively. These transforms are given for $\psi \in \mathcal{S}(\R^{\theta-1})$ and $u \in  \mathcal{S}(\R^\theta)$ by
		\begin{equation*}
			\begin{aligned}
				\hspace{23 pt}\mathcal{F}\psi(\xi') &= \frac{1}{\sqrt{(2\pi)^{\theta-1}}}\int_{\R^{\theta-1}} \psi(x') e^{-i\langle x',\xi'\rangle}dx',&& \xi' \in \R^{\theta-1},\\
				\mathcal{F}_1u(\xi) &= \frac{1}{\sqrt{(2\pi)^{\theta-1}}}\int_{\R^{\theta-1}} u(x',\xi_\theta) e^{-i\langle x',\xi'\rangle}dx',   && \xi=(\xi',\xi_{\theta}) \in \R^{\theta},\\
				\mathcal{F}_2u(\xi) &= \frac{1}{\sqrt{2\pi}}\int_{\R} u(\xi',x_\theta) e^{-i x_\theta \xi_\theta}dx_\theta, && \xi=(\xi',\xi_{\theta}) \in \R^{\theta},
			\end{aligned}
		\end{equation*}
		and can be uniquely extended 
		to bounded operators in $\mathcal{S}'(\R^{\theta-1})$  and $\mathcal{S}'(\R^{\theta})$, where $\mathcal{S}'$ denotes the space of tempered distributions; cf. \cite[Chapter IX]{RS75}. Moreover, the application of the Fourier transform to vector and matrix-valued functions or distributions  is defined component-wise. The (usual) Fourier transform in $\R^\theta$ with respect to all variables is given by $\mathcal{F}_{1,2} := \mathcal{F}_1 \mathcal{F}_2 = \mathcal{F}_2 \mathcal{F}_1$.
	\end{enumerate}

\subsection*{Acknowledgements}  
This research was funded by the Austrian Science Fund (FWF) 10.55776/P 33568-N. For the purpose of open access, the author has applied a CC BY public copyright licence to any Author Accepted Manuscript version arising from this submission.

	\section{Explicit conditions for the approximation of Dirac operators with $\delta$-shell potentials in the norm resolvent sense}\label{sec_expl_cond}

	 In the first main theorem of this paper we replace the nonexplicit smallness condition from Theorem~\ref{THEO_MAIN} by the simple and explicit condition \eqref{eq_d_cond} below. This condition will also turn out to be optimal later.
	
		\begin{theo}\label{theo_main_1}
		Let $q\in L^\infty((-1,1);[0,\infty))$ with $\int_{-1}^1 q(s) \,ds =1$ and assume that
		$\eta, \tau \in  C^1_b(\Sigma;\R)$ satisfy the condition 
		\begin{equation}\label{eq_d_cond}
		\sup_{x_\Sigma \in \Sigma} d(x_\Sigma) < \frac{\pi^2}{4}, \qquad d = \eta^2 - \tau^2.
	\end{equation}
		 Let $V$ and $V_\varepsilon$ be  as in \eqref{eq_V} and  \eqref{eq_V_eps}, and define $\widetilde V$ by \eqref{eq_scaling}.
			Then, for all $z \in \C\setminus\R$  and $r \in (0,\tfrac{1}{2} )$ there exist $C>0$ and $\varepsilon' \in (0,\varepsilon_1)$ such that 
		\begin{equation*} 
			\norm{(H_{V_\varepsilon}-z)^{-1} - (H_{\widetilde{V}\delta_\Sigma}-z)^{-1}}_{L^2(\R^\theta;\C^N) \to L^2(\R^\theta;\C^N)} \leq C \varepsilon^{1/2-r}, \quad  \varepsilon \in (0, \varepsilon').
		\end{equation*} 
		In particular, $H_{V_\varepsilon}$ converges to $H_{\widetilde{V}\delta_\Sigma}$ in the norm resolvent sense as $\varepsilon \to 0$. 
	\end{theo}

	The strategy to prove Theorem~\ref{theo_main_1} is to reduce it to the case where $\Sigma$ is a rotated $C^2_b$-graph, which is shown as a separate result in Section~\ref{sec_C^2_b_graphs}. The actual proof of Theorem~\ref{theo_main_1} makes use of three preparatory results, which relate the resolvents of $H_{\widetilde{V}\delta_\Sigma}$ and $H_{V_\varepsilon}$ to similar operators with potentials supported on the $C^2_b$-graphs $\Sigma_l$ and tubular neighbourhoods of $\Sigma_l$ (see \eqref{it_def_Sigma} in Section~\ref{sec_not}), respectively, via a suitable partition of unity.
	We also remark that the condition \eqref{eq_d_cond} implies 
	\begin{equation*}
		\sup_{x_\Sigma \in \Sigma} \widetilde d(x_\Sigma)= \sup_{x_\Sigma \in \Sigma} (\widetilde\eta^{\,2}(x_\Sigma) - \widetilde\tau^{\,2}(x_\Sigma))= \sup_{x_\Sigma \in \Sigma}4 \tan^2\biggl(\frac{\sqrt{d(x_\Sigma)}}{2}\biggr) < 4,
	\end{equation*} 
	so that $\widetilde V$ in \eqref{eq_scaling}
	is noncritical; cf. \eqref{eq_d_tilde_non_crit}.

    The first preparatory lemma  is a slight variation of \cite[Lemma~B.2]{BHS23} and allows us to construct a suitable partition of unity for $\Sigma$. 

	\begin{lemma}\label{lem_part_unity_phi_l} 
		Let $\Sigma \subset \mathbb{R}^\theta$, $\theta  \in \{ 2,3  \}$, be a set as described in Section~\ref{sec_not}~\eqref{it_def_Sigma}. Then, there exists a partition of unity $\varphi_1,\dots,\varphi_p \in C^\infty_b(\R^{\theta};\R)$  for $\Sigma$ subordinate to the open cover $W_1,\dots,W_p$ of $\Sigma$. This partition of unity can be chosen such that $\varphi_1,\dots,\varphi_{p}$ is also a partition of unity for $\Sigma + B(0,\delta)$ for some $\delta>0$ and $\supp \varphi_l + B(0,\delta) \subset W_l$ for all $l \in \{1,\dots,p\}$.
	\end{lemma}
	\begin{proof}
	According to \cite[Appendix A, Lemmas~1.2 and~1.3]{S92} there exists a sequence $(x_n)_{n \in \N} \subset \mathbb{R}^\theta$, $M \in \N$, $0<\delta < \frac{\varepsilon_0}{4}$, with $\varepsilon_0$ from Section~\ref{sec_not}~\eqref{it_def_Sigma}, and a sequence of real-valued $C^\infty$-functions $(\phi_n)_{n \in \N}$  such that  $(B(x_n,\delta))_{n \in \N}$ is an open cover of $\R^\theta$, $(\phi_n)_{n \in \N}$  is a partition of unity for $\mathbb{R}^\theta$, $\supp \phi_n \subset B(x_n,\delta)$ for all $n \in \N$, every point $x \in \R^\theta$ is contained in at most $M$ of the sets $B(x_n,\delta)$, and the derivatives (of any order) of the functions $\phi_n$ are uniformly bounded.  Next, we define the set 
	$$Y := \{ x_n: B(x_n,2\delta) \cap \Sigma \neq \emptyset \}.$$ 
	By construction, for all $x_n \notin Y$ one has $B(x_n,\delta) \cap (\Sigma + B(0,\delta)) = \emptyset$.
	Moreover, note that for all $x_n \in Y$ there exists an $l \in \{1,\dots, p\}$ such that $B(x_n,2\delta) \subset W_l $. In fact, as $B(x_n,2\delta) \cap \Sigma \neq \emptyset$, there exists  a $y_\Sigma \in B(x_n,2\delta) \cap \Sigma$ and thus item~\eqref{it_def_Sigma} in Section~\ref{sec_not} implies $B(y_\Sigma,\varepsilon_0) \subset W_l$ for an $l \in \{1,\dots,p\}$. Hence, for any $ y \in B(x_n,2\delta)$ one has
	\begin{equation*}
		\abs{y-y_\Sigma} \leq \abs{y -x_n} + \abs{x_n - y_\Sigma} < 4\delta < \varepsilon_0, 
	\end{equation*}
	which shows $B(x_n,2\delta) \subset W_l$. Define the sets 
	$$I_1 := \{ n: x_n \in Y, B(x_n, 2\delta) \subset W_1 \}$$ 
	and for $l \in \{ 2, \dots, p \}$
	$$I_l := \{ n: x_n \in Y, B(x_n, 2\delta) \subset W_l, B(x_n,2\delta) \not\subset  W_k, k \in \{1,\dots, l-1\} \}.$$ 
	Then, it is not difficult to see that
	\begin{equation*}
		\varphi_l = 
		\sum_{n \in I_l}\phi_n 
	\end{equation*}
	is a partition of unity having the claimed properties. Moreover, the construction of $\varphi_l$, $l \in \{1,\dots,p\}$, also implies $\supp \varphi_l + B(0,\delta) \subset W_l$.
	\end{proof}

In the following we use a partition of unity as in Lemma~\ref{lem_part_unity_phi_l} to connect Dirac operators with strongly localized potentials 
supported in $\Omega_\varepsilon$ with Dirac operators with strongly localized potentials supported in
the tubular neighbourhoods $\Omega_{\varepsilon,l}$ of the rotated $C^2_b$-graphs
\begin{equation*}
	\Sigma_l =  \{\kappa (x',\zeta_l(x')) : x'\in\R^{\theta-1} \},\qquad l \in \{1,\dots,p\};
\end{equation*}
cf. Section~\ref{sec_not}~\eqref{it_def_Sigma}. Note that  $V$ is only defined on $\Sigma$ and hence $V$ is a priori only defined on a subset of $\Sigma_l$. Thus, to be able to define a strongly localized potential $V_{\varepsilon,l}$ in $\Omega_{\varepsilon,l}$ in the same way as $V_\varepsilon$ in \eqref{eq_V_eps}, we first construct suitable extensions $V_l$ of $V$ to $\Sigma_l$. To do so, we choose  $\rho_\nu \in C^1(\R;\R)$ with $0\leq \rho_\nu \leq 1 $, $\rho_\nu(0) = 1$ and compact support in $(-\varepsilon_1, \varepsilon_1)$, where $\varepsilon_1$ is the number specified below \eqref{eq_iota}. Since $\Sigma$ is assumed to satisfy \eqref{it_def_Sigma} in Section~\ref{sec_not}, it is not difficult to show that for $ \omega \in \{V,\eta,\tau\}$ the function
\begin{equation}\label{eq_interaction_extensions}
	\omega_{\rm ext}(x) =
	\begin{cases}
		\omega(x_\Sigma)\rho_\nu(t), & \text{if } x = x_\Sigma + t \nu(x_\Sigma) \in \Omega_{\varepsilon_1} ,\\
		0,& \text{if } x \notin \Omega_{\varepsilon_1},
	\end{cases}
\end{equation}
is a  $C_b^1$-extension of $\omega$ to $\R^\theta$ which is supported in $\Omega_{\varepsilon_1}$. 
We then define the functions 
\begin{equation} \label{def_V_l}
  V_l:=V_{\rm ext}|_{\Sigma_l} \in C^1_b(\Sigma;\C^{N \times N}),\,\,\, \eta_l :=  \eta_{\rm ext} |_{\Sigma_l},\,\,\, \tau_l :=  \tau_{\rm ext}|_{\Sigma_l} \in C^1_b(\Sigma_l;\R),
\end{equation}
and note that $V_l=\eta_l I_N+\tau_l\beta$ for $l \in \{1,\dots,p\}$; cf. \eqref{eq_V}. Clearly, 
\begin{equation*}
	V_l|_{\Sigma_l \cap \Sigma} = V|_{\Sigma_l \cap \Sigma}.
\end{equation*}
Moreover, we mention  that  $d_l = \eta_l^2 - \tau_l^2$ satisfies by construction
\begin{equation}\label{eq_cond_d_l}
	\sup_{x_{\Sigma_l} \in \Sigma_l} d_l(x_{\Sigma_l}) < \frac{\pi^2}{4}, \qquad  l \in \{1,\dots,p\},
\end{equation}
as $d = \eta^2-\tau^2$ satisfies \eqref{eq_d_cond}.
With this extension at hand, we define $V_{\varepsilon,l}$ as $V_\varepsilon$ in \eqref{eq_V_eps} (with $\Sigma$ and $V$ replaced by $\Sigma_l$ and $V_l$)
and in the same way as $H_{V_\varepsilon} = H + V_\varepsilon$ in \eqref{eq_H_V_eps} we define
$$
H_{V_{\varepsilon,l}} = H+ V_{\varepsilon,l}, \quad  \dom H_{V_{\varepsilon,l}} :=  H^1(\R^\theta;\C^N),\qquad  l \in \{1,\dots,p\},
$$
where $H$ is the free Dirac operator (see \eqref{eq_def_free_Dirac}). Note that the operators $H_{V_{\varepsilon,l}}$ are self-adjoint in $L^2(\R^\theta;\C^N)$.

In the next lemma we express the resolvent of $H_{V_\varepsilon}$ in terms of the resolvents of $H_{V_{\varepsilon,l}}$. Here we fix  
the partition of unity $\varphi_1,\dots,\varphi_{p} \in C^\infty_b(\R^{\theta};\R)$ from Lemma~\ref{lem_part_unity_phi_l} and  we set $\varphi_{p+1} := 1-\sum_{l=1}^p \varphi_l$.

	\begin{lemma}\label{lem_H_V_eps_formula}
		Let the functions $V_\varepsilon, V_{\varepsilon,l}$ and the self-adjoint operators $H_{V_\varepsilon} = H + V_\varepsilon$, $H_{V_{\varepsilon,l}} = H+ V_{\varepsilon,l}$ for $l \in \{1,\dots,p\}$
		be as above, and set $H_{V_{\varepsilon,p+1}} := H$. Then, for $z \in \C$ such that $\abs{\textup{Im } z} > \sum_{l=1}^{p+1} \|\alpha \cdot\nabla \varphi_l\|_{L^\infty(\R^{\theta};\C^{N \times N})}$
		the operator  $$I +\sum_{l=1}^{p+1}  i(H_{V_{\varepsilon,l}}-z)^{-1} (\alpha \cdot \nabla \varphi_l) $$ is continuously invertible in $L^2(\R^{\theta};\C^N)$ and the resolvent formula
		\begin{equation*}
			(H_{V_\varepsilon}-z)^{-1} = \Bigg(I +\sum_{l=1}^{p+1}  i(H_{V_{\varepsilon,l}}-z)^{-1}(\alpha \cdot \nabla \varphi_l)  \Bigg)^{-1}\Bigg(\sum_{l=1}^{p+1} (H_{V_{\varepsilon,l}}-z)^{-1}\varphi_l \Bigg)
		\end{equation*}
		is valid for all $\varepsilon \in (0,\min\{\varepsilon_1,\delta\})$, where $\varepsilon_1>0$ and $\delta>0$ are chosen as below \eqref{eq_iota} and  as in Lemma~\ref{lem_part_unity_phi_l}, respectively.
	\end{lemma}
	
	\begin{proof}
	Our first goal is to verify that for $\varphi_l \in C^\infty_b(\R^{\theta};\R)$ from the fixed partition of unity the identity 
	\begin{equation}\label{jadochbitte}
	 H_{V_\varepsilon} \varphi_l u=H_{V_{\varepsilon,l}} \varphi_l u,\qquad l \in \{1,\dots,p\},
	\end{equation}
holds for $u \in \dom H_{V_\varepsilon} = \dom H_{V_{\varepsilon,l}} = H^1(\R^{\theta};\C^N)$ and $\varepsilon \in (0,\min\{\delta,\varepsilon_1\})$. For this it suffices to show 
		 \begin{equation}\label{eq_V_eps_phi_l}
		 	V_\varepsilon(x)\varphi_l(x) = V_{\varepsilon,l}(x) \varphi_l(x), \qquad  x \in \R^{\theta}.
		 \end{equation}
	 	 In fact, it is obvious that $V_\varepsilon(x) \varphi_l(x) = V_{\varepsilon,l}(x) \varphi_l(x)$ for $x \notin \supp \varphi_l$. Next, consider $x \in \Omega_\varepsilon \cap \supp \varphi_l$. Then, there exists $(x_{\Sigma},t) \in  \Sigma\times(-\varepsilon,\varepsilon)$ such that $x = x_{\Sigma} + t \nu(x_{\Sigma})$. The inclusion $\supp \varphi_l + B(0,\delta) \subset W_l $ implies $x_\Sigma \in W_l \cap \Sigma$. Hence, by Section~\ref{sec_not}~\eqref{it_def_Sigma} $x_{\Sigma} \in \Sigma_l$ and therefore $x = x_\Sigma + \nu_l(x_\Sigma) \in \Omega_{\varepsilon,l}$, where $\nu_l$ is the unit normal vector field at $\Sigma_l$ which has the same orientation on $\Sigma_l \cap \Sigma$ as $\nu$. In turn, we have
		 \begin{equation*}
		 	V_{\varepsilon}(x) \varphi_l(x) = V(x_\Sigma) \frac{q\bigl(\tfrac{t}{\varepsilon})}{\varepsilon} \varphi_l(x) = V_l(x_\Sigma) \frac{q\bigl(\tfrac{t}{\varepsilon})}{\varepsilon} \varphi_l(x) = V_{\varepsilon,l}(x) \varphi_l(x).
		 \end{equation*}
	 	If $x \in \Omega_{\varepsilon,l} \cap \supp \varphi_l$ one shows \eqref{eq_V_eps_phi_l} in the same way. It remains to treat the case $ x \in \Omega_{\varepsilon}^c \cap \Omega_{\varepsilon,l}^c \cap \supp \varphi_l$. However, then both sides of \eqref{eq_V_eps_phi_l} are zero. Therefore, \eqref{eq_V_eps_phi_l}
	 	and hence  \eqref{jadochbitte} are true.

	Applying the product rule and using \eqref{jadochbitte} yields for $u \in \dom H_{V_\varepsilon} = \dom H_{V_{\varepsilon,l}}$, $l \in \{1,\dots,p\}$, and $\varepsilon \in (0,\min\{\delta,\varepsilon_1\})$
			\begin{equation}\label{eq_H_V_eps_2_H_V_eps_l}
			\varphi_l H_{V_\varepsilon}u = H_{V_\varepsilon} \varphi_l u+ i(\alpha \cdot\nabla \varphi_l) u = H_{V_{\varepsilon,l}}\varphi_l  u + i(\alpha \cdot\nabla \varphi_l)u.
		\end{equation}
		As $\varphi_1,\dots,\varphi_p$ also form a partition of unity for $\Sigma + B(0,\delta) \supset \Omega_\varepsilon$, $\varepsilon \in (0,\delta)$, we have $V_\varepsilon \varphi_{p+1} =0$ for all $\varepsilon \in (0,\min\{\varepsilon_1,\delta\})$. Hence, the product rule and the convention  $H = H_{V_{\varepsilon,p+1}}$ show that \eqref{eq_H_V_eps_2_H_V_eps_l} remains valid for $l = p+1$ and $\varepsilon \in (0,\min\{\varepsilon_1,\delta\})$. 
		For $z \in \C \setminus \R$ and $u \in \dom H_{V_\varepsilon} = H^1(\R^{\theta};\C^N)$ it follows from \eqref{eq_H_V_eps_2_H_V_eps_l} that
			\begin{equation}\label{eq_res_formular_V_eps_0}
			\begin{aligned}
				\Bigg(\sum_{l=1}^{p+1} &(H_{V_{\varepsilon,l}}-z)^{-1}\varphi_l\Bigg)(H_{V_\varepsilon}-z) u \\
				&= \sum_{l=1}^{p+1} \Big((H_{V_{\varepsilon,l}}-z)^{-1}(H_{V_{\varepsilon,l}}-z) \varphi_lu + i(H_{V_{\varepsilon,l}}-z)^{-1}(\alpha \cdot \nabla \varphi_l)u \Big)\\
				&=\sum_{l=1}^{p+1} \Big(\varphi_l u + i(H_{V_{\varepsilon,l}}-z)^{-1}(\alpha \cdot \nabla \varphi_l) u \Big)\\
				&= \left(I +\sum_{l=1}^{p+1}  i(H_{V_{\varepsilon,l}}-z)^{-1}(\alpha \cdot \nabla \varphi_l) \right) u.
			\end{aligned}
		\end{equation}
		In particular, if $|\textup{Im }z| > \sum_{l=1}^{p+1} \|\alpha \cdot\nabla \varphi_l\|_{L^\infty(\R^{\theta};\C^{N \times N})}$ holds, then it is clear that the operator $I +\sum_{l=1}^{p+1}  i(H_{V_{\varepsilon,l}}-z)^{-1} (\alpha \cdot \nabla \varphi_l) $ is continuously invertible in $L^2(\R^{\theta};\C^N)$ and therefore
		\begin{equation*}
			(H_{V_\varepsilon}-z)^{-1} = \Bigg(I +\sum_{l=1}^{p+1}  i(H_{V_{\varepsilon,l}}-z)^{-1}(\alpha \cdot \nabla \varphi_l)  \Bigg)^{-1}\Bigg(\sum_{l=1}^{p+1} (H_{V_{\varepsilon,l}}-z)^{-1}\varphi_l \Bigg).
		\end{equation*}
	\end{proof}

	Next, we provide an analogous formula as in the previous lemma for the resolvents of Dirac operators with $\delta$-shell potentials. For this, recall first from 
	\eqref{eq_V} that 
	$V = \eta I_N + \tau \beta$ with $\eta,\tau \in C^1_b(\Sigma;\R)$, assume that $d = \eta^2-\tau^2$ satisfies \eqref{eq_d_cond}, and let $\widetilde{V} = \textup{tanc}\bigl(\tfrac{\sqrt{d}}{2}\bigr)V$
	be as in \eqref{eq_scaling}. Furthermore, with $V_l= \eta_l I_N + \tau_l \beta$ and $ \eta_l,\tau_l \in C^1_b(\Sigma;\R)$, $l \in \{1,\dots,p\}$, as described below \eqref{eq_interaction_extensions} we set (as in \eqref{eq_scaling})
	\begin{equation*}
		\widetilde{V}_l = \widetilde{\eta}_l I_N + \widetilde{\tau}_l \beta, \qquad (\widetilde{\eta}_l,\widetilde{\tau}_l) = \textup{tanc}\bigl(\tfrac{\sqrt{d_l}}{2}\bigr) (\eta_l,\tau_l), \qquad d_l = \eta^2_l -\tau^2_l.
	\end{equation*} 
	In the third preparatory lemma we shall again make use of 
	the partition of unity $\varphi_1,\dots,\varphi_{p} \in C^\infty_b(\R^{\theta};\R)$ from Lemma~\ref{lem_part_unity_phi_l} and as before we set 
	$\varphi_{p+1} = 1-\sum_{l=1}^p \varphi_l$.

	\begin{lemma}\label{lem_H_Vtilde}
	Let the functions $\widetilde{V}$ and $\widetilde V_l$ be as above, let the self-adjoint operator $H_{\widetilde{V} \delta_\Sigma}$ be defined as in \eqref{eq_def_H_V_tilde} and define
	in the same way (with $\widetilde{V}$ and $\Sigma$ replaced by $\widetilde{V}_l$ and $\Sigma_l$) the self-adjoint operators 
	$$H_{\widetilde{V}_l \delta_{\Sigma_l}},\,\,\, l \in \{1,\dots,p\},\qquad\text{and}\qquad H_{\widetilde{V}_{p+1} \delta_{\Sigma_{p+1}}} :=H.$$ 
	Then, for $z \in \C$ such that $\abs{\textup{Im } z} > \sum_{l=1}^{p+1} \|\alpha \cdot\nabla \varphi_l\|_{L^\infty(\R^{\theta};\C^{N \times N})}$ the operator  $$I +\sum_{l=1}^{p+1}  i(H_{\widetilde{V}_{l} \delta_{\Sigma_l}}-z)^{-1} (\alpha \cdot \nabla \varphi_l) $$ is continuously invertible in $L^2(\R^{\theta};\C^N)$ and the resolvent formula 
	\begin{equation*}
		(H_{\widetilde{V}\delta_\Sigma}-z)^{-1} = \Bigg(I +\sum_{l=1}^{p+1}  i(H_{\widetilde{V}_l \delta_{\Sigma_l}}-z)^{-1}(\alpha \cdot \nabla \varphi_l)  \Bigg)^{-1}\Bigg(\sum_{l=1}^{p+1} (H_{\widetilde{V}_l \delta_{\Sigma_l}}-z)^{-1}\varphi_l \Bigg)
	\end{equation*}
	is valid.
	\end{lemma}
	
	\begin{proof}
		Similar as in the proof of the previous lemma we start by showing for $u \in \dom H_{\widetilde{V} \delta_\Sigma}$ and $l \in \{1,\dots,p\}$ that
		\begin{equation}\label{zuerstdas}
		 H_{\widetilde{V}\delta_\Sigma}\varphi_l u = H_{\widetilde{V}_l \delta_{\Sigma_l}} \varphi_l u.
		\end{equation}
		First we argue that for $u \in \dom H_{\widetilde{V} \delta_\Sigma}$ we 
		have 
		\begin{equation}\label{bittedas}
		\varphi_l u \in \dom H_{\widetilde{V} \delta_\Sigma} \cap \dom H_{\widetilde{V}_l \delta_{\Sigma_l}}.\end{equation}
		In fact, since 
		$ \dom H_{\widetilde{V} \delta_\Sigma}\subset H^1(\R^\theta \setminus \Sigma;\C^N)$ and $\varphi_l \in C^\infty_b(\R^\theta;\R)$, we have $\varphi_l u \in H^1(\R^{\theta}\setminus \Sigma;\C^N)$ for $u \in \dom H_{\widetilde{V} \delta_\Sigma}$ as well as  
		\begin{equation*}
			\begin{aligned}
			&\Bigg(i(\alpha\cdot\nu)(\boldsymbol{t}_{\Sigma}^+ - \boldsymbol{t}_{\Sigma}^-) + \frac{\widetilde{V}}{2}(\boldsymbol{t}_{\Sigma}^+ + \boldsymbol{t}_{\Sigma}^-)\Bigg) \varphi_l u \\
			&\hspace{100 pt}= (\varphi_l|_\Sigma)\Big(i(\alpha\cdot\nu)(\boldsymbol{t}_{\Sigma}^+ - \boldsymbol{t}_{\Sigma}^-) + \frac{\widetilde{V}}{2}(\boldsymbol{t}_{\Sigma}^+ + \boldsymbol{t}_{\Sigma}^-)\Big)  u = 0,
			\end{aligned}
		\end{equation*}
		that is, $\varphi_l u \in \dom H_{\widetilde{V} \delta_\Sigma}$. For \eqref{bittedas} it remains to verify $\varphi_l u\in \dom H_{\widetilde{V}_l \delta_{\Sigma_l}}$. Here it is convenient to introduce $(\Omega_l)_+ := \Omega_l$ and $(\Omega_l)_- := \R^\theta \setminus \overline{\Omega_l}$. Since we already have $ \varphi_l u \in \dom H_{\widetilde{V}\delta_\Sigma} \subset H^1(\R^{\theta}\setminus \Sigma)$ it is clear that $( \varphi_l u)|_{\Omega_\pm} \in H^1(\Omega_\pm;\C^N)$. Furthermore, as $\supp \varphi_l \subset W_l $ and $\Omega_\pm \cap W_l = (\Omega_l)_\pm \cap W_l$, see Section~\ref{sec_not}~\eqref{it_def_Sigma}, it follows that $(\varphi_l u)|_{(\Omega_l)_\pm} \in  H^1((\Omega_l)_{\pm};\C^N)$.   Thus, we can apply the trace operator to $(\varphi_l u)|_{(\Omega_l)_\pm}$ and obtain 
		\begin{equation*}
			\begin{aligned}
				\Bigg(i&(\alpha\cdot\nu)(\boldsymbol{t}_{\Sigma_l}^+ - \boldsymbol{t}_{\Sigma_l}^-) + \frac{\widetilde{V}_l}{2}(\boldsymbol{t}_{\Sigma_l}^+ + \boldsymbol{t}_{\Sigma_l}^-)\Bigg) \varphi_l u  \\
				&= \begin{cases} \Big(i(\alpha\cdot\nu)(\boldsymbol{t}_{\Sigma_l}^+ - \boldsymbol{t}_{\Sigma_l}^-) + \frac{\widetilde{V}_l}{2}(\boldsymbol{t}_{\Sigma_l}^+ + \boldsymbol{t}_{\Sigma_l}^-)\Big)  \varphi_l u  & \textup{ on } \Sigma_l \cap W_l \\ 0 & \textup{ on } \Sigma_l \setminus W_l
				\end{cases}\\
				&= \begin{cases} \Big(i(\alpha\cdot\nu)(\tr^+ - \tr^-) + \frac{\widetilde{V}}{2}(\tr^+ + \tr^-)\Big)  \varphi_l u & \textup{ on } \Sigma \cap W_l \\ 0 & \textup{ on } \Sigma \setminus W_l
				\end{cases}\\
				&=0,
			\end{aligned}
		\end{equation*} 
		where we used $ \varphi_l u \in \dom H_{\widetilde{V}\delta_\Sigma}$, $\supp \varphi_l \subset W_l$, and $\widetilde{V}_l = \widetilde{V}$ on $\Sigma_l \cap W_l = \Sigma \cap W_l$. Hence, $\varphi_l u \in H^1(\R^{\theta}\setminus \Sigma_l)$ and $\varphi_l u$ fulfils the boundary condition
		\begin{equation*}
			\Bigg(i(\alpha\cdot\nu)(\boldsymbol{t}_{\Sigma_l}^+ - \boldsymbol{t}_{\Sigma_l}^-) + \frac{\widetilde{V}_l}{2}(\boldsymbol{t}_{\Sigma_l}^+ + \boldsymbol{t}_{\Sigma_l}^-)\Bigg) \varphi_l u  = 0,
		\end{equation*}
		that is, $\varphi_l u  \in \dom H_{\widetilde{V}_l \delta_{\Sigma_l}}$ and hence we have 
		\eqref{bittedas}. Moreover, 
		\begin{equation*}
			\begin{aligned}
				H_{\widetilde{V}\delta_\Sigma} \varphi_l u &= \begin{cases} (-i(\alpha \cdot\nabla) + m \beta)(\varphi_l u )|_{\Omega_\pm \cap W_l} & \textup{in } \Omega_\pm \cap W_l\\
					0 & \textup{else}
				\end{cases}\\
				&= \begin{cases} (-i(\alpha \cdot\nabla) + m \beta)(\varphi_l u)|_{(\Omega_l)_\pm \cap W_l} & \textup{in } (\Omega_l)_\pm \cap W_l\\
					0 & \textup{else}\\
				\end{cases}\\
				&= H_{\widetilde{V}_l  \delta_{\Sigma_l}} \varphi_l u,
			\end{aligned}
		\end{equation*} 
		and hence also \eqref{zuerstdas} is valid.
		
	Applying the product rule and using \eqref{zuerstdas} yields for $u \in \dom H_{\widetilde{V}\delta_\Sigma }$ and any $l \in \{1,\dots,p\}$
		\begin{equation}\label{eq_res_formula_V_tilde_0}
			\varphi_l H_{\widetilde{V}\delta_\Sigma}u  =H_{\widetilde{V}\delta_\Sigma} \varphi_l u +  i 	(\alpha \cdot \nabla \varphi_l) u = H_{\widetilde{V}_l \delta_{\Sigma_l}} \varphi_l u  + i (\alpha \cdot \nabla \varphi_l) u.
		\end{equation}
		Since $\varphi_{p+1} = 0$ on $\Sigma + B(0,\delta)$, we get for $u \in \dom H_{\widetilde{V} \delta_\Sigma} \subset H^1(\R^{\theta}\setminus \Sigma;\C^N)$ that 
		 $\varphi_{p+1}  u \in H^1(\R^{\theta};\C^N)$ and 
		 $$H_{\widetilde{V} \delta_\Sigma} \varphi_{p+1} u= H \varphi_{p+1} u =H_{\widetilde{V}_{p+1} \delta_{\Sigma_{p+1}}} \varphi_{p+1} u.$$  
		 Thus,  the product rule implies that \eqref{eq_res_formula_V_tilde_0} is also valid for $ l= p+1$.  Next, let us mention that  since $d = \eta^2-\tau^2$ and $d_l = \eta_l^2 - \tau_l^2$, $l \in \{1,\dots,p\}$,  satisfy \eqref{eq_d_cond} and \eqref{eq_cond_d_l},  respectively, $\widetilde{d}  = 4\tan^2\bigl(\tfrac{\sqrt{d}}{2}\bigr)$ and $\widetilde d_l=4\tan^2\bigl(\tfrac{\sqrt{d_l}}{2}\bigr) $, $l \in \{1,\dots,p\}$, fulfil \eqref{eq_d_tilde_non_crit}. Hence, $H_{\widetilde{V}\delta_\Sigma}$ and $H_{\widetilde{V}_l \delta_{\Sigma_l}}$ are self-adjoint operators in $L^2(\R^{\theta};\C^N)$. Now, one can show in the same way as in \eqref{eq_res_formular_V_eps_0} 
		for $z \in \C \setminus \R$  and $u \in \dom H_{\widetilde{V}\delta_\Sigma}$ that
			\begin{equation*}
				\Bigg( \sum_{l=1}^{p+1} (H_{\widetilde{V}_l \delta_{\Sigma_l}}-z)^{-1}\varphi_l \Bigg) (H_{\widetilde{V}\delta_\Sigma}-z)u = \Bigg(I +\sum_{l=1}^{p+1}  i(H_{\widetilde{V}_l \delta_{\Sigma_l}}-z)^{-1}(\alpha \cdot \nabla \varphi_l) \Bigg) u.
			\end{equation*}
			In particular, if $|\textup{Im }z| > \sum_{l=1}^{p+1} \|\alpha \cdot\nabla \varphi_l\|_{L^\infty(\R^{\theta};\C^{N \times N})}$, then it is clear that the operator $I +\sum_{l=1}^{p+1}  i(H_{\widetilde{V}_l \delta_{\Sigma_l}}-z)^{-1}(\alpha \cdot \nabla \varphi_l)$ is continuously invertible in $L^2(\R^{\theta};\C^N)$ and therefore
		\begin{equation*}
			(H_{\widetilde{V}\delta_\Sigma}-z)^{-1} = \Bigg(I +\sum_{l=1}^{p+1}  i(H_{\widetilde{V}_l \delta_{\Sigma_l}}-z)^{-1}(\alpha \cdot \nabla \varphi_l)  \Bigg)^{-1}\Bigg(\sum_{l=1}^{p+1} (H_{\widetilde{V}_l 	\delta_{\Sigma_l}}-z)^{-1}\varphi_l \Bigg).
		\end{equation*}
	\end{proof}
	
	With the help of the previous lemmas we are now in the  position to prove Theorem~\ref{theo_main_1}.
	\begin{proof}[Proof of Theorem~\ref{theo_main_1}]
	Let $V_l= \eta_l I_N + \tau_l \beta$, $l \in \{1,\dots,p\}$, be defined by~\eqref{def_V_l} and consider $z \in \C$ such that 
	$\abs{\textup{Im } z} > \sum_{l=1}^{p+1} \|\alpha \cdot\nabla \varphi_l\|_{L^\infty(\R^{\theta};\C^{N \times N})}$ and $r \in (0,\tfrac{1}{2})$.
		Since  $d_l = \eta_l^2 - \tau_l^2$ satisfies \eqref{eq_cond_d_l} we can apply Theorem~\ref{theo_rotated_graph}, which implies 
		that there exists an $\varepsilon' \in (0,\min\{\varepsilon_1,\delta,\varepsilon_{3,1},\dots,\varepsilon_{3,p}\})$, where $\varepsilon_{3,l}$, $l \in \{1,\dots,p\}$, corresponds to $\varepsilon_3$ from Theorem~\ref{theo_rotated_graph} for $\Sigma_l$, such that 
		\begin{equation}\label{eq_conv_H_V_l}
			\norm{(H_{V_{\varepsilon,l}}-z)^{-1} - (H_{\widetilde{V}_l \delta_{\Sigma_l}}-z)^{-1}}_{L^2(\R^\theta;\C^N) \to L^2(\R^\theta;\C^N)} \leq C \varepsilon^{1/2-r}
		\end{equation} 
		for $l \in \{1,\dots,p\}$ and $\varepsilon \in (0, \varepsilon')$. Our conventions from Lemma~\ref{lem_H_V_eps_formula} and  Lemma~\ref{lem_H_Vtilde} ensure that  \eqref{eq_conv_H_V_l} is also true for $l = p+1$.
		
	Now, we define  for $\varepsilon \in (0,\varepsilon')$
	\begin{equation*}
		\begin{aligned}
		\mathfrak{R}_{\varepsilon} &= \sum_{l=1}^{p+1}  (H_{V_{\varepsilon,l}}-z)^{-1} \varphi_l , &\mathfrak{S}_\varepsilon &= \sum_{l=1}^{p+1}  i(H_{V_{\varepsilon,l}}-z)^{-1}(\alpha \cdot \nabla \varphi_l),  \\
		\mathfrak{R}_{0} &= \sum_{l=1}^{p+1}  (H_{\widetilde{V}_l \delta_{\Sigma_l}}-z)^{-1} \varphi_l ,&\mathfrak{S}_0 &= \sum_{l=1}^{p+1}  i(H_{\widetilde{V}_l \delta_{\Sigma_l}}-z)^{-1}(\alpha \cdot \nabla \varphi_l).
		\end{aligned}
	\end{equation*}
 Thus, \eqref{eq_conv_H_V_l} and $\varphi_l \in C^{\infty}_b(\R^{\theta};\R)$, $l \in \{1,\dots,p+1\}$, imply 
	\begin{equation*}
	\lVert \mathfrak{S}_\varepsilon - \mathfrak{S}_0\rVert_{L^2(\R^{\theta};\C^N) \to L^2(\R^{\theta};\C^N)}, \, 	\lVert \mathfrak{R}_\varepsilon - \mathfrak{R}_0\rVert_{L^2(\R^{\theta};\C^N) \to L^2(\R^{\theta};\C^N)} \leq C \varepsilon^{1/2-r}.
	\end{equation*} 
	Furthermore,  since
	\begin{equation*} 
		\begin{aligned}
		&\max\Big\{ \sup_{\varepsilon \in (0,\varepsilon')}	\lVert \mathfrak{S}_\varepsilon \rVert_{L^2(\R^{\theta};\C^N) \to L^2(\R^{\theta};\C^N)},  	\lVert \mathfrak{S}_0\rVert_{L^2(\R^{\theta};\C^N) \to L^2(\R^{\theta};\C^N)} \Big\} \\
		&\hspace{180 pt}\leq \frac{\sum_{l=1}^{p+1} \|\alpha \cdot\nabla \varphi_l\|_{L^\infty(\R^{\theta};\C^{N \times N})}}{\abs{\textup{Im }z}} <1,
		\end{aligned}
	\end{equation*}
	we conclude with the help of Lemma~\ref{lem_H_V_eps_formula} and Lemma~\ref{lem_H_Vtilde} 
	\begin{equation*}
		\begin{aligned}
		&\lVert (H_{V_\varepsilon} -z)^{-1} - (H_{\widetilde{V}\delta_\Sigma} -z)^{-1} \rVert_{L^2(\R^{\theta};\C^N) \to L^2(\R^{\theta};\C^N)} \\
		&\hspace{50 pt}= \lVert (I+\mathfrak{S}_\varepsilon)^{-1}\mathfrak{R}_\varepsilon - (1+\mathfrak{S}_0)^{-1}\mathfrak{R}_0\rVert_{L^2(\R^{\theta};\C^N) \to L^2(\R^{\theta};\C^N)} \leq C \varepsilon^{1/2-r},
		\end{aligned}
	\end{equation*}
	which proves the claim if $|\textup{Im }z| > \sum_{l=1}^{p+1} \|\alpha \cdot \nabla \varphi_l\|_{L^\infty(\R^{\theta};\C^{N \times N})}$. 
	Using the identity 
	\begin{equation*}
		\begin{aligned}
			(H_{\widetilde{V}\delta_\Sigma} -w)^{-1} -(H_{V_\varepsilon} -w)^{-1} 
			= (I+(w-z)(H_{\widetilde{V}\delta_\Sigma} -w)^{-1})&\\
			\cdot\big((H_{\widetilde{V}\delta_\Sigma} -z)^{-1} - (H_{V_\varepsilon} -z)^{-1}\big)(I + (w-z)(H_{V_\varepsilon} -w)^{-1})&, \quad  w \in \C\setminus\R,
		\end{aligned}
	\end{equation*}
	it follows that the claim is true for all $w \in \C\setminus\R$.
	This completes the proof of Theorem~\ref{theo_main_1}.
	\end{proof}

	In order to obtain the operator $H_{\widetilde{V}\delta_\Sigma}$ as a limit operator of Dirac operators with squeezed potentials also in the case 
	\begin{equation*}   
\inf_{x_\Sigma \in \Sigma} |\widetilde{d}(x_\Sigma)| > 4
	\end{equation*}
	we use an approach that is inspired by the last paragraph of  \cite[Section~8]{CLMT21} and now  add  a strongly localized magnetic potential 
	$V_{m,\varepsilon}$ to $H_{V_\varepsilon}$ with the interaction strength $\pi$. It turns out that in this case $H_{V_\varepsilon}+V_{m,\varepsilon}$  also converges in the norm resolvent sense,  see the following Theorem~\ref{theo_main_2}. 
	However, by the specific choice of $\pi$ as the magnetic interaction strength,  the magnetic term disappears 
	in the limit. Hence, we end up with a limit operator $H_{\widetilde{V}\delta_\Sigma}$ which is again a Dirac operator with $\delta$-shell potential and only electrostatic and Lorentz-scalar interactions; however, we emphasize that here a
	different rescaling than \eqref{eq_scaling} in Theorem~\ref{theo_main_1} appears. 
	
	\begin{theo}\label{theo_main_2}
		Let $q\in L^\infty((-1,1);[0,\infty))$ with $\int_{-1}^1 q(s) \,ds =1$, let $\eta, \tau \in  C^1_b(\Sigma;\R)$  and assume that $d = \eta^2-\tau^2$
		satisfies the condition \eqref{eq_d_cond} and  $\inf_{x_\Sigma \in \Sigma} |d(x_\Sigma)| >0$.
		 Let $V$ and $V_\varepsilon$ be  as in \eqref{eq_V} and  \eqref{eq_V_eps},
        let the strongly localized magnetic potential $V_{m,\varepsilon} \in L^\infty(\R^{\theta};\C^{N\times N})$, ${\varepsilon \in (0,\varepsilon_1)}$, be defined by
		\begin{equation}\label{eq_V^m_eps}
			V_{m,\varepsilon}(x) := \begin{cases} \pi (\alpha \cdot \nu(x_\Sigma)) \frac{1}{\varepsilon} q\bigl(\tfrac{t}{\varepsilon}\bigr)& \textup{for } x= \iota(x_\Sigma,t)\in \Omega_\varepsilon, \\
				0&  \textup{else},
			\end{cases}
		\end{equation}
		and set 
		\begin{equation}\label{eq_magnetic_scaling}
			\widetilde{V} = \widetilde{\eta} I_N + \widetilde{\tau} \beta,\qquad (\widetilde{\eta},\widetilde{\tau}) = \frac{-2 }{\sqrt{d}\tan\bigl(\tfrac{\sqrt{d}}{2}\bigr)}(\eta,\tau),\qquad d = \eta^2-\tau^2.
		\end{equation} 
		Then, for all $z \in \C\setminus\R$  and $r \in (0,\tfrac{1}{2} )$ there exist $C>0$ and $\varepsilon' \in (0,\varepsilon_1)$ such that 
		\begin{equation*}
			\norm{(H_{V_\varepsilon} + V_{m,\varepsilon} -z)^{-1} - (H_{\widetilde{V}\delta_\Sigma}-z)^{-1}}_{L^2(\R^\theta;\C^N) \to L^2(\R^\theta;\C^N)} \leq C \varepsilon^{1/2-r}, \,\,  \varepsilon \in (0, \varepsilon').
		\end{equation*} 
		In particular, $H_{V_\varepsilon} + V_{m,\varepsilon}$ converges to $H_{\widetilde{V} \delta_\Sigma}$ in the norm resolvent sense as $\varepsilon \to 0$. 
	\end{theo}

	We note that the function
	\begin{equation*}
		(-\infty,\tfrac{\pi^2}{4}) \ni w \mapsto 4\tan^2\bigl(\tfrac{\sqrt{w}}{2}\bigr)  = \begin{cases}
			\tan^2\bigl(\tfrac{\sqrt{w}}{2}\bigr) \in [0,1),& w \in [0,\tfrac{\pi^2}{4}), \\\
			-\tanh^2\bigl(\tfrac{\sqrt{-w}}{2}\bigr) \in (-1,0), & w \in (-\infty,0),
		\end{cases}
	\end{equation*}
	is continuous, monotonically increasing and has its only root in zero. Hence, the assumption $\inf_{x_\Sigma \in \Sigma} \abs{d(x_\Sigma)} >0$, \eqref{eq_d_cond}, and the boundedness of $\eta$ and $\tau$ imply 
	\begin{equation*}
		\inf_{x_\Sigma \in \Sigma}  \Bigl|\tan^2\Bigr(\tfrac{\sqrt{d(x_\Sigma)}}{2}\Bigl)\Bigr| >0 \quad \textup{and} \quad\sup_{x_\Sigma \in \Sigma}  \Bigl|\tan^2\Bigr(\tfrac{\sqrt{d(x_\Sigma)}}{2}\Bigr)\Bigr|  <1.
	\end{equation*}
	 In particular,  $\widetilde{d} = \widetilde{\eta}^2 - \widetilde{\tau}^2 = 4/\tan^2\bigl(\tfrac{\sqrt{d}}{2}\bigr)$ is well-defined and
	\begin{equation*}
		\inf_{x_\Sigma \in \Sigma} \abs{\widetilde{d}(x_\Sigma)} = \frac{4}{\sup_{x_\Sigma \in \Sigma}  \Bigl|\tan^2\Bigl(\tfrac{\sqrt{d(x_\Sigma)}}{2}\Bigr)\Bigr| } > 4.
	\end{equation*}
	 Thus, 
	$\widetilde V$ in \eqref{eq_magnetic_scaling}
	is noncritical; cf. \eqref{eq_d_tilde_non_crit}. 
	We point out again that the rescaling \eqref{eq_magnetic_scaling} in Theorem~\ref{theo_main_2} differs 
	from the rescaling \eqref{eq_scaling} in Theorem~\ref{theo_main_1}.

	The following preparatory lemma is an essential ingredient in the proof of Theorem~\ref{theo_main_2} and also of independent interest.
	
	\begin{lemma}\label{lem_int_u^2_estimate}
		Let $\varepsilon_1$ be chosen as below \eqref{eq_iota}, $\varepsilon \in (0, \varepsilon_1)$, and let $\Omega_\varepsilon$ be the tubular neighbourhood of $\Sigma$. Then, for all
		$s \in (\tfrac{1}{2},\infty)$ there exists $C>0$ such that  
		\begin{equation*}
			\int_{\Omega_\varepsilon} \lvert u(x) \rvert^2 \,dx \leq C \varepsilon \lVert  u \rVert_{ H^s(\R^{\theta}\setminus \Sigma;\C^N)}^2,\qquad 
			u \in H^s(\R^{\theta} \setminus \Sigma;\C^N).
		\end{equation*}
	\end{lemma} 
	\begin{proof}
		Since $H^{s'}(\R^{\theta}\setminus \Sigma;\C^N)$ is continuously embedded in $H^{s}(\R^{\theta}\setminus \Sigma;\C^N)$ for $s < s'$, it is no restriction to assume $ s \in (\tfrac{1}{2},1]$. We start by considering $u = u_+ \oplus u_-$ with $u_\pm \in \mathcal{D}(\overline{\Omega_\pm};\C^N)$, where 
		$$\mathcal{D}(\overline{\Omega_\pm};\C^N) := \{ u|_{\Omega_\pm}: u \in \mathcal{D}(\mathbb{R}^\theta; \mathbb{C}^N) \}.$$
		Moreover, by \cite[Proposition~2.4 and Corollary~A.3]{BHS23} and  our choice of $\varepsilon_1>0$, see the lines below \eqref{eq_iota},  we obtain
		\begin{equation*}
			\begin{aligned}
				\int_{\Omega_\varepsilon} |u(x)|^2 \,dx & \leq C  \int_{-\varepsilon}^\varepsilon \int_{\Sigma} |u(x_\Sigma + t \nu(x_\Sigma))|^2 \, d\sigma(x_\Sigma)  \,dt \\
				&=  {C}{\varepsilon} \bigg(\int_{-1}^0 \int_{\Sigma} |u_+(x_\Sigma + t\varepsilon  \nu(x_\Sigma))|^2 \, d\sigma(x_\Sigma)  \,dt   \\
				&\quad \quad +\int_{0}^1 \int_{\Sigma} |u_-(x_\Sigma +  t\varepsilon \nu(x_\Sigma))|^2 \, d\sigma(x_\Sigma)\,dt \bigg).
			\end{aligned}
		\end{equation*}
		Next, we estimate the term $\int_{-1}^0 \int_{\Sigma} |u_+(x_\Sigma + t \varepsilon\nu(x_\Sigma))|^2 \,d\sigma(x_\Sigma) \,dt$. Note that the smoothness of $u$ implies that for  $t \in (-1,0)$ the function 
		\begin{equation*}
			\Sigma \ni x_\Sigma \mapsto u_+(x_\Sigma + t\varepsilon \nu(x_\Sigma))
		\end{equation*}
		coincides with  the trace of the function
		\begin{equation*}
			\R^\theta \ni x  \mapsto \widetilde{u}_+(x + t\varepsilon \nu_{\rm ext}(x)),
		\end{equation*}
		where $\nu_{\rm ext}$ is a $C^1_b$-extension of the unit normal vector $\nu$ to $\R^{\theta}$, which can be constructed in the same way as the extensions of $\eta$, $\tau$, and $V$ in \eqref{eq_interaction_extensions}, and $\widetilde{u}_+ = E u_+ \in H^1(\R^{\theta};\C^N)$ is the extension of $u_+$ to $\R^{\theta}$ defined by Stein's continuous extension operator
		\begin{equation*}
			E:H^s(\Omega_+;\C^N) \mapsto H^s(\R^{\theta};\C^N);
		\end{equation*}
		cf. \cite[Chapter~6,~Section~3,~Theorem~5]{S70}. Then,
		\begin{equation*}
			\int_{-1}^0 \int_{\Sigma} |u_+(x_\Sigma + t\varepsilon \nu(x_\Sigma))|^2 \,d\sigma(x_\Sigma)\,dt = \int_{-1}^0 \int_{\Sigma} \big|\tr \big(\widetilde{u}_+( (\cdot) + t \varepsilon\nu_{\rm ext})\big)(x_\Sigma)\big|^2 \,d\sigma(x_\Sigma)\,dt
		\end{equation*}
		and we can  estimate this term  by
		\begin{equation*}
			\begin{aligned}
				\int_{-1}^0 \int_{\Sigma} |\tr \big(\widetilde{u}_+( (\cdot) + &t \varepsilon\nu_{\rm ext})\big)(x_\Sigma)|^2 \,d\sigma(x_\Sigma)\,dt \\
				&=\int_{-1}^0 \norm{\tr (E u_+)((\cdot) + t\varepsilon \nu_{\rm ext})}_{L^2(\Sigma;\C^N)}^2   \,dt \\
				&\leq \int_{-1}^{0} \norm{\tr (E u_+)((\cdot) + t\varepsilon \nu_{\rm ext})}_{H^{s-1/2}(\Sigma;\C^N)}^2 \,dt\\
				&\leq C\int_{-1}^{0}\norm{ (E u_+)((\cdot) + t\varepsilon \nu_{\rm ext})}_{H^s(\R^{\theta};\C^N)}^2 \,dt.
			\end{aligned}
		\end{equation*}
		According to \cite[Proposition~3.3]{BHS23} the term $\norm{ (E u_+)((\cdot) + t\varepsilon\nu_{\rm ext})}_{H^s(\R^{\theta};\C^N)}$ is uniformly bounded with respect to $t\in(0,1)$ and $\varepsilon \in (0,\varepsilon_1)$ by $C \norm{ E u_+}_{H^s(\R^{\theta};\C^N)}$, which is in turn bounded by $C \norm{u_+}_{H^s(\Omega_+;\C^N)}$. Therefore,
		\begin{equation*}
			\int_{-1}^0 \int_{\Sigma} |u_+(x_\Sigma + t \nu(x_\Sigma))|^2 \,d\sigma(x_\Sigma)\,dt \leq C \norm{u_+}_{H^s(\Omega_+;\C^N)}^2;
		\end{equation*}
		in the same way one gets
		\begin{equation*}
			\int_{0}^1 \int_{\Sigma} |u_-(x_\Sigma + t \nu(x_\Sigma))|^2 \,d\sigma(x_\Sigma)\,dt \leq C \norm{u_-}_{H^s(\Omega_-;\C^N)}^2.
		\end{equation*} 
		This implies the assertion for $u \in \mathcal{D}(\overline{\Omega_+};\C^N) \oplus \mathcal{D}(\overline{\Omega_-};\C^N)$, which is a dense subspace of $H^s(\R^{\theta}\setminus\Sigma;\C^N) = H^s(\Omega_+;\C^N) \oplus H^s(\Omega_-;\C^N)$, see, e.g., \cite[Chapter~3]{M00}. Therefore, the assertion of the lemma follows for all $u\in  H^s(\R^{\theta}\setminus\Sigma;\C^N)$.
	\end{proof}

	\begin{proof}[Proof of Theorem~\ref{theo_main_2}]
		The proof is based on an idea that is also described below the proof of Theorem~2.6 in \cite[Section~8]{CLMT21} and the equality  
		\begin{equation}\label{eq_H_tilde_V_unitary_equivalence}
			H_{\widetilde{V}\delta_\Sigma} = U H_{(-4/\widetilde d)\widetilde{V} \delta_\Sigma} U,
		\end{equation} 
		where $U$ is the self-adjoint unitary multiplication operator induced by the function $w = \chi_{\Omega_+} - \chi_{\Omega_-}$ in $L^2(\R^\theta;\C^N)$. Recall that $\widetilde{V} = \widetilde{\eta} I_N + \widetilde{\tau} \beta$ is as in~\eqref{eq_magnetic_scaling}, $\widetilde{d} = \widetilde{\eta}^2-\widetilde{\tau}^2$, and by our assumptions we have $\inf_{x_\Sigma \in \Sigma} |\widetilde{d}(x_\Sigma)| > 4$, so that all terms in~\eqref{eq_H_tilde_V_unitary_equivalence} are well-defined. Equation \eqref{eq_H_tilde_V_unitary_equivalence} can be proven as in, e.g., \cite[Lemma~5.11]{BHSS22}, \cite[Section~4]{CLMT21} or \cite[Theorem~1.1]{M17}. Furthermore, according to \eqref{eq_magnetic_scaling} we have
		\begin{equation*}
			\widetilde{V}  =\frac{-2}{\sqrt{d}\tan\bigl(\tfrac{\sqrt{d}}{2}\bigr)} V \quad \textup{and} \quad \widetilde{d} = \widetilde{\eta}^2 - \widetilde{\tau}^2 = \frac{4}{\tan^2\bigl(\tfrac{\sqrt{d}}{2}\bigr)}. 
		\end{equation*}
		Set $ \widetilde{V}_{\textup{cl}} := \textup{tanc}\bigl(\tfrac{\sqrt{d}}{2}\bigr) V $ and observe
		\begin{equation*}
			\frac{-4}{\widetilde d}\widetilde{V} = -\tan^2\bigl(\tfrac{\sqrt{d}}{2}\bigr) \frac{-2}{\sqrt{d}\tan\bigl(\tfrac{\sqrt{d}}{2}\bigr)} V 
			= \frac{\tan\bigl(\tfrac{\sqrt{d}}{2}\bigr)}{\frac{\sqrt{d}}{2}}  V
			= \widetilde{V}_{\textup{cl}}. 
		\end{equation*}
		In particular, 
		\begin{equation}\label{eq_H_tilde_V_cl}
		H_{\widetilde{V} \delta_\Sigma}= UH_{\widetilde{V}_{\textup{cl}}\delta_\Sigma}U,
		\end{equation}
		and since $\widetilde{V}_{\textup{cl}} = \textup{tanc}\bigl(\tfrac{\sqrt{d}}{2}\bigr) V$, which is the rescaling from Theorem~\ref{theo_main_1}, it follows from Theorem~\ref{theo_main_1} that 
		$H_{V_\varepsilon}$ converges to $H_{\widetilde{V}_{\textup{cl}} \delta_\Sigma}$ in the norm resolvent sense.
		We then proceed in a similar way as in \cite{CLMT21}, that is, we provide unitary multiplication operators $W_\varepsilon$ such that 
		\begin{equation}\label{herrje33}
			W_\varepsilon^*H_{V_\varepsilon} W_\varepsilon = H_{V_\varepsilon} + V_{m,\varepsilon},
		\end{equation}
		where $V_{m,\varepsilon}$ is the strongly localized magnetic potential introduced in \eqref{eq_V^m_eps}, and $W_\varepsilon \to U$ for $\varepsilon \to 0$, so that the left hand side of \eqref{herrje33} converges in the norm resolvent sense to 
		$U H_{\widetilde{V}_{\textup{cl}} \delta_\Sigma} U=H_{\widetilde{V} \delta_\Sigma}$.
		
		We now start the main part of the proof by defining for $\varepsilon \in (0,\varepsilon_1)$ the function 
		\begin{equation*}
			w_\varepsilon: \R^{\theta} \to \C, \qquad w_\varepsilon(x):= \begin{cases}1, & x \in \Omega_+ \setminus \Omega_\varepsilon, \\ e^{i \pi \int_{-1}^{t/\varepsilon} q(s) \,ds}, & x = \iota(x_\Sigma,t)\in \Omega_\varepsilon, \\ -1, & x  \in \Omega_- \setminus \Omega_\varepsilon.
			\end{cases}
		\end{equation*}
		This function is well-defined according to the text below \eqref{eq_iota}. We define $W_\varepsilon$ to be the unitary multiplication operator in $L^2(\R^{\theta};\C^N)$ induced by $w_\varepsilon$. Using $\int_{-1}^1 q(s) \,ds = 1$ and tubular coordinates one shows that $w_\varepsilon \in W^1_\infty(\R^{\theta})$. Moreover, by \cite[eq. (2.11)]{Z23} and \cite[eq. (3.10)]{CLMT21} we have 
		$\nabla (\iota^{-1}(x))_2 = \nu(x_\Sigma) $ for ${x = \iota(x_\Sigma,t)  \in \Omega_\varepsilon}$ ($(\iota^{-1}(x))_2 = t$) and therefore the chain rule implies
		\begin{equation*}
			\nabla w_\varepsilon(x) =\begin{cases}
				\nu(x_\Sigma) \frac{i\pi q\bigl(\tfrac{t}{\varepsilon}\bigr)}{\varepsilon}w_\varepsilon(x), &  x = \iota(x_\Sigma,t) \in \Omega_\varepsilon,\\
				0, & \textup{else}.
			\end{cases}
		\end{equation*}
		These considerations and the definition of $V_{m,\varepsilon}$ show  
		\begin{equation}\label{eq_H_V_eps_W_eps}
			W_\varepsilon^*H_{V_\varepsilon} W_\varepsilon = H_{V_\varepsilon} - i \overline{w_\varepsilon}(\alpha \cdot \nabla w_\varepsilon) = H_{V_\varepsilon} +  V_{m,\varepsilon};
		\end{equation}
		cf. \cite[Section~8, below the proof of Theorem~2.6]{CLMT21}. We note that $w_\varepsilon$ converges pointwise to $w = \chi_{\Omega_+} - \chi_{\Omega_-}$ and therefore $W_\varepsilon$ converges in the strong sense to the operator $U$. In addition,  Lemma~\ref{lem_int_u^2_estimate} shows that for $\varepsilon \in (0,\varepsilon_1)$ the estimate 
		\begin{equation}\label{eq_diff_W_eps_U_H^1}
			\begin{aligned}
				\norm{(W_\varepsilon^* -U)u}_{L^2(\R^{\theta};\C^N)}^2&= \int_{\Omega_\varepsilon} |(\overline{w_\varepsilon(x)}- \chi_{\Omega_+}(x)+\chi_{\Omega_-}(x))u(x)|^2 \,dx\\
				&\leq 4\int_{\Omega_\varepsilon} |u(x)|^2 \,dx\\
			 & \leq C \varepsilon  \norm{u}_{H^1(\R^\theta\setminus \Sigma;\C^N)}^2, \qquad  u \in H^1(\R^{\theta}\setminus \Sigma),
			 \end{aligned}
		\end{equation}
		is also valid. Moreover, since $(H_{\widetilde{V}_{\textup{cl}}\delta_\Sigma} -z)^{-1}$ is closed as an operator 
        from $L^2(\R^{\theta};\C^N)$ to $H^1(\R^{\theta}\setminus \Sigma;\C^N)$ and defined on $L^2(\R^{\theta};\C^N)$, it is also bounded
        from $L^2(\R^{\theta};\C^N)$ to $H^1(\R^{\theta}\setminus \Sigma;\C^N)$. Thus, \eqref{eq_diff_W_eps_U_H^1} implies for $\varepsilon \in (0,\varepsilon_1)$
		\begin{equation}\label{hurra12}
			\|(W_\varepsilon^* -U)(H_{\widetilde{V}_{\textup{cl}}\delta_\Sigma}-z)^{-1}\|_{L^2(\R^{\theta};\C^N) \to L^2(\R^{\theta};\C^N)} \leq C \varepsilon^{1/2}.
		\end{equation}
		We use \eqref{eq_H_V_eps_W_eps} and \eqref{eq_H_tilde_V_cl} and estimate for $\varepsilon \in (0,\varepsilon')$ (with $\varepsilon'>0$ from Theorem~\ref{theo_main_1})
		\begin{equation*}
			\begin{aligned}
				\|(H_{V_\varepsilon} +& V_{m,\varepsilon} - z)^{-1} - (H_{\widetilde{V} \delta_\Sigma} - z)^{-1}\|_{L^2(\R^{\theta};\C^N) \to L^2(\R^{\theta};\C^N)} \\
				&=\| (W_\varepsilon^* H_{V_\varepsilon} W_\varepsilon - z)^{-1}- (UH_{\widetilde{V}_{\textup{cl}}\delta_\Sigma}U-z)^{-1}\|_{L^2(\R^{\theta};\C^N) \to L^2(\R^{\theta};\C^N)}\\
				&=\| W_\varepsilon^* (H_{V_\varepsilon}-z)^{-1} W_\varepsilon - U(H_{\widetilde{V}_{\textup{cl}}\delta_\Sigma}-z)^{-1}U\|_{L^2(\R^{\theta};\C^N) \to L^2(\R^{\theta};\C^N)}\\
				&\leq  \|W_\varepsilon^*((H_{V_\varepsilon}-z)^{-1}  -(H_{\widetilde{V}_{\textup{cl}}\delta_\Sigma}-z)^{-1})W_\varepsilon\|_{L^2(\R^{\theta};\C^N) \to L^2(\R^{\theta};\C^N)} \\
				&\hspace{20 pt}+ \|(W_\varepsilon^*-U)(H_{\widetilde{V}_{\textup{cl}}\delta_\Sigma}-z)^{-1}W_\varepsilon\|_{L^2(\R^{\theta};\C^N) \to L^2(\R^{\theta};\C^N)}\\
				&\hspace{20 pt}+ \|U(H_{\widetilde{V}_{\textup{cl}}\delta_\Sigma}-z)^{-1}(W_\varepsilon - U)\|_{L^2(\R^{\theta};\C^N) \to L^2(\R^{\theta};\C^N)}.
			\end{aligned}
		\end{equation*}
		Furthermore, since $U$ is unitary and self-adjoint and $W_\varepsilon$ is unitary we can 
		continue the above estimate for $r \in (0,\tfrac{1}{2})$ and obtain
		\begin{equation*}
			\begin{aligned}
				\|(H_{V_\varepsilon} +& V_{m,\varepsilon} - z)^{-1} - (H_{\widetilde{V} \delta_\Sigma} - z)^{-1}\|_{L^2(\R^{\theta};\C^N) \to L^2(\R^{\theta};\C^N)} \\
				&\leq \|(H_{V_\varepsilon}-z)^{-1}  -(H_{\widetilde{V}_{\textup{cl}}\delta_\Sigma}-z)^{-1}\|_{L^2(\R^{\theta};\C^N) \to L^2(\R^{\theta};\C^N)} \\
				&\hspace{20 pt}+ \|(W_\varepsilon^*-U)(H_{\widetilde{V}_{\textup{cl}}\delta_\Sigma}-z)^{-1}\|_{L^2(\R^{\theta};\C^N) \to L^2(\R^{\theta};\C^N)}\\
				&\hspace{20 pt}+ \|(W_\varepsilon^* - U)(H_{\widetilde{V}_{\textup{cl}}\delta_\Sigma}-\overline{z})^{-1}\|_{L^2(\R^{\theta};\C^N) \to L^2(\R^{\theta};\C^N)}.\\
				&\leq C \varepsilon^{1/2-r} + C \varepsilon^{1/2} \\
				&\leq C \varepsilon^{1/2-r},
			\end{aligned}
		\end{equation*}
		where the norm resolvent convergence of $H_{V_\varepsilon}$ to $H_{\widetilde{V}_{\textup{cl}} \delta_\Sigma}$ (see Theorem~\ref{theo_main_1}) and 
		\eqref{hurra12} were used in the penultimate estimate.
	\end{proof}
	
	 In the next corollary we observe that every Dirac operator with a given $\delta$-shell potential $\widetilde{V} \delta_\Sigma$, $\widetilde{V} = \widetilde{\eta} I_N + \widetilde{\tau} \beta$, $\widetilde{\eta},\widetilde{\tau} \in C_b^1(\Sigma;\R)$, and $\widetilde{d} = \widetilde{\eta}^2 -\widetilde{\tau}^2$ such that
	\begin{equation}\label{eq_cond_d_tilde_cor}
		\sup_{x_\Sigma \in \Sigma} |\widetilde{d}(x_\Sigma)|<4 \quad \textup{or} \quad \inf_{x_\Sigma \in \Sigma}|\widetilde{d}(x_\Sigma)|>4
	\end{equation} 
	holds, can  be approximated by a sequence of Dirac operators with strongly localized potentials. 
	
	\begin{cor}\label{cor_main_1}
		Let $q\in L^\infty((-1,1);[0,\infty))$ with $\int_{-1}^1 q(s) \,ds =1$, let $\widetilde\eta, \widetilde\tau \in  C^1_b(\Sigma;\R)$ 
		and $\widetilde{V} = \widetilde{\eta}I_N + \widetilde{\tau}\beta$, and assume that $\widetilde d = \widetilde\eta^{\,2}-\widetilde\tau^{\,2}$
		satisfies the condition \eqref{eq_cond_d_tilde_cor}.
		Define
		the interaction strengths $\eta,\tau\in C^1_b(\Sigma)$ by 
		\begin{equation}\label{eq_cor_def_eta_tau}
			(\eta,\tau) =
			\begin{cases}
				\frac{2 \arctan\left(\frac{\sqrt{\widetilde d}}{2}\right)}{\sqrt{\widetilde{d}}} (\widetilde{\eta},\widetilde{\tau}), &\quad  \text{if } \sup_{x_\Sigma \in \Sigma} |\widetilde{d}(x_\Sigma)| < 4,\\
				\frac{-2 \arctan\left(\frac{2}{\sqrt{\widetilde{d}}}\right)}{\sqrt{\widetilde{d}}} (\widetilde{\eta},\widetilde{\tau}), &\quad \text{if }\inf_{x_\Sigma \in \Sigma} |\widetilde{d}(x_\Sigma)| > 4,
			\end{cases}
		\end{equation}
		and let $V=\eta I_N+ \tau\beta$ and $V_\varepsilon$ be  as in   \eqref{eq_V_eps}.
		Then, for all $r \in (0,\tfrac{1}{2})$ and $z \in \C\setminus \R$ there exist $C >0$ and $\varepsilon' \in (0,\varepsilon_1)$ such that the following is true:
		\begin{enumerate}[\upshape(i)]
			\item If $\sup_{x_\Sigma \in \Sigma} |\widetilde{d}(x_\Sigma)| < 4$, then
		\begin{equation*}
			\norm{({H_{V_\varepsilon}} - z)^{-1} -(H_{\widetilde{V}\delta_\Sigma} - z)^{-1} }_{L^2(\R^\theta;\C^N) \to L^2(\R^\theta;\C^N)} \leq C \varepsilon^{1/2 -r}
		\end{equation*}
				for all   $\varepsilon \in (0, \varepsilon')$.
			\item If $\inf_{x_\Sigma \in \Sigma}|\widetilde{d}(x_\Sigma)|>4$, then
		\begin{equation*}
			\norm{({H_{V_\varepsilon}} + V_{m,\varepsilon} - z)^{-1} -(H_{\widetilde{V}\delta_\Sigma} - z)^{-1} }_{L^2(\R^\theta;\C^N) \to L^2(\R^\theta;\C^N) } \leq C \varepsilon^{1/2 -r}
		\end{equation*}
			for all  $\varepsilon \in (0, \varepsilon')$, where $V_{m,\varepsilon}$ is as in \eqref{eq_V^m_eps}.
	\end{enumerate}
	\end{cor}
	
	Note that in~\eqref{eq_cor_def_eta_tau} the convention $\frac{ \arctan(0)}{0} = 1$ is used. Moreover, we would like to point out that for constant $\widetilde{\eta}, \widetilde{\tau} \in \mathbb{R}$ the previous corollary is particularly interesting, as it shows that every Dirac operator with a $\delta$-potential and constant electrostatic and Lorentz-scalar interaction strengths satisfying  $|\widetilde{d}|\neq 4$ can be approximated by Dirac operators with strongly localized potentials.
	
	\begin{proof}[Proof of Corollary~\ref{cor_main_1}]
	Throughout this proof we use $\arctan^2(w) < \tfrac{\pi^2}{16}$ for $w \in [0,1]$ and $w \in i [0,1]$. 
	
	(i) The assumption $\sup_{x_\Sigma \in \Sigma} |\widetilde{d}(x_\Sigma)| < 4$ implies for $d = \eta^2 - \tau^2$
	\begin{equation*}
		\sup_{x_\Sigma \in \Sigma} d(x_\Sigma) = \sup_{x_\Sigma \in \Sigma} 4\arctan^2\Bigl(\tfrac{\sqrt{\widetilde{d}(x_\Sigma)}}{2}\Bigr)<\frac{\pi^2}{4}
	\end{equation*}
	and thus condition \eqref{eq_d_cond} is fulfilled. Hence, the assertion  follows from Theorem~\ref{theo_main_1}.
	
	(ii) The assumption $\inf_{x_\Sigma \in \Sigma} |\widetilde{d}(x_\Sigma)| > 4$ implies
	\begin{equation*}
		\sup_{x_\Sigma \in \Sigma} d(x_\Sigma) = \sup_{x_\Sigma \in \Sigma} 4\arctan^2\Bigl(\tfrac{2}{\sqrt{\widetilde{d}(x_\Sigma)}}\Bigr)<\frac{\pi^2}{4}
	\end{equation*}
	and thus condition \eqref{eq_d_cond} is fulfilled. Furthermore, since $\widetilde d=\widetilde\eta^{\,2}- \widetilde\tau^{\,2} \in  C^1_b(\Sigma;\R)$ we also have $\inf_{x_\Sigma \in \Sigma} |d(x_\Sigma)| >0$.
	Hence, the assertion  follows from Theorem~\ref{theo_main_2}.
	\end{proof}
	
	\section{Counterexamples}\label{sec_counter}
	
	We show in this section that the condition \eqref{eq_d_cond} for the norm resolvent convergence of $H_{V_\varepsilon}$ is optimal by providing suitable counterexamples. Throughout this section we assume that $V$ has the form  $V = \eta I_N + \tau \beta$ with $\eta,\tau \in \R$. In this situation \eqref{eq_d_cond} simplifies to 
	$d=\eta^2-\tau^2 < \tfrac{\pi^2}{4}$ and by \eqref{eq_scaling} we have 
	\begin{equation}\label{nochamal}
		\widetilde{d} = \widetilde{\eta}^2 - \widetilde{\tau}^2 = \textup{tanc}^2\bigl(\tfrac{\sqrt{d}}{2}\bigr) d= 4\tan^2\bigl(\tfrac{\sqrt{d}}{2}\bigr).
	\end{equation}
	
	In our first counterexample we treat the cases $d \geq \tfrac{\pi^2}{4}$ which lead to the critical interaction strength $\widetilde{d} = 4$. These are exactly the cases $d = (2k+1)^2 \tfrac{\pi^2}{4}$, $k \in \N_0$, and include, in particular, the important border case $d = \tfrac{\pi^2}{4}$. In this situation the operator
	$H_{\widetilde{V}\delta_\Sigma}$ is only essentially self-adjoint and hence $H_{V_\varepsilon}$ cannot converge in the  norm resolvent sense to $H_{\widetilde{V}\delta_\Sigma}$. However, if the interaction support $\Sigma$ is a compact $C^\infty$-hypersurface it turns out in Theorem~\ref{theo_main_counter_critical} that $H_{V_\varepsilon}$ does not even converge to the closure of $H_{\widetilde{V} \delta_\Sigma}$ 
	in the norm resolvent sense. In our second counterexample Theorem \ref{theo_main_3} we treat the case $d >  \tfrac{\pi^2}{4}$ and assume $d \neq (2k+1)^2 \tfrac{\pi^2}{4}$, $k \in \N_0$, 
	so that $\widetilde V$ is noncritical (i.e., $\widetilde{d} \neq 4$). If $\theta=2$ and the interaction support $\Sigma$ is the $y$-axis, 
	then we show that $H_{V_\varepsilon}$ does not converge to $H_{\widetilde{V} \delta_\Sigma}$ in the norm resolvent sense.
	
	\begin{theo}\label{theo_main_counter_critical} 
		Let $\Sigma \subset \R^{\theta}$ be a compact $C^\infty$-hypersurface, $q\in L^\infty((-1,1);[0,\infty))$ with $\int_{-1}^1 q(s) \,ds =1$, and 
		$\eta,\tau \in \R$ such that $d = \eta^2-\tau^2 = (2k+1)^2\tfrac{\pi^2}{4}$, $k \in \N_0$. Let
		$V=\eta I_N+ \tau\beta$ and $V_\varepsilon$ be  as in   \eqref{eq_V_eps}, and define $\widetilde V = \widetilde{\eta}I_N + \widetilde{\tau} \beta$ by \eqref{eq_scaling}. Then, $\widetilde{V}$ is critical (i.e., $\widetilde{d} = 4$), $H_{\widetilde{V} \delta_\Sigma}$ is essentially self-adjoint but not self-adjoint, and  $H_{V_\varepsilon}$ does not converge in the norm resolvent sense to the closure of $H_{\widetilde{V}\delta_\Sigma}$.
	\end{theo}
	
	\begin{proof}
		Since the convergence in norm resolvent sense is invariant with respect to bounded perturbations it is no restriction to assume $m>0$. 
		It is clear from \eqref{nochamal} that $d = (2k+1)^2 \tfrac{\pi^2}{4}$ leads to $\widetilde d=4$ and hence the interaction strength $\widetilde V$ is critical.
		The claims regarding the (essential) self-adjointness follow from \cite[Theorem~4.11]{BHOP20} for $\theta =2$ and from \cite[Theorem~3.1~(ii)]{Ben21} for $\theta =3$. Thus, it only remains to prove that $H_{V_\varepsilon}$ does not converge in the norm resolvent sense to the closure of $H_{\widetilde{V}\delta_\Sigma}$. 
		If $\Sigma$ is compact, then $\supp V_\varepsilon \subset \Omega_\varepsilon$ is compact and hence by \cite[Theorem~3.27~(ii)]{M00} $V_\varepsilon$ induces a compact operator from $H^1(\R^{\theta};\C^N)$ to $L^2(\R^{\theta};\C^N)$. In turn the resolvent difference
		\begin{equation*}
			(H-z)^{-1}-(H_{V_\varepsilon}-z)^{-1} = (H_{V_\varepsilon}-z)^{-1} V_\varepsilon (H-z)^{-1}
		\end{equation*}
		is compact in $L^2(\R^{\theta};\C^N)$, which shows 
		$$\sigma_{\textup{ess}}(H_{V_\varepsilon}) = \sigma_{\textup{ess}}(H) = (-\infty,-m]\cup[m,\infty).$$ Consequently, \cite[Satz~9.24~a)]{W00}  implies that if $H_{V_\varepsilon}$ would converge in norm resolvent sense to the closure of  $H_{\widetilde{V} \delta_\Sigma}$, 
		then also
		$$\sigma_{\textup{ess}}(\overline{H_{\widetilde{V}\delta_\Sigma}}) = \sigma_{\textup{ess}}(H) = (-\infty,-m] \cup [m, \infty).$$ 
		However, since $\widetilde{\eta}^2-\widetilde{\tau}^2 =4$ one has $\vert\widetilde\tau\vert <\vert\widetilde\eta\vert$ and 
		$$-\frac{\widetilde{\tau}}{\widetilde{\eta}} m \in \sigma_\textup{ess}(\overline{H_{\widetilde{V} \delta_\Sigma }})\cap (-m,m)$$ 
		according to \cite[Theorem~1.2 and  Theorem~1.3]{BHOP20} and \cite{BP24}; a contradiction. 
	\end{proof}
	
	Our second counterexample concerns the noncritical case and the special interaction support $\Sigma = \{0\} \times \R$ in $\R^2$. The idea of the proof is to use the direct integral method and to verify that
	$0 \in \sigma(H_{V_\varepsilon})$ for all $\varepsilon>0$ sufficiently small, while $0 \notin \sigma(H_{\widetilde{V} \delta_\Sigma})$; cf.
	Remark~\ref{remco} below for a further discussion.

	\begin{theo}\label{theo_main_3}
		Let $\Sigma = \{0\} \times \R  \subset \R^2$, let $q = \tfrac{1}{2} \chi_{(-1,1)}$, and let $\eta,\tau \in \R$ be such that $d = \eta^2-\tau^2 >  \tfrac{\pi^2}{4}$ and $d \neq (2k+1)^2 \pi^2$, $k \in \N_0$. Let
		$V=\eta I_N+ \tau\beta$ and $V_\varepsilon$ be  as in \eqref{eq_V} and  \eqref{eq_V_eps}, and define $\widetilde V = \widetilde{\eta}I_N + \widetilde{\tau} \beta$ by \eqref{eq_scaling}. Moreover, assume that $\widetilde{d} = \widetilde{\eta}^2 - \widetilde{\tau}^2 \neq 4$. Then, the operator $H_{V_\varepsilon}$ does not converge in norm resolvent sense to $H_{\widetilde{V}\delta_\Sigma}$.
	\end{theo}
	
	\begin{proof}
		 The main idea of the proof is to use the direct integral method and to show that if $\varepsilon >0$ is sufficiently small, then $0 \in \sigma(H_{V_\varepsilon})$. However, then norm resolvent convergence would imply $0 \in \sigma(H_{\widetilde{V}\delta_\Sigma})$, see, e.g., 
		 \cite[Chapter~IV, Theorem~3.1]{kato}, \cite[Theorem~VIII.23~(a)]{RS72} or \cite[Satz~9.24]{W00}. In the case that $\Sigma$ is the straight line the spectrum has been calculated explicitly in  \cite[eqs. (5.7), (6.7)]{BHT23}. In particular, if $m \neq 0$ and $(\widetilde{d}-4)m \widetilde{\tau} \leq 0$, then $0 \notin \sigma(H_{\widetilde{V}\delta_\Sigma})$, which yields  a contradiction. Note that it is no 
		 restriction to assume $m \neq 0$ and $(\widetilde{d}-4)m \widetilde{\tau}\leq 0$, since the bounded perturbation $m \beta$ does not influence the norm resolvent convergence. 
		 
		 We show  $0 \in \sigma(H_{V_\varepsilon})$ by applying the direct integral method.
		For this observe first that in the special case of Theorem~\ref{theo_main_3}  the operator $H_{V_\varepsilon}$ can be represented by 
		\begin{equation*}
			\begin{aligned}
				H_{V_\varepsilon}  &=   \sigma_1(-i\partial_1) + \sigma_2(-i \partial_2) + m \sigma_3 + (\eta I_2 + \tau \sigma_3 )\frac{\chi_{(-\varepsilon,\varepsilon) \times \R}}{2\varepsilon},  \\
				\dom H_{V_\varepsilon} &= H^1(\R^2;\C^{2}) \subset L^2(\R^2;\C^2);
			\end{aligned}
		\end{equation*}	
		cf. \eqref{eq_V_eps}, \eqref{eq_H_V_eps} and Section~\ref{sec_not}~\eqref{it_Dirac_matrices}.
		By identifying $L^2(\R^2;\C^2)$ with $ \int_{\R}^\oplus L^2(\R;\C^2) \, d\xi$, we get similar to \cite[eq. (2.3) and the text below]{BHT23}
		\begin{equation*}
			\mathcal{F}_{2}	H_{V_\varepsilon}  \mathcal{F}^{-1}_{2} = \int_{\R}^\oplus H_{V_\varepsilon}[\xi] \, d\xi
		\end{equation*}
		with the fiber operators
		\begin{equation*}
		\begin{split}
		H_{V_\varepsilon}[\xi]  &=   -i\sigma_1 \frac{d}{dx}+ \xi \sigma_2  + m \sigma_3 + (\eta I_2 + \tau \sigma_3)\frac{\chi_{(-\varepsilon,\varepsilon)}}{2\varepsilon}, \\
		\dom H_{V_\varepsilon}[\xi] &= H^1(\R;\C^{2}),
		\end{split}
		\end{equation*}	
		for $\xi \in \R$.
		We split the proof of $ 0 \in \sigma(H_{V_\varepsilon})$ for $\varepsilon>0$ sufficiently small in 4 steps. In \textit{Step~1} we find a condition for $0$ being in the point spectrum of $H_{V_\varepsilon}[\xi]$. \textit{Step~2} is an intermediate step in which we consider the inverse of the function $ {[\tfrac{\pi}{2},\pi) \ni u \mapsto -u \cot(u)}$. Using this function we verify in \textit{Step~3} that if $\varepsilon>0$ is sufficiently small, then there always exists an $\xi_\varepsilon >0$ such that $0 \in \sigma_{\textup{p}}(H_{V_\varepsilon}[\xi_\varepsilon])$. Finally, we prove in \textit{Step~4} that this implies $0 \in \sigma(H_{V_\varepsilon})$. 
		
		\textit{Step~1.}
		In this step we prove  for $\xi \in \R$ that  $0 \in \sigma_{\textup{p}}(H_{V_\varepsilon}[\xi])$ is equivalent to the condition
		\begin{equation}\label{eq_cond_mu}
			\cos(\mu_{\xi,\varepsilon}) + \frac{d-\mu_{\xi,\varepsilon}^2-2 \varepsilon \tau  m}{\sqrt{d - \mu_{\xi,\varepsilon}^2 - 4\varepsilon \tau  m}} \textup{sinc}(\mu_{\xi,\varepsilon}) = 0,
		\end{equation}
		where $\mu_{\xi,\varepsilon} := \sqrt{ d - 4\varepsilon^2\upsilon_\xi^2-4 \varepsilon \tau m  }$ and $\upsilon_\xi := \sqrt{\xi^2 + m^2}$. 
		
		Let us assume that there exists a nonzero function $u \in H^1(\R;\C^2)$ such that $H_{V_\varepsilon} [\xi] u = 0$. Then,
		\begin{equation*}
			\begin{aligned}
				\frac{d}{dx} u &= -i\sigma_1\left( \xi \sigma_2 + m \sigma_3 + (\eta I_2 + \tau \sigma_3) \frac{\chi_{(-\varepsilon,\varepsilon)}}{2\varepsilon}\right) u
			\end{aligned}
		\end{equation*}
		a.e. on  $\R$. Thus, there exist $w_1, w_2, w_3 \in \C^2$ such that 
		\begin{equation}\label{eq_mu_cond_ef}
			u(x) = \begin{cases} \exp(Ax)w_1, &  x \in (-\infty,-\varepsilon), \\ \exp(Bx)w_2, &x \in (-\varepsilon,\varepsilon),\\
				\exp(Ax)w_3, & x \in  (\varepsilon,\infty),\end{cases}
		\end{equation}
		where
		\begin{equation*}
			A = -i\sigma_1\big( \xi \sigma_2 + m \sigma_3\big) =\begin{pmatrix} \xi & im  \\ -im  & -\xi \end{pmatrix}  \text{ and }  B = A-i\sigma_1 (\eta I_2 + \tau \sigma_3) \frac{1}{2\varepsilon}.
		\end{equation*} 
		Note that $A$ has the distinct eigenvalues $\pm \upsilon_\xi = \pm \sqrt{\xi^2+m^2}$ and the corresponding orthogonal eigenvectors are given by $a_+ =( -im , \xi - \upsilon_\xi)$ and $a_-=(\xi - \upsilon_\xi,  -im)$.
		Since $u \in H^1(\R;\C^2)$, $u(x)$ has to converge to zero for $x \to \pm \infty$. Hence, $w_1 = c_1 a_+$ and $w_3 = c_3 a_-$ with $c_1,c_3 \in \C$. Moreover, $u \in H^1(\R;\C^2) \subset C(\R;\C^2)$ yields the conditions
		\begin{equation}\label{eq_u_A_B}
			\begin{aligned}
				\exp(-A\varepsilon)w_1 &=c_1  e^{-\varepsilon\upsilon_\xi} a_+ = \exp(-B \varepsilon)w_2\\
				\exp(A\varepsilon)w_3 &=c_3  e^{-\varepsilon\upsilon_\xi} a_- = \exp(B \varepsilon)w_2;
			\end{aligned}
		\end{equation}
		this implies
		\begin{equation}\label{eq_ev}
			c_1 a_+ - c_3 \exp(-2B\varepsilon)a_- = 0.
		\end{equation}
		Next, we write 
		\begin{equation*}
			\exp(-2B \varepsilon) =  \cos(i2B \varepsilon) - 2B \varepsilon \textup{sinc}(i2B \varepsilon).
		\end{equation*} 
		Note that $(i2B \varepsilon)^2 = \mu_{\xi,\varepsilon}^2 I_2 $ and hence $\cos(i2B \varepsilon) = \cos(\mu_{\xi,\varepsilon}) I_2 $ as well as $\textup{sinc}(i2B \varepsilon) = \textup{sinc}(\mu_{\xi,\varepsilon}) I_2 $. These considerations and $2 B \varepsilon  = 2 A \varepsilon  -i \sigma_1(\eta I_2 + \tau \sigma_3)$ yield
		\begin{equation*}
			\begin{aligned}
				\exp(-2B \varepsilon) a_- &= \cos(\mu_{\xi,\varepsilon}) a_-  -2B \varepsilon \textup{sinc}(\mu_{\xi,\varepsilon}) a_- \\
				&=\cos(\mu_{\xi,\varepsilon}) a_-  -(2 A \varepsilon  -i \sigma_1 (\eta I_2 + \tau \sigma_3) ) \textup{sinc}(\mu_{\xi,\varepsilon}) a_- \\
				&= \big(\cos(\mu_{\xi,\varepsilon}) + 2 \varepsilon \upsilon_\xi \textup{sinc}(\mu_{\xi,\varepsilon}) \big) a_- +i \sigma_1(\eta I_2 + \tau \sigma_3)  \textup{sinc}(\mu_{\xi,\varepsilon}) a_- .
			\end{aligned}
		\end{equation*}
		Since $a_+$  and $a_-$ are an orthogonal basis of $\C^2$, \eqref{eq_ev} is fulfilled if the scalar product of \eqref{eq_ev} with $a_+$ and $a_-$ is zero. This yields the system
		\begin{equation}\label{eq_mu_cond_0}
			\begin{aligned}
				0 &= c_1 |a_+|^2 - c_3 \big\langle i \sigma_1(\eta I_2 + \tau \sigma_3)  \textup{sinc}(\mu_{\xi,\varepsilon}) a_-,a_+ \big\rangle \\
				0 &= c_3 \Big( (\cos(\mu_{\xi,\varepsilon}) + 2 \varepsilon \upsilon_\xi \textup{sinc}(\mu_{\xi,\varepsilon}))|a_-|^2  \\
				&\hspace{38 pt}+  \big\langle i \sigma_1(\eta I_2 + \tau \sigma_3)  \textup{sinc}(\mu_{\xi,\varepsilon}) a_-,a_- \big\rangle \Big).
			\end{aligned}
		\end{equation}
		Note that $c_3 \neq 0$, as otherwise $c_1 =0$ and \eqref{eq_u_A_B} would imply  $w_2 =0$, and thus $w_1 = w_2 =w_3=0$, which in turn would lead to $u =0$. However, $u\neq0$  by assumption. Thus, the second line of the above system implies
		\begin{equation*}
			0 = (\cos(\mu_{\xi,\varepsilon}) + 2 \varepsilon \upsilon_\xi \textup{sinc}(\mu_{\xi,\varepsilon}))|a_-|^2  + \big\langle i \sigma_1(\eta I_2 + \tau \sigma_3)  \textup{sinc}(\mu_{\xi,\varepsilon}) a_-,a_- \big\rangle .
		\end{equation*}
		Using the relations $ \sigma_1^* = \sigma_1$, $\sigma_1 a_- = a_+$, $a_+\perp a_-$,  $\langle\sigma_3 a_-, a_+ \rangle = 2im(\xi - \upsilon_\xi)$ and $|a_-|^2 = m^2 + (\xi-\upsilon_\xi)^2$ we can simplify this equation to
		\begin{equation}\label{eq_mu_cond_1}
			\begin{aligned}
				0 =\cos(\mu_{\xi,\varepsilon}) +  \Big(2\varepsilon \upsilon_\xi -  2\frac{ (\xi - \upsilon_\xi)\tau m}{m^2+ (\xi - \upsilon_\xi)^2}\Big)\textup{sinc}(\mu_{\xi,\varepsilon}).
			\end{aligned}
		\end{equation}
		Next, we  use $m^2 = \upsilon_\xi^2 - \xi^2 = -(\xi + \upsilon_\xi)(\xi-\upsilon_\xi)$ and $2\varepsilon \upsilon_\xi  = \sqrt{d-\mu_{\xi,\varepsilon}^2-4\varepsilon \tau  m}$ to rewrite 
		\begin{equation}\label{eq_mu_cond_2}
			\begin{aligned}
				2\varepsilon \upsilon_\xi -  \frac{2 (\xi - \upsilon_\xi)\tau m}{m^2+ (\xi - \upsilon_\xi)^2} &=2\varepsilon \upsilon_\xi  -  2\frac{ (\xi - \upsilon_\xi) \tau m}{(-(\xi + \upsilon_\xi)+ (\xi - \upsilon_\xi))(\xi-\upsilon_\xi)}\\
				& = 2\varepsilon \upsilon_\xi +\frac{\tau m}{\upsilon_\xi}\\
				&=\frac{1}{2\varepsilon \upsilon_\xi} \Big( (2 \varepsilon \upsilon_\xi)^2 + 2\varepsilon \tau m \Big) \\
				&= \frac{d-\mu_{\xi,\varepsilon}^2 - 2\varepsilon \tau m}{2\varepsilon \upsilon_\xi}\\
				&=\frac{d-\mu_{\xi,\varepsilon}^2 -2\varepsilon\tau m}{\sqrt{d - \mu_{\xi,\varepsilon}^2 - 4\varepsilon \tau  m}}.
			\end{aligned}
		\end{equation}
		Plugging \eqref{eq_mu_cond_2} into \eqref{eq_mu_cond_1} gives us \eqref{eq_cond_mu}. 
		
		Now, we argue that the reverse direction is also true. If \eqref{eq_cond_mu} is fulfilled, then using  \eqref{eq_mu_cond_2} implies that \eqref{eq_mu_cond_1} is valid. Then, we fix a arbitrary $c_3 \in \C \setminus \{0\}$. For this  $c_3$ there exists exactly one $c_1 \in  \C$ such that \eqref{eq_mu_cond_0} is satisfied. Furthermore, \eqref{eq_ev} is also true. The choices $w_1 = c_1 a_+, w_2 = \exp(B \varepsilon)\exp(-A \varepsilon)w_1$ and $w_3 = c_3 a_-$ yield with \eqref{eq_u_A_B} that $u$ defined by \eqref{eq_mu_cond_ef} is in $H^1(\R;\C^N) = \dom H_{V_\varepsilon}[\xi]$. Moreover, by construction $u \in (\ker H_{V_\varepsilon}[\xi])\setminus \{0\}$ which implies $0 \in \sigma_{\textup{p}}(H_{V_\varepsilon}[\xi])$.
		
		\textit{Step~2.}
		Consider the function $a_0:[\tfrac{\pi}{2},\pi) \mapsto [0,\infty)$, $a_0(u) = -u \cot(u)$.
		It easy to see that $a_0$ is continuous and $a_0([\tfrac{\pi}{2},\pi)) = [0,\infty)$. Furthermore,
		\begin{equation*}
			\begin{aligned}
				a_0'(u) &= - \cot(u) +\frac{u}{\sin^2(u)}  = \frac{-\cos(u)\sin(u)+u}{\sin^2(u)} \\
				&= \frac{-\frac{1}{2}\sin(2u)+u}{\sin^2(u)} \geq  \frac{-\frac{1}{2}+\frac{\pi}{2}}{\sin^2(u)}>0, \qquad u \in [\tfrac{\pi}{2},\pi),
			\end{aligned}
		\end{equation*}
		and hence $a_0$ is monotonically increasing and bijective. Thus, the same holds for its inverse function 
		\begin{equation*}
			u_0:= a_0^{-1}:  [0,\infty) \mapsto [\tfrac{\pi}{2},\pi).
		\end{equation*} 
		For $a \in [0,\infty)$ we have\ $a = -u_0(a) \cot(u_0(a))$ and therefore multiplying with $\textup{sinc}(u_0(a))$ shows that $u_0$ fulfils the relation
		\begin{equation}\label{eq_u_0_relation}
			\cos(u_0(a)) + a \textup{sinc}(u_0(a)) = 0, \qquad  a \in [0,\infty).
		\end{equation}
		
		\textit{Step~3.} Now we use the function $u_0$ from the previous step to find for sufficiently small $\varepsilon>0$ a $\xi_\varepsilon$ such that $0 \in \sigma_{\textup{p}}(H_{V_\varepsilon}[\xi_\varepsilon])$. We start by claiming that there exists an $a_\varepsilon>0$ which fulfils
		\begin{equation}\label{eq_def_a_eps}
			a_\varepsilon = \frac{d-u_0^2(a_\varepsilon) - 2 \varepsilon \tau m}{\sqrt{d - u_0^2(a_\varepsilon) - 4 \varepsilon\tau m}}.
		\end{equation} 
		In fact, if we set
		\begin{equation*}
			b_\varepsilon := \sqrt{\min\{d,\pi^2\} - 4\varepsilon(|\tau m| +1)}
		\end{equation*}
		and choose $\varepsilon>0$ sufficiently small we can guarantee $b_\varepsilon \in (\tfrac{\pi}{2},\pi) $, and hence $u_0^{-1}(b_\varepsilon) \in (0,\infty)$. Since $u_0$ is monotonically increasing and continuous,  the function 
		\begin{equation*}
			F:[0 , u_0^{-1}(b_\varepsilon)] \to \R,  \quad F(a) =  a - \frac{d-u_0^2(a)-2 \varepsilon \tau m}{\sqrt{d - u_0^2(a) - 4 \varepsilon \tau m}}
		\end{equation*}
		is continuous and well-defined. The properties of the function $u_0$  and the assumption $d > \tfrac{\pi^2}{4}$ show that 
		\begin{equation*}
			F(0) = -\frac{d-\frac{\pi^2}{4}-2 \varepsilon \tau  m}{ \sqrt{d- \frac{\pi^2}{4} -4 \varepsilon \tau  m}}
		\end{equation*}
		 is smaller than zero for sufficiently small $\varepsilon>0$. Moreover, since $u_0^{-1}$ is continuous and $\lim_{b \to \pi} u_0^{-1}(b) = \infty$, we get
		  \begin{equation*}
		  	\begin{aligned}
		  	\lim_{\varepsilon \to 0} F(u_0^{-1}(b_\varepsilon)) &=  \lim_{\varepsilon \to 0} \Bigl(u_0^{-1}(b_\varepsilon) - \frac{d-\min\{d,\pi^2\} + 4\varepsilon(|\tau m| +1) -2\varepsilon\tau m}{\sqrt{d-\min\{d,\pi^2\} + 4\varepsilon(|\tau m| +1) -4\varepsilon \tau m}}\Bigr).
		  	\\
		  	&= \begin{cases}
		  		u_0^{-1}(\sqrt{d}), & d < \pi^2,\\
		  		\infty, & d \geq \pi^2,
		  	\end{cases}\\
	  			&>0. 
	  		\end{aligned}
		  \end{equation*}
	  	Thus, $F(u_0^{-1}(b_\varepsilon)) >0$ if $\varepsilon >0$ is sufficiently small.
	  	Hence, there exists an $a_\varepsilon \in (0, u_0^{-1}(b_\varepsilon))$ such that \eqref{eq_def_a_eps} is fulfilled. 
		  
		For this $a_\varepsilon $ we also have
		\begin{equation*}
			\begin{aligned}
				&d-u_0^2(a_\varepsilon)-4\varepsilon^2m^2 -4\varepsilon \tau m \\
				&\qquad > d- u_0^2(u_0^{-1}(b_\varepsilon)) -  4\varepsilon^2m^2 -4\varepsilon \tau m \\
				&\qquad= d-\min\{d,\pi^2\} + 4\varepsilon(|\tau m| +1) -4\varepsilon^2 m^2 -4\varepsilon \tau m \\
				&\qquad > 0
			\end{aligned}
		\end{equation*}
		for sufficiently small $\varepsilon>0$. Thus, 
		\begin{equation}\label{eq_def_xi_eps}
			\xi_{\varepsilon} := \frac{1}{2\varepsilon}\sqrt{d-u_0^2(a_\varepsilon)  - 4\varepsilon^2m^2-4\varepsilon \tau m} > 0
		\end{equation} 
		is well-defined. From \eqref{eq_def_xi_eps} and the definition of $\mu_{\xi,\varepsilon}$ below \eqref{eq_cond_mu} it follows that
		\begin{equation*}
			u_0(a_\varepsilon) = \sqrt{d - 4\varepsilon^2 \upsilon_{\xi_\varepsilon}^2 - 4 \varepsilon \tau m}= \mu_{\xi_\varepsilon,\varepsilon},
		\end{equation*}
		and plugging this expression into \eqref{eq_def_a_eps} yields 
		\begin{equation*}
			a_\varepsilon = \frac{d-\mu_{\xi_\varepsilon,\varepsilon}^2 - 2\varepsilon \tau  m}{\sqrt{d - \mu_{\xi_\varepsilon,\varepsilon}^2 - 4\varepsilon \tau  m}}.
		\end{equation*}
		Combining these relations with \eqref{eq_u_0_relation} shows that
		\eqref{eq_cond_mu} is fulfilled for $\xi = \xi_\varepsilon$. Thus, $0 \in \sigma_{\textup{p}}(H_{V_\varepsilon}[\xi_\varepsilon])$ by \textit{Step~1}.

		\textit{Step~4.}
		Finally, we show in this step $0 \in \sigma(H_{V_\varepsilon})$ for $\varepsilon>0$ chosen sufficiently small. This follows from  \cite[Theorem XIII.85 (d)]{RS77} if we can show that for all $\delta >0$ there is a $\gamma_\delta >0$  such that $(-\delta,\delta) \cap \sigma(H_{V_\varepsilon}[\xi]) \neq \emptyset$ for all $\xi \in (\xi_{\varepsilon} - \gamma_\delta, \xi_{\varepsilon} + \gamma_\delta)$.  We assume that our claim is not true. In this case there exists a $\delta' > 0$ and a sequence $(\xi_{n})_{n \in \N}$ such that $\xi_n \rightarrow \xi_{\varepsilon}$ for $n\to\infty$ and $(-\delta',\delta') \cap \sigma(H_{V_\varepsilon}[\xi_n]) = \emptyset$ for all $n \in \N$. Note that 
		\begin{equation*}
			\begin{aligned}
				&\norm{(H_{V_\varepsilon}[\xi_{\varepsilon}]-z)^{-1}-(H_{V_\varepsilon}[\xi_n]-z)^{-1}}_{ L^2(\R;\C^2) \to L^2(\R;\C^2)} \\
				&= \norm{(H_{V_\varepsilon}[\xi_{\varepsilon}]-z)^{-1}\sigma_2(\xi_n- \xi_{\varepsilon})(H_{V_\varepsilon}[\xi_n]-z)^{-1}} _{ L^2(\R;\C^2) \to L^2(\R;\C^2)}\\
				&\leq \frac{1}{(\textup{Im } z)^2} \abs{\xi_n - \xi_{\varepsilon}} \rightarrow 0,\quad n \to \infty,
			\end{aligned}
		\end{equation*}
		holds for $z \in \C \setminus \R$, i.e., $H_{V_\varepsilon}[\xi_n]$ converges in norm resolvent sense to $H_{V_\varepsilon}[\xi_{\varepsilon}]$. Moreover,  $(-\delta',\delta') \cap \sigma(H_{V_\varepsilon}[\xi_n]) = \emptyset$, $n \in \N$, and \cite[Theorem VIII.24 (a)]{RS72} imply the contradiction $(-\delta',\delta') \cap \sigma(H_{V_\varepsilon}[\xi_{\varepsilon}]) = \emptyset$.
	\end{proof}

	\begin{remark}\label{remco}
	    We note that in the situation of Theorem~\ref{theo_main_3} one can even show $\sigma(H_{V_\varepsilon}) = \R$ for all $\varepsilon>0$ which are sufficiently small; cf. \cite[Theorem~6.5]{SL24}. Furthermore, the statement of Theorem~\ref{theo_main_3}, that $H_{V_\varepsilon}$ does not converge in the norm resolvent sense to $H_{\widetilde{V} \delta_\Sigma}$, remains true in more general situations, e.g., if $\Sigma \subset \R^{\theta}$, $\theta \in \{2,3\}$, is as in  Section~\ref{sec_not}~\eqref{it_def_Sigma} and contains a flat part, see 
	    \cite[Theorem~6.7]{SL24} for more details.
	\end{remark}

	\section{Approximation of Dirac operators with $\delta$-shell potentials supported on rotated $C^2_b$-graphs}\label{sec_C^2_b_graphs}
	
	The main aim of this section is to prove Theorem~\ref{theo_main_1} for the case that $\Sigma$ is a rotated $C^2_b$-graph. In this case, there exist $\zeta \in C^2_b(\R^{\theta-1};\R)$ and $\kappa \in \textup{SO}(\theta)$ such that 
	\begin{equation}\label{eq_rotated_graph}
		\begin{aligned}
			\Omega_+ &= \{\kappa (x',x_\theta): (x',x_\theta) \in \R^{\theta} \textup{ and } x_\theta < \zeta(x') \}, \\
			 \Omega_- &= \R^\theta \setminus \overline{\Omega_+} = \{\kappa (x',x_\theta):(x',x_\theta) \in \R^{\theta} \textup{ and } x_\theta > \zeta(x')  \}, \\
			\Sigma &= \partial \Omega_+ = \Sigma_{\zeta,\kappa} := \{ \kappa(x',\zeta(x')) : x' \in \R^{\theta-1}\}.
		\end{aligned}
	\end{equation}
	
	To emphasize the importance of this result, let us rephrase  Theorem~\ref{theo_main_1} for the special case of a rotated $C^2_b$-graph.
	\begin{theo}\label{theo_rotated_graph}
		Let  $\Sigma$ be a rotated $C^2_b$-graph, let $q\in L^\infty((-1,1);[0,\infty))$ with $\int_{-1}^1 q(s) \,ds =1$ and assume that
		$\eta, \tau \in  C^1_b(\Sigma;\R)$ satisfy the condition 
		\begin{equation}\label{eq_d_cond2}
		\sup_{x_\Sigma \in \Sigma} d(x_\Sigma) < \frac{\pi^2}{4}, \qquad d = \eta^2 - \tau^2.
	\end{equation}
		 Let $V$ and $V_\varepsilon$ be  as in \eqref{eq_V} and  \eqref{eq_V_eps}, and define $\widetilde V$ by \eqref{eq_scaling}.
		Then, for all $z \in \C\setminus\R$  and $r \in (0,\tfrac{1}{2} )$ there exist $C>0$ and $\varepsilon_3 \in (0,\varepsilon_1)$ such that
		\begin{equation*}
			\norm{(H_{V_\varepsilon}-z)^{-1} - (H_{\widetilde{V}\delta_\Sigma}-z)^{-1}}_{L^2(\R^\theta;\C^N) \to L^2(\R^\theta;\C^N)} \leq C \varepsilon^{1/2-r}, \quad  \varepsilon \in (0, \varepsilon_3).
		\end{equation*} 
		In particular, $H_{V_\varepsilon}$ converges to $H_{\widetilde{V}\delta_\Sigma}$ in the norm resolvent sense as $\varepsilon \to 0$. 
	\end{theo}
	This theorem is a direct consequence of Proposition~\ref{prop_resolvent_convergence},  Proposition~\ref{prop_inv_I+BVq} and Proposition~\ref{prop_uniform_bdd_inv_B_eps}.

	Section~\ref{sec_C^2_b_graphs} is devoted to proving these propositions. We proceed as follows: In Section~\ref{sec_free_Dirac} we introduce various integral operators associated to the free Dirac operator. Then, in Section~\ref{sec_res_conv} we recall comparable resolvent  formulas for $H_{\widetilde{V}\delta_\Sigma}$ and $H_{V_\varepsilon}$ from \cite{BHS23}. Moreover, relying again on results from \cite{BHS23} we provide convergence results for the operators involved in these resolvent formulas and end the section by stating Proposition~\ref{prop_resolvent_convergence} which gives abstract conditions for the norm resolvent convergence of $H_{V_\varepsilon}$. Afterwards, we show in Section~\ref{sec_B_0} and Section~\ref{sec_I+B_eps} (in Proposition~\ref{prop_inv_I+BVq} and Proposition~\ref{prop_uniform_bdd_inv_B_eps}, respectively) that these conditions are met if \eqref{eq_d_cond2} is fulfilled. 
	
	Finally, let us mention that the restriction to  rotated $C^2_b$-graphs is only necessary in Section~\ref{sec_I+B_eps} and all results  from the Sections~\ref{sec_free_Dirac}--\ref{sec_B_0} remain valid for the general class of hypersurfaces described in Section~\ref{sec_not}~\eqref{it_def_Sigma}.

	\subsection{The free Dirac operator and associated integral operators}\label{sec_free_Dirac}
	
	Let $m \in \mathbb{R}$ and recall that the Dirac matrices $\alpha_1, \dots, \alpha_\theta, \beta \in \mathbb{C}^{N \times N}$ are given by~\eqref{def_Dirac_matrices_2d}--\eqref{def_Dirac_matrices_3d}. The free Dirac operator $H$ is the differential operator in $L^2(\mathbb{R}^\theta; \mathbb{C}^N)$ given by
	\begin{equation}\label{eq_def_free_Dirac}
		\begin{aligned}
			H  &:= - i (\alpha \cdot \nabla)  + m \beta , \qquad \dom H := H^1(\R^\theta;\C^N).
		\end{aligned}
	\end{equation}
	With the help of the Fourier transform one gets that $H$ is self-adjoint in $L^2(\R^\theta;\C^N)$ and $\sigma(H) = \left(-\infty, -\abs{m}\right] \cup \left[\abs{m}, \infty \right)$, see for instance \cite[Section~2]{BHT23} for $\theta=2$ and \cite[Theorem 1.1]{T92} for $\theta =3$. For $z \in \rho(H) = \C \setminus \sigma(H)$ we have
	\begin{equation*}
		R_z u(x) := (H-z)^{-1} u(x) = \int_{\mathbb{R}^\theta} G_z(x-y) u(y) \,d y, \qquad u \in L^2(\mathbb{R}^\theta; \mathbb{C}^N), ~x \in \mathbb{R}^\theta,
	\end{equation*}
	where  $G_z$ is given for $\theta=2$ and $x \in \mathbb{R}^2 \setminus \{ 0 \}$ by
	\begin{equation}\label{eq_G_z_2D}
		\begin{aligned}
			G_z(x) &= \frac{\sqrt{ z^2-m^2}}{2\pi} K_1\big(-i \sqrt{ z^2-m^2}\abs{x}\big)\frac{\alpha \cdot x}{\abs{x}} 
			\\
			&\hspace{100 pt}+\frac{1}{2\pi} K_0\big(-i \sqrt{ z^2-m^2}\abs{x}\big)\left(m\beta +  z I_2\right)
		\end{aligned}
	\end{equation}	
	and for $\theta=3$ and $x \in \mathbb{R}^3 \setminus \{ 0 \}$ by
	\begin{equation}\label{eq_G_z_3D}
		G_ z(x) = \left(  z I_4 + m \beta + i\left( 1 - i \sqrt{ z^2 -m^2}\abs{x} \right) \frac{ \alpha \cdot x }{\abs{x}^2}\right)\frac{e^{i\sqrt{ z^2 -m^2} \abs{x}}}{4 \pi \abs{x}};
	\end{equation}
	cf., e.g., \cite{BHOP20, BHSS22, T92}.
	Here, $K_0$ and $K_1$ denote the modified Bessel functions of the second kind of order zero and one, respectively.
	Note that $R_z$ is bounded in $L^2(\mathbb{R}^\theta; \mathbb{C}^N)$ and it can also be viewed as a bounded operator from $L^2(\R^\theta;\C^N)$ to $H^1(\R^\theta;\C^N)$.
	
	We move on to the discussion of potential and boundary integral operators associated with the free Dirac operator. In the following, let $z \in \rho(H) = \mathbb{C} \setminus ((-\infty, -|m|] \cup [|m|, \infty))$ be fixed and let $\Omega_\pm$ and $\Sigma \subset \mathbb{R}^\theta$ be as in \eqref{eq_rotated_graph}. First, we introduce the potential operator $\Phi_z: L^2(\Sigma; \mathbb{C}^N) \rightarrow L^2(\mathbb{R}^\theta; \mathbb{C}^N)$ by
	\begin{equation} \label{def_Phi_z}
		\begin{aligned}
			\Phi_ z   \varphi(x) &:= \int_\Sigma G_ z (x-y_\Sigma) \varphi(y_\Sigma) \,d\sigma(y_\Sigma),  \qquad \varphi \in L^2(\Sigma; \mathbb{C}^N),~x \in \R^\theta.
		\end{aligned}
	\end{equation}
	We note that $\Phi_z$ is indeed well-defined and bounded, see \cite[Lemma 2.1]{AMV14}. 
	Further properties of $\Phi_z$ are summarized in the following proposition. These results are well-known and a proof can be found in \cite[Appendix C]{BHS23}.
	
	\begin{proposition}\label{prop_Phi_z}
		Let $z \in \rho(H) = \mathbb{C} \setminus ((-\infty, -|m|] \cup [|m|, \infty))$ and let $\Phi_z$ be given by~\eqref{def_Phi_z}. Then, the following is true:
		\begin{itemize}
			\item[$\textup{(i)}$] For any $ r \in [0, \tfrac{1}{2}]$ the operator $\Phi_z$ gives rise to a bounded operator
			\begin{equation*}
				\Phi_z :H^r(\Sigma;\C^N) \to H^{r+1/2}(\R^\theta \setminus \Sigma ;\C^N).\\
			\end{equation*}
			\item[$\textup{(ii)}$] For $\varphi \in H^{1/2}(\Sigma; \mathbb{C}^N)$ one has $(-i (\alpha \cdot \nabla) + m \beta - zI_N) (\Phi_z \varphi)_\pm = 0$. 
			\item[$\textup{(iii)}$] The adjoint $\Phi_z^*: L^2(\mathbb{R}^\theta; \mathbb{C}^N) \to  L^2(\Sigma; \mathbb{C}^N)$ of $\Phi_z$ acts on $u \in L^2(\mathbb{R}^\theta; \mathbb{C}^N)$ as
			\begin{equation*}
				\Phi_{z}^* u(x_\Sigma) = \int_{\R^\theta} G_{\overline{ z}}(x_\Sigma -y) u(y) \, dy = \tr R_{\overline{z}} u(x_\Sigma), \qquad x_\Sigma \in \Sigma,
			\end{equation*}
			and $\Phi_z^*$ gives rise to a bounded operator $\Phi_z^*: L^2(\mathbb{R}^\theta; \mathbb{C}^N) \rightarrow H^{1/2}(\Sigma; \mathbb{C}^N)$.
		\end{itemize}
	\end{proposition}
	
	Finally, we introduce a family of boundary integral operators for the Dirac equation. Let $z \in \rho(H) = \mathbb{C} \setminus ((-\infty, -|m|] \cup [|m|, \infty))$. Then, we define the map $\mathcal{C}_z: H^{1/2}(\Sigma; \mathbb{C}^N) \rightarrow H^{1/2}(\Sigma; \mathbb{C}^N)$ by
	\begin{equation} \label{def_C_z}
		\mathcal{C}_z \varphi := \frac{1}{2} ( \tr^+  + \tr^-)\Phi_z \varphi, \qquad \varphi \in H^{1/2}(\Sigma; \mathbb{C}^N).
	\end{equation}
	We remark that the operator $\mathcal{C}_z$ can be represented as a strongly singular boundary integral operator, see for instance \cite[equation (4.5) and Proposition 4.4 (ii)]{BHSS22} for the case that $\Omega_+$ is bounded. However, for our purposes the representation in \eqref{def_C_z} is more convenient.
	The basic properties of $\mathcal{C}_z$ are stated in the following proposition. Again, a proof can be found in \cite[Appendix C]{BHS23}.
	
	\begin{proposition} \label{prop_C_z}
		Let $z \in \rho(H) = \mathbb{C} \setminus ((-\infty, -|m|] \cup [|m|, \infty))$ and let $\mathcal{C}_z$ be given by~\eqref{def_C_z}. Then, the following is true:
		\begin{itemize}
			\item[$\textup{(i)}$] For any $r \in [-\frac{1}{2},\frac{1}{2}]$ the map $\mathcal{C}_z$ has a bounded extension $\mathcal{C}_z: H^r(\Sigma;\C^N) \rightarrow H^r(\Sigma;\C^N)$.
			\item[$\textup{(ii)}$] For any $r \in(0,\frac{1}{2}]$ and $\varphi \in H^r(\Sigma;\C^N)$ one has
			\begin{equation*}
				\mathcal{C}_ z \varphi =  \pm \frac{i}{2}  (\alpha \cdot \nu) \varphi +  \tr^\pm \left(	\Phi_ z   \varphi\right)_{\pm}. 
			\end{equation*}
		\end{itemize}
	\end{proposition}
	
	\subsection{Formulas and corresponding convergence estimates for the resolvents of $H_{V_\varepsilon}$ and $H_{\widetilde{V}\delta_\Sigma}$}\label{sec_res_conv}
	In this section we introduce resolvent formulas for $H_{V_\varepsilon}$ and for  $H_{\widetilde{V}\delta_\Sigma}$ in terms of (Bochner) integral operators and study their convergence properties.
	
	Recall that $\mathcal{B}^0(\Sigma) = L^2((-1,1);L^2(\Sigma;\C^N))$ and that $W$ is the Weingarten map associated to $\Sigma$; cf. \cite[Definition 2.3]{BHS23} and \cite[Definition 2.2]{BEHL17}. We (formally) define for $\varepsilon>0$ and $z \in \rho(H)$ the following integral operators:
	\begin{subequations} 
		\begin{align}
			\label{def_A_eps}
			&\hspace{ 15pt}\begin{aligned}
				A_\varepsilon(z) &: \mathcal{B}^0(\Sigma) \to L^2(\R^{\theta};\C^N),\\
				A_\varepsilon( z)f(x) &:= \int_{-1}^1\int_{\Sigma} G_ z(x-y_\Sigma -\varepsilon s\nu(y_\Sigma))   f(s)(y_\Sigma)\\
				&\hspace{120 pt} \cdot \det(I-\varepsilon sW(y_\Sigma)) \, d\sigma(y_\Sigma) \,  ds,
			\end{aligned}\\
			\label{def_B_eps}
			&\begin{aligned}
				B_\varepsilon(z) &: \mathcal{B}^0(\Sigma) \to \mathcal{B}^0(\Sigma),\\
				B_\varepsilon( z)f(t)(x_\Sigma) &:=   \int_{-1}^1 \int_{\Sigma} G_ z(x_\Sigma +\varepsilon t\nu(x_\Sigma) -y_\Sigma - \varepsilon s\nu(y_\Sigma))   f(s)(y_\Sigma)\\
				&\hspace{120 pt }\cdot \det(I- \varepsilon sW(y_\Sigma)) \, d\sigma(y_\Sigma) \, ds,
			\end{aligned}\\
			\label{def_C_eps}
			&\begin{aligned}
				C_\varepsilon(z) &: L^2(\R^{\theta};\C^N) \to \mathcal{B}^0(\Sigma),  \\
				C_\varepsilon( z)u(t)(x_\Sigma) &:= \int_{\R^\theta} G_ z(x_\Sigma +  \varepsilon t\nu(x_\Sigma) -y) u(y)\, dy,
			\end{aligned}
		\end{align}
	\end{subequations}
	In the next proposition the well-definedness of these operators and also their connection to the resolvent of $H_{V_\varepsilon}$ are discussed.
	\begin{proposition}\label{prop_res_formula_eps}
		Let $z \in \rho(H)$, $q$ be as in \eqref{eq_q2}, $V$ be as in \eqref{eq_V}, $H_{V_\varepsilon} $ be  defined as in \eqref{eq_H_V_eps} and $\varepsilon_1>0$ be as below \eqref{eq_iota}. Then, for all $\varepsilon \in (0,\varepsilon_1)$ the operators given in \eqref{def_A_eps}--\eqref{def_C_eps} are well-defined and if $-1 \in \rho(B_\varepsilon(z)Vq)$, then $z \in \rho(H_{V_\varepsilon})$ and the resolvent identity
		\begin{equation*}
			(H_{V_\varepsilon}-z)^{-1} = (H-z)^{-1} - A_\varepsilon(z)Vq(I+B_\varepsilon(z)Vq)^{-1}C_\varepsilon(z)
		\end{equation*}
		holds.
	\end{proposition}
	\begin{proof}
		Note that $\varepsilon_1 > 0$ is chosen in exactly the same way as in \cite[Proposition~2.4]{BHS23}. Thus, the assertions follow from \cite[(3.1a)--(3.1c), Proposition~3.1 and Proposition~3.2]{BHS23}.
	\end{proof}
	Next, we introduce the operators $A_0(z)$, $B_0(z)$, and $C_0(z)$ which will turn out to be the limit operators of $A_\varepsilon(z)$, $B_\varepsilon(z)$, and $C_\varepsilon(z)$, respectively. For $z \in \rho (H)$, they are defined by
	\begin{equation*}
		\begin{aligned}
			A_0(z) &: \mathcal{B}^0(\Sigma) \to L^2(\R^{\theta};\C^N),\\
			A_0( z)f &:= \Phi_z \int_{-1}^1 f(t) \,dt, \\
			B_0(z) &: \mathcal{B}^0(\Sigma) \to \mathcal{B}^0(\Sigma),\\
			B_0( z) f(t)&:= \frac{i}{2}(\alpha \cdot \nu)\int_{-1}^1 \sign(t-s)f(s) \,ds + \mathcal{C}_z \int_{-1}^1  f(s) \, ds,\\
		C_0(z) &: L^2(\R^{\theta};\C^N) \to \mathcal{B}^0(\Sigma),  \\
			C_0( z)u(t) &:= \Phi_{\overline{z}}^*u,
		\end{aligned}
	\end{equation*}
	where  $\Phi_z$ is the operator defined in $\eqref{def_Phi_z}$  and $\mathcal{C}_z$ is the extension of the operator defined in  \eqref{def_C_z} to $L^2(\Sigma;\C^N)$, see also Proposition~\ref{prop_C_z}~(i). Using the identifications described in Section~\ref{sec_not}~\eqref{it_Bochner}, one sees that the operators $A_0(z)$, $B_0(z)$, and $C_0(z)$ are well-defined and bounded since by Proposition~\ref{prop_Phi_z}~(i) $\Phi_z$ is a bounded operator from $L^2(\Sigma;\C^N)$ to $L^2(\R^{\theta};\C^N)$ and by Proposition~\ref{prop_C_z}~(i) $\mathcal{C}_z$ is a bounded operator in $L^2(\Sigma;\C^N)$. The operator $B_0(z)$ can also be represented by 
	\begin{equation}\label{B_0_alt_rep}
		B_0(z) = T(\alpha \cdot \nu) + \mathfrak{I} \mathcal{C}_z \mathfrak{I}^*,
	\end{equation}
	where  $\mathfrak{J}$ is defined in Section~\ref{sec_not}~\eqref{it_Bochner} and
	\begin{equation*}
		\begin{aligned}
			T &:\mathcal{B}^0(\Sigma) \to \mathcal{B}^0(\Sigma), \qquad
			&Tf(t) &:= \frac{i}{2}  \int_{-1}^1 \textup{sign}(t-s)f(s) \,ds.
		\end{aligned}
	\end{equation*}
	Moreover, since $\mathcal{C}_z$  also acts as a bounded operator in $H^r(\Sigma;\C^N)$ for $r \in [-\tfrac{1}{2},\tfrac{1}{2}]$, see Proposition~\ref{prop_C_z}~(i), $B_0(z)$ acts also as a bounded operator in $\mathcal{B}^r(\Sigma)$ for $r \in [-\tfrac{1}{2},\tfrac{1}{2}]$.
	
	Our next goal is to show a  resolvent formula for $H_{\widetilde{V}\delta_\Sigma}$, which is defined by \eqref{eq_def_H_V_tilde}, in terms of $A_0(z), B_0(z)$, and $C_0(z)$. 
	\begin{proposition}\label{prop_H_V_tilde}
		Let  $z \in \rho(H)$, $q$ and $V= \eta I_N + \tau \beta$ be as described in \eqref{eq_q2} and \eqref{eq_V}, respectively, and $ d= \eta^2-\tau^2$ such that \eqref{eq_scaling_cond} is fulfilled. Moreover, let $(\widetilde{\eta},\widetilde{\tau}) = \textup{tanc}\bigl(\tfrac{\sqrt{d}}{2}\bigr) (\eta,\tau)$, $\widetilde{V} = \widetilde{\eta} I_N + \widetilde{\tau} \beta$ and $\widetilde{d} = \widetilde{\eta}^2 - \widetilde{\tau}^2$ fulfil \eqref{eq_d_tilde_non_crit}. If $-1 \in  \rho(B_0(z)Vq)$, then for the self-adjoint operator $H_{\widetilde{V} \delta_\Sigma}$ the resolvent identity 
		\begin{equation*}
			(H_{\widetilde{V}\delta_\Sigma}-z)^{-1} = (H-z)^{-1} - A_0(z)Vq(I+B_0(z)Vq)^{-1}C_0(z)
		\end{equation*}
		is valid.
	\end{proposition}
	\begin{proof}
		Since $\widetilde{d}$ fulfils \eqref{eq_d_tilde_non_crit}, $H_{\widetilde{V}\delta_\Sigma}$ is self-adjoint by the text above \eqref{eq_d_tilde_non_crit}. Moreover, in the same way as in \cite[below eq.~(4.15)]{BHS23} one can show that if $-1 \in  \rho(B_0(z)Vq)$,  then  for $u \in L^2(\R^\theta;\C^N)$ holds
		\begin{equation*}
			\big((H-z)^{-1} - A_0(z)Vq(I+B_0(z)Vq)^{-1}C_0(z)\big) u \in \dom H_{\widetilde{V}\delta_\Sigma}
		\end{equation*}
		and 
		\begin{equation*}
			(H_{\widetilde{V}\delta_\Sigma}-z)\big((H-z)^{-1} - A_0(z)Vq(I+B_0(z)Vq)^{-1}C_0(z)\big)u =u.
		\end{equation*}
		This shows that the resolvent identity is true. 
	\end{proof}
	In the next proposition we summarize the convergence properties of $A_\varepsilon(z)$,$B_\varepsilon(z)$, and $C_\varepsilon(z)$; cf. \cite[Proposition~3.7, Proposition~3.8, and Proposition~3.10]{BHS23} for the proof.
	\begin{proposition}\label{prop_conv_res}
		Let $z \in \rho(H)$ and $\varepsilon_1>0$ be as below \eqref{eq_iota}. Then, there exists an $\varepsilon_2 \in (0,\varepsilon_1)$ such that  $A_\varepsilon(z)$, $B_\varepsilon(z)$, and $C_\varepsilon(z)$
		are uniformly bounded operators with respect to  $\varepsilon \in (0, \varepsilon_2)$. Moreover,  $C_\varepsilon(z)$ and $C_0(z)$ act also as uniformly bounded operators from  $L^2(\R^{\theta};\C^N)$ to $\mathcal{B}^{1/2}(\Sigma)$ and for $r \in (0,\frac{1}{2})$ there exists a $C >0$ such that one has for $\varepsilon \in (0,\varepsilon_2)$
		\begin{equation*}
			\begin{aligned}
				\norm{A_\varepsilon(z)-A_0(z)}_{0 \to L^2(\R^{\theta};\C^N)} &\leq C \varepsilon^{1/2-r}, \\
				\norm{B_\varepsilon(z)-B_0(z)}_{1/2 \to 0} &\leq C \varepsilon^{1/2-r},\\
				\norm{C_\varepsilon(z)-C_0(z)}_{L^2(\R^{\theta};\C^N) \to 0} &\leq C \varepsilon^{1/2-r}.
			\end{aligned}
		\end{equation*}
	\end{proposition}

	After summarizing the convergence properties of $A_\varepsilon(z)$, $B_\varepsilon(z)$, and $C_\varepsilon(z)$ we state in Proposition~\ref{prop_resolvent_convergence} conditions for the norm resolvent convergence of $H_{V_\varepsilon}$.

	\begin{proposition}[{\cite[Proposition~3.12 and Remark 4.5]{BHS23}}]\label{prop_resolvent_convergence}
		Let  the assumptions of Proposition~\ref{prop_H_V_tilde} hold. Moreover, let $r \in (0,\frac{1}{2})$ and $\varepsilon_2>0$ be as in Proposition \ref{prop_conv_res}. Assume that the following conditions hold: 
		\begin{itemize}
			\item[(i)] There exists an $\varepsilon_3 \in (0,\varepsilon_2)$ such that  $(I+B_\varepsilon(z)Vq)^{-1}$ exists for $\varepsilon \in (0,\varepsilon_3)$ and is uniformly bounded in $\mathcal{B}^0(\Sigma)$.
			\item[(ii)] $I + B_0(z)Vq$ is bijective in $\mathcal{B}^{1/2}(\Sigma)$.
		\end{itemize}
		Then,
		\begin{equation*}
			\norm{(H_{V_\varepsilon}-z)^{-1} - (H_{\widetilde{V}\delta_\Sigma}-z)^{-1}}_{L^2(\R^\theta;\C^N) \to L^2(\R^\theta;\C^N)} \leq C \varepsilon^{1/2-r} 
		\end{equation*}
		for $\varepsilon \in (0,\varepsilon_3)$. In particular, $H_{V_\varepsilon}$ converges for $\varepsilon \to 0$ to $H_{\widetilde{V}\delta_\Sigma}$ in norm resolvent sense.
	\end{proposition}

	According to Proposition~\ref{prop_resolvent_convergence} it is essential to study the operators $I + B_0(z)Vq$ and $I + B_\varepsilon(z)Vq$ in order to find explicit conditions for the norm resolvent convergence of $H_{V_\varepsilon}$. In Section~\ref{sec_B_0} and Section~\ref{sec_I+B_eps} we show that \eqref{eq_d_cond2} guarantees that the requirements (i) and (ii) of Proposition~\ref{prop_resolvent_convergence} are met.
	
	\subsection{Analysis of $I + B_0(z)Vq$}\label{sec_B_0}
	We study the operator $I + B_0(z)Vq$ and are particularly interested under which conditions (ii) of Proposition~\ref{prop_resolvent_convergence} is fulfilled, i.e. under which conditions  $I + B_0(z)Vq$ is bijective in $\mathcal{B}^{1/2}(\Sigma)$. First, we state an auxiliary result about the invertibility of $I + \mathcal{C}_z \widetilde V$ with $\mathcal{C}_z$ defined by \eqref{def_C_z}, see also Proposition~\ref{prop_C_z}~(i).
	\begin{lemma}\label{lem_I+C_z_inverse}
		Let $z \in \C\setminus\R$, $r \in [0,\tfrac{1}{2}]$, $\widetilde{\eta},\widetilde{\tau} \in C^1_b(\Sigma;\R)$, and $\widetilde{V} = \widetilde{\eta} I_N + \widetilde{\tau} \beta$ such that $\widetilde{d} = \widetilde{\eta}^2 - \widetilde{\tau}^2$ fulfils \eqref{eq_d_tilde_non_crit}.
		Then, the operator $I + \mathcal{C}_z \widetilde V$ is continuously invertible in $H^{r}(\Sigma;\C^N)$.   
	\end{lemma}
	\begin{proof}
		We split the proof in three steps. In \textit{Step~1} we show that   $\widetilde{V}$ and $\mathcal{C}_z$ act  as bounded operators in $H^r(\Sigma;\C^N)$ for $r \in [0,\tfrac{1}{2}]$ and $z \in \C \setminus \R$. Afterwards, we show in \textit{Step~2} that for $z \in \C \setminus \R$ the operator $I + \widetilde{V}\mathcal{C}_z$ is continuously invertible  in $H^{1/2}(\Sigma;\C^N)$. In \textit{Step~3} we use  \textit{Step~1~\&~2}  to  prove the assertion.
		
		\textit{Step~1.} Let $z \in \C \setminus \R$ and $r \in [0,\tfrac{1}{2}]$. The assumption $\widetilde{\eta},\widetilde{\tau} \in C^1_b(\Sigma;\R)$ implies $\widetilde{V} = \widetilde{\eta} I_N + \widetilde{\tau} \beta \in C^1_b (\Sigma;\C^{N \times N}) \subset W^1_\infty(\Sigma;\C^{N \times N})$. Hence, $\widetilde{V}$ induces a bounded multiplication operator in $H^r(\Sigma;\C^N)$  which is bounded by $\|\widetilde{V}\|_{W^1_\infty(\Sigma;\C^{N \times N})}$. Moreover, $\mathcal{C}_z$ acts as a bounded operator in $H^r(\Sigma;\C^N)$ by Proposition~\ref{prop_C_z}~(i).
		
		\textit{Step~2.} We already know from \textit{Step~1} that $\mathcal{C}_z$ and $\widetilde{V}$  act as bounded operators in $H^{1/2}(\Sigma;\C^N)$ for $z \in \C\setminus\R$. Hence, in order to prove that $I + \widetilde{V}\mathcal{C}_z$ is continuously invertible in $H^{1/2}(\Sigma;\C^N)$, it suffices to show that $I + \widetilde{V}\mathcal{C}_z$ is bijective in $H^{1/2}(\Sigma;\C^N)$. We begin with the injectivity. Let $\psi \in H^{1/2}(\Sigma;\C^N)$ such that $ (I+ \widetilde{V}\mathcal{C}_z)\psi = 0$. We set $u = \Phi_z \psi$ with $\Phi_z$ defined by \eqref{def_Phi_z}. Then, Proposition~\ref{prop_Phi_z}~(i) implies $u \in H^{1}(\R^\theta \setminus  \Sigma;\C^N)$. Furthermore, Proposition~\ref{prop_Phi_z}~(ii) yields  $(- i (\alpha \cdot \nabla) + m \beta -z I_N)u_{\pm} = 0$ and Proposition~\ref{prop_C_z}~(ii) gives us \begin{equation}\label{eq_u_trace_properties}
			i(\alpha \cdot \nu)(\tr^+ - \tr^-)u = \psi \quad \textup{and} \quad \frac{1}{2}(\tr^+ + \tr^-)u = \mathcal{C}_z \psi.
		\end{equation}
		Thus, \eqref{eq_u_trace_properties} leads to
		\begin{equation}\label{eq_u_trace_properties_0}
			i(\alpha \cdot \nu)(\tr^+ - \tr^-)u +  \widetilde{V}\frac{1}{2}(\tr^+ + \tr^-)u= \psi+   \widetilde{V}\mathcal{C}_z\psi= 0
		\end{equation}
		and hence $u \in \ker (H_{\widetilde{V}\delta_\Sigma} -z)$. Since $\widetilde{d}$ fulfils \eqref{eq_d_tilde_non_crit}, $H_{\widetilde{V}\delta_\Sigma}$ is by the text above \eqref{eq_d_tilde_non_crit} self-adjoint and therefore $\ker (H_{\widetilde{V}\delta_\Sigma} -z) = \{0\}$; implying $u =0$ and therefore \eqref{eq_u_trace_properties} shows $\psi = 0$. Now we turn to the surjectivity. Let $\varphi \in H^{1/2}(\Sigma;\C^N)$. Then, according to \cite[Theorem 2]{M87} there exist $w_\pm \in H^{1}(\Omega_\pm;\C^N)$ such that $\tr^\pm w_\pm = \frac{\mp i (\alpha \cdot \nu)}{2} \varphi \in H^{1/2}(\Sigma;\C^N)$. Next, we set $w = w_+ \oplus w_- \in H^{1}(\R^{\theta}\setminus\Sigma;\C^N)$ and see 
		\begin{equation}\label{eq_w_trace_properties}
			i(\alpha \cdot \nu)(\tr^+ - \tr^-)w  = \varphi \quad \textup{as well as} \quad \frac{1}{2}(\tr^+ + \tr^-)w  = 0.
		\end{equation} 
		Moreover, let 
		\begin{equation*}
			v:=(H_{\widetilde{V}\delta_\Sigma}-z)^{-1}[(-i(\alpha\cdot \nabla) + m\beta -zI_N) w_+  \oplus(-i(\alpha\cdot \nabla) + m\beta -zI_N) w_-] .
		\end{equation*}
		Then, $v \in \dom H_{\widetilde{V}\delta_\Sigma}$ and $(-i (\alpha \cdot \nabla) + m \beta - zI_N) (v-w)_\pm = 0$, and thus due to \cite[eq. (C.4) and the text below]{BHS23} there exists a $\psi \in H^{1/2}(\Sigma;\C^N)$ such that $\Phi_z \psi = w-v$. Hence, we use the relations \eqref{eq_u_trace_properties} (for $u  = \Phi_z \psi$) as in \eqref{eq_u_trace_properties_0}, \eqref{eq_w_trace_properties} and $v \in \dom H_{\widetilde{V}\delta_\Sigma}$ to obtain
		\begin{equation*}
			\begin{aligned}
				(I + \widetilde{V} \mathcal{C}_z)\psi 
				&= i(\alpha \cdot \nu)(\tr^+ - \tr^-)\Phi_z \psi  +   \widetilde{V}\frac{1}{2}(\tr^+ + \tr^-)\Phi_z \psi\\
				&= i(\alpha \cdot \nu)(\tr^+ - \tr^-)(w-v)   + \widetilde{V}\frac{1}{2}(\tr^+ + \tr^-)(w-v)  \\
				& =  i(\alpha \cdot \nu)(\tr^+ - \tr^-)w   +  \widetilde{V}\frac{1}{2}(\tr^+ + \tr^-)w  = \varphi.
			\end{aligned}
		\end{equation*}
		This proves the surjectivity and completes \textit{Step~2}.
		
		\textit{Step~3.}
		First, let $r =\tfrac{1}{2}$. In this case we already know from \textit{Step~2} that $I+ \widetilde{V}\mathcal{C}_z$ is continuously invertible in $H^{1/2}(\Sigma;\C^N)$. Hence, elementary calculations show that then $I - \mathcal{C}_z(I+\widetilde{V}\mathcal{C}_z)^{-1}\widetilde{V}$ is the inverse of $I + \mathcal{C}_z \widetilde{V}$. Moreover, it follows from \textit{Step~1 \& Step~2} that $I - \mathcal{C}_z(I+\widetilde{V}\mathcal{C}_z)^{-1}\widetilde{V}$ is also bounded as  an operator in $H^{1/2}(\Sigma;\C^N)$.
		Consequently, the assertion is true for $r=1/2$. Next, let us consider the case $r \in [0,\frac{1}{2})$. From the proof of \cite[Proposition 2.9]{BHS23} we obtain that $(\mathcal{C}_{\overline{z}} \upharpoonright H^{1/2}(\Sigma;\C^N))'$ (where $'$ is used to denote the anti-dual operator) is a continuous  extension of $\mathcal{C}_z$ to $H^{-1/2}(\Sigma;\C^N)$. Moreover, using the symmetry of $\widetilde{V}$ and the fact that $\widetilde{V}$ induces a bounded multiplication operator in $H^{1/2}(\Sigma;\C^N)$ shows that $\widetilde{V}$ can also be extended to a bounded multiplication operator in $H^{-1/2}(\Sigma;\C^N)$. Therefore, 
		\begin{equation}\label{eq_I+C_z_ext}
			\big((I+ \widetilde{V}\mathcal{C}_{\overline{z}}) \upharpoonright H^{1/2}(\Sigma;\C^N) \big)' = I+ (\mathcal{C}_{\overline{z}} \upharpoonright H^{1/2}(\Sigma;\C^N))' (\widetilde{V} \upharpoonright H^{1/2}(\Sigma;\C^N))' 
		\end{equation}is a continuous extension of $I+ \mathcal{C}_z\widetilde{V}$ in $H^{-1/2}(\Sigma;\C^N)$. \textit{Step~2} shows that $I+ \widetilde{V}\mathcal{C}_{\overline{z}}$ is continuously invertible in $H^{1/2}(\Sigma;\C^N)$ and therefore the operator in \eqref{eq_I+C_z_ext} has the bounded inverse $\big((I+ \widetilde{V}\mathcal{C}_{\overline{z}})^{-1} \upharpoonright H^{1/2}(\Sigma;\C^N) \big)'$. Hence, one can use interpolation to show the assertion for $r \in [0,\tfrac{1}{2}]$; cf. \cite[eq.~(2.2)]{BHS23} and \cite[Theorem~B.11]{M00}.
	\end{proof}
	
	\begin{proposition}\label{prop_inv_I+BVq}
		Let  $z \in \C\setminus\R$, $r \in [0,\tfrac{1}{2}]$, $q$ and $V= \eta I_N + \tau \beta$ be as described in \eqref{eq_q2} and \eqref{eq_V}, respectively, and $ d= \eta^2-\tau^2$ such that \eqref{eq_scaling_cond} is fulfilled. Moreover, let $(\widetilde{\eta},\widetilde{\tau}) = \textup{tanc}\bigl(\tfrac{\sqrt{d}}{2}\bigr) (\eta,\tau)$, $\widetilde{V} = \widetilde{\eta} I_N + \widetilde{\tau} \beta$ and $\widetilde{d} = \widetilde{\eta}^2 - \widetilde{\tau}^2$ fulfil \eqref{eq_d_tilde_non_crit}.
		Then, the operator $I+ B_0(z)Vq$ is continuously invertible in the space $\mathcal{B}^r(\Sigma)$. In particular, assumption \textup{(ii)} of Proposition~\ref{prop_resolvent_convergence} is fulfilled in this case.
	\end{proposition}
	\begin{proof}
		We start by arguing that $I + B_0(z)Vq$ is a bounded operator in $\mathcal{B}^r(\Sigma)$. Due to the representation of $B_0(z)$ in \eqref{B_0_alt_rep} and the text below, $B_0(z)$ acts as a bounded operator in $\mathcal{B}^r(\Sigma)$. Moreover, $V \in W^1_\infty(\Sigma;\C^{N \times N})$ and $q \in L^{\infty}((-1,1);[0,\infty))$ imply that  $Vq$ induces a bounded operator in $\mathcal{B}^r(\Sigma)$, see Section~\ref{sec_not}~\eqref{it_Bochner}.  Hence, it remains to prove the bijectivity of $I + B_0(z)Vq$ in $\mathcal{B}^r(\Sigma)$.
		
		Let us start with the injectivity. To do so, we use the representation of $B_0(z)$ given by \eqref{B_0_alt_rep} and assume for $f \in \mathcal{B}^r(\Sigma)$
		\begin{equation}\label{eq_I+B_0_inj}
			(I +  B_0(z)Vq)f =(I +T(\alpha \cdot \nu)Vq)f + \mathfrak{I}\mathcal{C}_zV \mathfrak{I}^*qf = 0.
		\end{equation}
		By \cite[Lemma~4.2~(i)]{BHS23} $I + T(\alpha \cdot \nu)Vq$ is continuously invertible in in the space $\mathcal{B}^r(\Sigma)$ if $\cos\bigl(\tfrac{(\alpha\cdot\nu)V}{2}\bigr)^{-1} \in W^1_\infty(\Sigma;\C^{N \times N})$. Note that the structure of $V = \eta I_N + \tau \beta$ and the rules for the Dirac matrices from \eqref{eq_dirac_matrices} imply $((\alpha\cdot\nu)V)^2 = (\eta^2-\tau^2)I_N =d I_N$ and therefore
		\begin{equation*}
			\cos\bigl(\tfrac{(\alpha \cdot \nu) V}{2}\bigr) = \sum_{j=0}^\infty (-1)^{j}\frac{((\alpha \cdot \nu) V/2)^{2j}}{(2j)!} =  \sum_{j=0}^\infty (-1)^{j}\frac{(\sqrt{d}/2)^{2j}}{(2j)!} I_N = \cos\bigl(\tfrac{\sqrt{d}}{2}\bigr) I_N.
		\end{equation*}
		Consequently, $\cos\bigl(\tfrac{(\alpha\cdot\nu)V}{2}\bigr)^{-1} \in W^1_\infty(\Sigma;\C^{N \times N})$ since \eqref{eq_scaling_cond} is satisfied. Hence, we can apply $(I + T(\alpha \cdot \nu)Vq)^{-1}$ to \eqref{eq_I+B_0_inj}. This yields
		\begin{equation*}
			f + (I + T(\alpha \cdot \nu)Vq)^{-1}\mathfrak{I}\mathcal{C}_zV \mathfrak{I}^*qf =0.
		\end{equation*}
		Using \cite[Lemma 4.2~(ii)]{BHS23} and defining $Q(t) = \int_{-1}^t q(s) \,ds - \tfrac{1}{2}$, $t \in [-1,1]$,  gives us 
		\begin{equation*}
			(I+T(\alpha \cdot \nu)Vq)^{-1}\mathfrak{I} = \cos\bigl(\tfrac{(\alpha \cdot \nu)V}{2}\bigr)^{-1}\exp(-i(\alpha \cdot \nu) V Q) \mathfrak{I}
		\end{equation*}
		and therefore
		\begin{equation}\label{eq_I+B_0_invert_1}
			f + \cos\bigl(\tfrac{(\alpha \cdot \nu)V}{2}\bigr)^{-1}\exp(-i(\alpha \cdot \nu) V Q)\mathfrak{I} \mathcal{C}_zV \mathfrak{I}^*qf =0.
		\end{equation}
		By applying $\mathfrak{I}^*q$ and \cite[Lemma~4.3~(ii)]{BHS23} we obtain
		\begin{equation}\label{eq_I+B_0_invert_2}
			\begin{aligned}
				&\mathfrak{I}^*q f + \mathfrak{I}^*q\cos\bigl(\tfrac{(\alpha \cdot \nu)V}{2}\bigr)^{-1}\exp(-i(\alpha \cdot \nu) V Q)\mathfrak{I}\mathcal{C}_zV \mathfrak{I}^*qf\\
				&\hspace{160 pt}= (I+S \mathcal{C}_z V)\mathfrak{I}^*q f =0
			\end{aligned}
		\end{equation}
		with $S= \textup{sinc}\bigl(\tfrac{(\alpha\cdot\nu)V}{2}\bigr)\cos\bigl(\tfrac{(\alpha\cdot\nu)V}{2}\bigr)^{-1}$. Similar as we showed $\cos\bigl(\tfrac{(\alpha\cdot\nu)V}{2}\bigr) = \cos\bigl(\tfrac{\sqrt{d}}{2}\bigr)I_N$ one can show $\textup{sinc}\bigl(\tfrac{(\alpha\cdot\nu)V}{2}\bigr) = \textup{sinc}\bigl(\tfrac{\sqrt{d}}{2}\bigr) I_N$ and therefore 
		\begin{equation} \label{equation_rescaling}
			V S = V \textup{tanc}\bigl(\tfrac{\sqrt{d}}{2}\bigr) = \widetilde{V}.
		\end{equation}  
		Moreover, Lemma~\ref{lem_I+C_z_inverse} shows that $-1 \in \rho(\mathcal{C}_z\widetilde{V})$. By \cite[Proposition~2.1.8]{P94} there  holds $\rho(\mathcal{C}_z\widetilde{V}) \setminus  \{0\} = \rho(S\mathcal{C}_zV) \setminus  \{0\}$. Consequently, $-1 \in \rho(S\mathcal{C}_zV)$ and hence \eqref{eq_I+B_0_invert_2} implies $\mathfrak{J}^*qf=0$. This, in turn, implies  according to \eqref{eq_I+B_0_invert_1}  $f =0$.
		
		Next, we show the surjectivity of the operator $I + B_0(z)Vq$. Let $g \in \mathcal{B}^r(\Sigma)$. We set $f_g = (I +T(\alpha \cdot \nu)Vq)^{-1}(g +  \mathfrak{I}\psi)$, where 
		\begin{equation*}
			\psi =-(I+ \mathcal{C}_z\widetilde{V})^{-1}  \mathcal{C}_z\mathfrak{I}^* Vq(I +T(\alpha \cdot \nu)Vq)^{-1}g.
		\end{equation*}
		Item (iii) from \cite[Lemma 4.3]{BHS23}  yields together with~\eqref{equation_rescaling}
		\begin{equation*}
			(I +   B_0(z) Vq)(I +T(\alpha \cdot \nu)Vq)^{-1}\mathfrak{I} = \mathfrak{I}(I+ \mathcal{C}_z\widetilde{V}).
		\end{equation*}
		Thus, by  applying  $I +   B_0(z) Vq$ to $f_g$ and using $B_0(z) = T(\alpha\cdot\nu) + \mathfrak{I}\mathcal{C}_z \mathfrak{I}^*$, see \eqref{B_0_alt_rep},  we obtain
		\begin{equation*}
			\begin{aligned}
				(I +   B_0(z) Vq)f_g &= (I +   B_0(z) Vq)(I +T(\alpha \cdot \nu)Vq)^{-1}g \\ 
				&\qquad+ (I +   B_0(z) Vq)(I +T(\alpha \cdot \nu)Vq)^{-1}\mathfrak{I}\psi\\
				&= (I +T(\alpha \cdot \nu)Vq + \mathfrak{I}\mathcal{C}_z\mathfrak{I}^*Vq)(I +T(\alpha \cdot \nu)Vq)^{-1}g \\
				&\qquad+ \mathfrak{I}(I+ \mathcal{C}_z\widetilde{V}) \psi\\
				&= g+ \mathfrak{I}\mathcal{C}_z\mathfrak{I}^*Vq(I +T(\alpha \cdot \nu)Vq)^{-1}g + \mathfrak{I}(I+ \mathcal{C}_z\widetilde{V}) \psi\\
				&=g,
			\end{aligned}
		\end{equation*}
		which completes the proof.
	\end{proof}
	
	\subsection{Analysis of $I + B_\varepsilon(z)Vq$}\label{sec_I+B_eps}
	In this section we show  that if 
	\begin{equation*}
		\sup_{x_\Sigma \in \Sigma} d(x_\Sigma) < \frac{\pi^2}{4}, \qquad d = \eta^2-\tau^2,
	\end{equation*}  then (i) of Proposition~\ref{prop_resolvent_convergence} is fulfilled, i.e. $(I+B_\varepsilon(z)Vq)^{-1}$ is uniformly bounded in $\mathcal{B}^0(\Sigma)$. Combining this result with  Proposition~\ref{prop_resolvent_convergence} and Proposition~\ref{prop_inv_I+BVq} proves the norm resolvent convergence of $H_{V_\varepsilon}$. To prove the uniform boundedness of  $(I+B_\varepsilon(z)Vq)^{-1}$ a  careful and deep analysis of this operator is necessary. We do this by studying $(I+B_\varepsilon(z)Vq)^{-1}$ in the case that $\Sigma$ is a hyperplane and $\eta$ and $\tau$ are constant in detail in Section~\ref{sec_hyperplane}. Then, we use a parameter dependent partition of unity in Section~\ref{sec_loc_part} to transfer the results to  the case where 
	$\Sigma = \Sigma_{\zeta,\kappa}$,  $\zeta \in C^2_b(\R^{\theta-1};\R)$, $\kappa \in \textup{SO}(\theta)$,
	 is a rotated $C^2_b$-graph as in \eqref{eq_rotated_graph} and $\eta,\tau \in C^1_b(\Sigma;\R)$. 
	We start by introducing for $\varepsilon \in (0,\varepsilon_2)$ with $\varepsilon_2$ as in Proposition \ref{prop_conv_res} the auxiliary operator
	\begin{equation}\label{def_B_eps_bar}
		\begin{aligned}
			\overline{B}_\varepsilon(z) &: \mathcal{B}^0(\Sigma) \to \mathcal{B}^0(\Sigma),\\
			\overline{B}_\varepsilon( z)f(t)(x_\Sigma) &:=   \int_{-1}^1 \int_{\Sigma} G_ z(x_\Sigma +\varepsilon (t-s)\nu(x_\Sigma) -y_\Sigma )   f(s)(y_\Sigma) \, d\sigma(y_\Sigma) \, ds;
		\end{aligned}
	\end{equation} 
	cf. \cite[eq. (3.17), eq. (3.23) and Appendix~B]{BHS23}. According to \cite[eq. (3.29)]{BHS23}, 
	\begin{equation}\label{eq_B_diff}
		\norm{\overline{B}_\varepsilon(z)  -	B_\varepsilon(z)}_{0 \to 0} \leq C \varepsilon^{1/2} (1+\abs{\log(\varepsilon)})^{1/2}, \qquad  \varepsilon \in (0,\varepsilon_2).
	\end{equation}
	 Next, we transfer this operator to $\mathcal{B}^0(\R^{\theta-1})$. 
	To do so, we introduce the isomorphism
	\begin{equation}\label{def_iota}
		\iota_{\zeta,\kappa} :L^2(\Sigma;\C^N) \to L^2(\R^{\theta-1};\C^N), \quad	(\iota_{\zeta,\kappa} f)(x') := f\big(\kappa (x',\zeta(x'))\big).
	\end{equation}
	By transforming the integral on $\Sigma$ to an integral on $\R^{\theta-1}$ one obtains that the norm of $\iota_{\zeta,\kappa}$ and its inverse are given by
	\begin{equation}\label{eq_norm_iota}
		\begin{aligned}
			\|\iota_{\zeta,\kappa}\|_{L^2(\Sigma;\C^N) \to L^2(\R^{\theta-1};\C^N)} = \sup_{x' \in \R^{\theta-1}} (1+\abs{\nabla\zeta(x')}^2)^{-1/4} \\
			\textup{ and } 	\big\|\iota_{\zeta,\kappa}^{-1}\big\|_{ L^2(\R^{\theta-1};\C^N) \to L^2(\Sigma;\C^N)}  =\sup_{x' \in \R^{\theta-1}} (1+\abs{\nabla\zeta(x')}^2)^{1/4}.
		\end{aligned}
	\end{equation}
	Note that the definition of $H^r(\Sigma;\C^N)$, $r \in [0,2]$, see \cite{M00} or \cite[Section~2.1]{BHS23}, implies that $\iota_{\zeta,\kappa}$ also acts  as an isomorphic operator from  $H^r(\Sigma;\C^N)$ to $H^r(\R^{\theta-1};\C^N)$ for $r \in [0,2]$. Recall  that in this case $\iota_{\zeta,\kappa}$ can also be viewed as a bounded operator in $\mathcal{B}^r(\Sigma)$ to $\mathcal{B}^r(\R^{\theta-1})$ which has the same norm as the operator acting from $H^r(\Sigma;\C^N)$ to $H^r(\R^{\theta-1};\C^N)$.

	Next, we introduce for $\varepsilon \in (0,\varepsilon_2)$ the operators
	\begin{equation}\label{def_D_operators}
		\begin{aligned}
			D_\varepsilon^{\zeta,\kappa}(z)&:=  \iota_{\zeta,\kappa}\overline{B}_\varepsilon({z})  \iota_{\zeta,\kappa}^{-1} : \mathcal{B}^0(\R^{\theta-1}) \to \mathcal{B}^0(\R^{\theta-1})\\
			\text{and }D_0^{\zeta,\kappa}(z)&:=  \iota_{\zeta,\kappa}B_0({z})  \iota_{\zeta,\kappa}^{-1}\: : \mathcal{B}^0(\R^{\theta-1}) \to\mathcal{B}^0(\R^{\theta-1}).
		\end{aligned}
	\end{equation}
	The results from Proposition~\ref{prop_conv_res} and \eqref{eq_B_diff} imply that $D_\varepsilon^{\zeta,\kappa}(z)$ is uniformly bounded in $\mathcal{B}^0(\R^{\theta-1})$ with respect to $\varepsilon \in (0,\varepsilon_2)$ and for $r \in (0,\frac{1}{2})$  the inequality
	\begin{equation}\label{eq_D_eps_conv}
		\begin{aligned}
		\|D_0^{\zeta,\kappa}(z) - D_\varepsilon^{\zeta,\kappa}(z)\|_{1/2 \to 0} & = 
		\|\iota_{\zeta,\kappa}\bigl(B_0({z})  - \overline{B}_\varepsilon({z})\bigr)   \iota_{\zeta,\kappa}^{-1}\|_{1/2 \to 0}\\
		&\leq C \bigl(\|B_0({z})  -  B_\varepsilon({z})\|_{1/2 \to 0} \\ & \qquad+\|B_\varepsilon({z})  -  \overline{B}_\varepsilon({z})\|_{1/2 \to 0} \bigr) \\
		&\leq C \bigl(\|B_0({z})  -  B_\varepsilon({z})\|_{1/2 \to 0} \\ & \qquad+\|B_\varepsilon({z})  -  \overline{B}_\varepsilon({z})\|_{0 \to 0} \bigr)\\
		 &\leq   C \varepsilon^{1/2-r}, \quad \varepsilon \in (0, \varepsilon_2),
		\end{aligned}
	\end{equation}
	holds.
	In particular, $D_\varepsilon^{\zeta,\kappa}(z)f$ converges for $\varepsilon \to 0$ to $D_0^{\zeta,\kappa}(z)f$ in $\mathcal{B}^0(\R^{\theta-1})$ for $f \in \mathcal{B}^{1/2}(\R^{\theta-1})$.  Furthermore, $\mathcal{B}^{1/2}(\R^{\theta-1})$ is by \cite[Lemma~1.2.19]{HNVW16} and \cite[the lines above eq.~(3.22)]{M00} a dense subset of  $\mathcal{B}^{0}(\R^{\theta-1})$. Combining these considerations with the uniform boundedness of $D_\varepsilon^{\zeta,\kappa}(z)$ in $\mathcal{B}^0(\R^{\theta-1})$ shows that for all $f \in \mathcal{B}^0(\R^{\theta-1})$
	\begin{equation}\label{eq_D_strong_conv}
		D_\varepsilon^{\zeta,\kappa}(z)f \to D_0^{\zeta,\kappa}(z)f \quad\textup{in } \mathcal{B}^0(\R^{\theta-1}), \textup{ as } \varepsilon \to 0.
	\end{equation}
	Using \eqref{def_B_eps_bar} and \eqref{def_iota}, and setting 
	\begin{equation}\label{eqref_x_nu_zeta_kappa}
		x_{\zeta,\kappa}= \kappa(\cdot,\zeta(\cdot))\quad \textup{and} \quad \nu_{\zeta, \kappa} = \nu\circ x_{\zeta,\kappa} =\frac{\kappa (-\nabla \zeta, 1)}{\sqrt{1 + \abs{\nabla \zeta}^2}}
	\end{equation} yields for $f \in \mathcal{B}^0(\R^{\theta-1})$ and a.e. $(t,x') \in (-1,1) \times \R^{\theta-1}$
	\begin{equation}\label{eq_D_eps_int_rep_0}	
		\begin{aligned}
			D_\varepsilon^{\zeta,\kappa}({z}) f (t)(x')&= \int_{-1}^1\int_{\R^{\theta-1}} G_{z}\big(x_{\zeta,\kappa}(x') - x_{\zeta,\kappa}(y') + \varepsilon(t-s) \nu_{\zeta,\kappa}(x') \big)\\
			&\hspace{115 pt}\cdot \sqrt{1+\abs{\nabla \zeta(y')}^2}f(s)(y') \,dy' \, ds.
		\end{aligned}
	\end{equation}
	Another useful representation is given by
	\begin{equation}\label{eq_D_eps_int_rep}
		D_\varepsilon^{\zeta,\kappa}({z}) f (t) = \int_{-1}^1 d_{\varepsilon(t-s)}^{\zeta,\kappa}({z})f(s) \, ds, \quad f \in \mathcal{B}^0(\R^{\theta-1}), t \in (-1,1),
	\end{equation}	
	where the integral is considered as a Bochner integral and 
	\begin{equation}\label{eq_d_eps}
		\begin{aligned}
			d_{\widetilde{\varepsilon}}^{\zeta,\kappa}({z}) &: L^2(\R^{\theta-1};\C^N) \to L^2(\R^{\theta-1};\C^N),\\
			d_{\widetilde{\varepsilon}}^{\zeta,\kappa}({z})g(x') 
			&=\int_{\R^{\theta-1}} G_{z}\big(x_{\zeta,\kappa}(x') - x_{\zeta,\kappa}(y') + \widetilde{\varepsilon} \nu_{\zeta,\kappa}(x') \big)\\
			&\hspace{130 pt}\cdot \sqrt{1+\abs{\nabla \zeta(y')}^2}g(y') \,dy',
		\end{aligned}
	\end{equation}
	for $\widetilde{\varepsilon} \in (-2\varepsilon_2,2\varepsilon_2)\setminus\{0\}$. For the interaction strengths  $\eta,\tau \in C^1_b(\Sigma;\R)$ we also define the  matrix valued function
	\begin{equation}\label{eq_Q_matrix}
		Q^{\zeta,\kappa}_{\eta,\tau}:= V \circ x_{\zeta,\kappa} = \eta \circ x_{\zeta,\kappa} I_N + \tau \circ x_{\zeta,\kappa}\beta.
	\end{equation} 
	There  holds $Q^{\zeta,\kappa}_{\eta,\tau} = \iota_{\zeta,\kappa} V\iota_{\zeta,\kappa}^{-1}$ in the sense of operators in $L^2(\R^{\theta-1};\C^N)$.

	\subsubsection{Hyperplanes and constant interaction strengths}\label{sec_hyperplane}
	
	In this section we assume that $\zeta$ is a constant function having the value $y_0 \in \R$, i.e. $\Sigma =  \Sigma_{\zeta,\kappa} = \Sigma_{y_0,\kappa}$ is an  affine $(\theta-1)$-dimensional hyperplane in $\R^{\theta}$. We also assume in this section that the interaction strengths are constant and given by $\eta,\tau \in \R$. This implies that  $Q^{y_0,\kappa}_{\eta,\tau}$ is equal to the constant matrix 
	\begin{equation}\label{eq_Q_const}
		Q_{\eta,\tau} := \eta I_N + \tau \beta 
	\end{equation}
	 in this case. The main goal of this section is to show that for every compact set $S \subset \R^2$ satisfying
	\begin{equation*}
		\max_{(\eta,\tau) \in S} \eta^2- \tau^2 < \frac{\pi^2}{4}
	\end{equation*} 
	there exists a $\delta_2 = \delta_2(S)$ such that  
	\begin{equation*}
		\|(I+ D^{y_0,\kappa}_\varepsilon(z)Q_{\eta,\tau})^{-1}\|_{0 \to 0}
	\end{equation*} 
	is uniformly bounded with respect to $(\varepsilon,y_0,(\eta,\tau),\kappa) \in (0,\delta_2)\times \R \times S \times\textup{SO}(\theta)$; cf. Corollary~\ref{cor_uniform_bdd_D}. This result will play a major role when we prove the uniform boundedness of $(I + B_\varepsilon(z)Vq)^{-1}$ in Section~\ref{sec_loc_part}. 
	To do so we proceed as follows: We start by using the Fourier transform to transform $D^{y_0,\kappa}_\varepsilon(z)$ into a decomposable direct integral operator with frequency  dependent fiber operators, see the considerations up to \eqref{eq_def_frak_d}. Then, up to Lemma~\ref{lem_frak_h_inv_estimate} we find and analyse suitable approximations for the fiber operators for high and low frequencies. Finally, we use these results to  prove the main statements (Proposition~\ref{prop_uniform_bdd_D} and Corollary~\ref{cor_uniform_bdd_D})  of this section. 
	
	Before we start, let us fix some notations. In the present setting the normal vector $\nu$ is constant and given by  $\kappa e_\theta$, where $e_\theta$ is the $\theta$-th euclidean unit vector. Let us also mention that $\varepsilon_2$ from  Proposition~\ref{prop_res_formula_eps} is according to \cite[eq. (3.11)]{BHS23} given by
	\begin{equation*}
		\varepsilon_2 = \min \biggl\{\frac{\varepsilon_1}{2}, \frac{1}{2 \norm{D \widetilde{\nu}}_{L^\infty(\R^{\theta};\R^{\theta \times \theta})}} \biggr\}.
	\end{equation*}
	where, $\varepsilon_1$ is chosen as below \eqref{eq_iota}, see also \cite[Proposition~2.4]{BHS23} and $\widetilde{\nu}$ is an extension $C^1_b$ of $\nu$ to $\R^{\theta}$; cf. \cite[the discussion above eq. (3.5)]{BHS23}. Next, we argue that $\varepsilon_2$ can be set independently of $\kappa$ to $\infty$. First, $\varepsilon_1$ is chosen  such that one can identify $\Omega_{\varepsilon_1}$ via the map $\iota$ in~\eqref{eq_iota} with $\Sigma \times (-\varepsilon_1,\varepsilon_1)$ and  the eigenvalues of $ \varepsilon_1 W(x_\Sigma)$, $x_\Sigma \in \Sigma$, are sufficiently small, so that this identification is bijective. However, in the current case $\Omega_{\varepsilon_1} = \kappa(\R^{\theta-1} \times (y_0-\varepsilon_1, y_0 + \varepsilon_1))$ can be identified with $\Sigma \times (-\varepsilon_1,\varepsilon_1) = \kappa(\R^{\theta-1} \times \{y_0\}) \times (-\varepsilon_1,\varepsilon_1)$ for arbitrary $\varepsilon_1>0$ and the Weingarten map is zero. Thus, we can set $\varepsilon_1$   independently of $y_0$ and $\kappa$ to $\infty$. Furthermore, as $\nu$ is constant, its constant extension is a $C^1_b$-extension of $\nu$. Hence,  $(2 \norm{D \widetilde{\nu}}_{L^\infty(\R^{\theta};\R^{\theta \times \theta})})^{-1} = \infty$ and therefore $\varepsilon_2 = \infty$. 
	Using \eqref{eq_D_eps_int_rep_0} and 
	\begin{equation*}
		x_{y_0,\kappa}(x')-x_{y_0,\kappa}(y') + \varepsilon (t-s) \nu_{y_0,\kappa}(x') = \kappa(x'-y',\varepsilon(t-s)), \,  x',y' \in \R^{\theta-1}, t,s \in (-1,1),
	\end{equation*}  
shows that $D_\varepsilon^{y_0,\kappa}(z) $  has for $\varepsilon \in (0,\infty)$ and $f \in \mathcal{B}^0(\R^{\theta-1})$ the representation
	\begin{equation}\label{eq_D_eps_straight}
		D^{y_0,\kappa}_{\varepsilon}(z)f(t)(x')  =
		\int_{-1}^1\int_{\R^{\theta-1}} G_{z}\big(\kappa(x'-y',\varepsilon (t-s))\big) f(s)(y') \,dy' \,ds
	\end{equation}
	for a.e. $(x',t) \in \R^{\theta-1}  \times (-1,1) $, which proves  that $D_\varepsilon^{y_0,\kappa}(z)$ is independent of $y_0 \in \R$. Furthermore,  \eqref{eq_D_strong_conv} implies that also $D_0^{y_0,\kappa}(z)$ is independent of $y_0$. Thus, w.l.o.g. we can set $y_0=0$. Before we state the next result, we define for convenience the matrices $\widetilde{\alpha}_j := \alpha \cdot \kappa e_j$, $j \in \{1,\dots, \theta\}$, $\widetilde{\alpha} \cdot \xi := \sum_{j=1}^{\theta} \widetilde{\alpha}_j \xi_j$, $\xi \in \R^{\theta}$, and $\widetilde{\alpha}' \cdot \xi' = \sum_{j=1}^{\theta-1} \widetilde{\alpha}_j \xi_{j}'$, $\xi'\in\R^{\theta-1}$. Note that $\widetilde{\alpha}_1, \dots, \widetilde{\alpha}_\theta$ satisfy the same anti-commutation relations as $\alpha_1, \dots, \alpha_\theta$; cf. \eqref{it_Dirac_matrices} of Section~\ref{sec_not}.
	
	We start by calculating the Fourier transform of the function $G_z(\kappa(\cdot,\widetilde{\varepsilon}))$ for fixed $\widetilde{\varepsilon} \neq 0$; cf. \cite[eqs. (44)--(45)]{R22a} for similar considerations. Recall that $\mathcal{F}$ is the $(\theta-1)$-dimensional Fourier transform defined in Section~\ref{sec_not}~\eqref{it_Fourier}.

	\begin{lemma}\label{lem_Fourier_transform_G_z_eps} 
		Let  $z \in \rho(H)$, $G_z$ be the integral kernel of $(H-z)^{-1}$ given by \eqref{eq_G_z_2D}--\eqref{eq_G_z_3D}, and $\widetilde{\varepsilon}\neq0$. Then, 
		\begin{equation*}
			\mathcal{F} G_z(\kappa(\cdot,\widetilde{\varepsilon})) = \bigg(\frac{\widetilde{\alpha}' \cdot(\cdot) + m \beta + zI_N)}{\sqrt{z^2-m^2-\abs{\cdot}^2}} + \widetilde{\alpha}_{\theta}\sign(\widetilde{\varepsilon})\bigg)\frac{ie^{\abs{\widetilde{\varepsilon}}i\sqrt{z^2-m^2-\abs{\cdot}^2}}}{ 2 \sqrt{(2\pi)^{\theta-1}}}.
		\end{equation*}
	\end{lemma}
	\begin{proof}
		Let $\mathcal{F}$, $\mathcal{F}_1$, $\mathcal{F}_2$ and $\mathcal{F}_{1,2}$ be  defined as in Section~\ref{sec_not}~\eqref{it_Fourier}. First, we consider $\mathcal{F}_1 G_z(\kappa(\cdot))$.  Since $G_z(\kappa(\cdot)) \in L^1(\R^\theta;\C^{N\times N}) \subset \mathcal{S}'(\R^{\theta};\C^{N\times N})$, see \eqref{eq_G_z_2D}--\eqref{eq_G_z_3D}, the expression $\mathcal{F}_1 G_z(\kappa(\cdot))$ is well-defined in $\mathcal{S}'(\R^{\theta};\C^{N\times N})$. Moreover,  $\mathcal{F}_1 G_z(\kappa(\cdot)) = \mathcal{F}_{2}^{-1}\mathcal{F}_{1,2} G_z (\kappa(\cdot))$. Thus, we calculate $\mathcal{F}_{1,2}G_z(\kappa(\cdot))$ next. The function $G_z$ satisfies the equation $(-i(\alpha \cdot \nabla )+m\beta- zI_N)G_z =  \delta I_N $, with $\delta$ denoting the $\delta$ distribution supported in $\{0\}$. Hence, the standard rules for the Fourier transform, see \cite[Chapter IX]{RS75}, show 
		\begin{equation*}
			\big(\alpha \cdot (\cdot)+ m\beta - zI_N\big)\mathcal{F}_{1,2}G_z= \frac{1}{\sqrt{(2\pi)^{\theta}}}I_N \quad \textup{in }\mathcal{S}'(\R^\theta;\C^{N\times N}). 
		\end{equation*}
		Furthermore, since $G_z \in L^1(\R^{\theta};\C^{N\times N})$, we have by the Riemann-Lebesgue lemma $\mathcal{F}_{1,2}G_z \in C_0(\R^\theta;\C^{N\times N})$. Using the properties of $\alpha_j$, $j \in \{1, \dots ,\theta\}$, and $\beta$ yields $(\alpha \cdot \xi+ m\beta - zI_N)^{-1} = \frac{\alpha \cdot \xi+ m\beta + zI_N}{\abs{\xi}^2 + m^2 - z^2}$ for $\xi \in \R^{\theta}$. Thus, 
		\begin{equation*}
			(\mathcal{F}_{1,2}G_z)(\kappa \xi)= \frac{\alpha \cdot (\kappa \xi) + m\beta + zI_N}{ (\abs{\kappa \xi}^2 + m^2 - z^2) \sqrt{(2\pi)^{\theta}}}, \quad  \xi \in \R^\theta.
		\end{equation*}
		Since $\kappa \in \textup{SO}(\theta)$, we obtain
		\begin{equation}\label{eq_G_z_Fourier_Transform_1}
			\mathcal{F}_{1,2}G_z(\kappa(\cdot))(\xi) = (\mathcal{F}_{1,2}G_z)(\kappa \xi)= \frac{\widetilde{\alpha} \cdot  \xi + m\beta + zI_N}{ (\abs{\xi}^2 + m^2 - z^2) \sqrt{(2\pi)^{\theta}} }, \quad \xi \in \R^{\theta }.
		\end{equation}
		Next, we determine $\mathcal{F}_2^{-1}\mathcal{F}_{1,2}G_z$. We claim that for a.e. $(\xi',x_\theta) \in \R^\theta$ the equation
		\begin{equation}\label{eq_G_z_Fourier_Transform_2}
			\begin{aligned}
				(\mathcal{F}_2^{-1}\mathcal{F}_{1,2}&G_z(\kappa(\cdot)))(\xi',x_\theta)\\
				&=\bigg(\frac{\widetilde{\alpha}'\cdot \xi'+ m \beta + zI_N}{\sqrt{z^2-m^2-|\xi'|^2}} + \widetilde{\alpha}_\theta\sign(x_{\theta})\bigg) \cdot\frac{ie^{\abs{x_\theta}i\sqrt{z^2-m^2-\abs{\xi'}^2}}}{2\sqrt{(2\pi)^{\theta-1}}}
			\end{aligned}
		\end{equation} 
		holds.
		We verify \eqref{eq_G_z_Fourier_Transform_2} by applying $\mathcal{F}_2$ and comparing the result with  \eqref{eq_G_z_Fourier_Transform_1}. Since the right hand side of \eqref{eq_G_z_Fourier_Transform_2} decays exponentially for $|x_\theta|\to \infty$, we can use the integral representation of the Fourier transform. Hence, simple integration gives us for $\xi= (\xi',\xi_\theta) \in \R^\theta$
		\begin{equation*}
			\begin{aligned}
				&\frac{1}{\sqrt{2\pi}}\int_\R \bigg(\frac{\widetilde{\alpha}'\cdot \xi'+ m \beta + zI_N}{\sqrt{z^2-m^2-|\xi'|^2}} + \widetilde{\alpha}_\theta\sign(x_\theta)\bigg) \\
				&\hspace{55 pt}\cdot\frac{ie^{i(-x_\theta \xi_\theta +\abs{x_\theta}\sqrt{z^2-m^2-\abs{\xi'}^2})}}{2 \sqrt{(2\pi)^{\theta-1}}} dx_\theta
				= \frac{\widetilde{\alpha} \cdot  \xi + m\beta + zI_N}{ (\abs{\xi}^2 + m^2 - z^2) \sqrt{(2 \pi)^{\theta}}},
			\end{aligned}
		\end{equation*}
		which verifies \eqref{eq_G_z_Fourier_Transform_2}. 
		Therefore, $\mathcal{F}_1 G_z(\kappa(\cdot))= \mathcal{F}_2^{-1}\mathcal{F}_{1,2}G_z(\kappa(\cdot))$ can be represented by the function 
		\begin{equation*}
			\begin{aligned}
				\R^{\theta} \ni	(\xi',x_\theta) \mapsto \bigg(\frac{\widetilde{\alpha}'\cdot \xi'+ m \beta + zI_N}{\sqrt{z^2-m^2-|\xi'|^2}} + \widetilde{\alpha}_\theta\sign(x_\theta)\bigg)
				\frac{ie^{|x_\theta|i\sqrt{z^2-m^2-|\xi'|^2}}}{2\sqrt{(2\pi)^{\theta-1}}}.
			\end{aligned}
		\end{equation*}
		Furthermore, since $G_z(\kappa(\cdot,\widetilde{\varepsilon}))\in L^1(\R^{\theta-1};\C^{N\times N})$ for $\widetilde{\varepsilon} \neq 0$, this shows that for $\widetilde{\varepsilon} \neq 0$ and $\xi' \in \R^{\theta-1}$ there holds
		\begin{equation*}
			\begin{aligned}
				\mathcal{F}G_z(\kappa(\cdot,\widetilde{\varepsilon})) (\xi') &=\frac{1}{\sqrt{(2\pi)^{\theta-1}}}\int_{\R^{\theta-1}} G_z(\kappa(x',\widetilde{\varepsilon}))e^{-i \langle x', \xi' \rangle} \, dx'  \\
				&= \mathcal{F}_1G_z(\kappa(\cdot))(\xi',\widetilde{\varepsilon}) \\
				&= \bigg(\frac{\widetilde{\alpha}'\cdot\xi' + m \beta + zI_N}{\sqrt{z^2-m^2-\abs{\xi'}^2}} + \widetilde{\alpha}_\theta\sign(\widetilde{\varepsilon})\bigg)\frac{ie^{\abs{\widetilde{\varepsilon}}i\sqrt{z^2-m^2-\abs{\xi'}^2}}}{2\sqrt{(2\pi)^{\theta-1}}}.
			\end{aligned}
		\end{equation*}
	\end{proof}
	
	\begin{proposition}\label{prop_Fourier_D}
		Let  $z \in \rho(H)$, $\varepsilon >0$, and $\mathcal{F}$ be the $(\theta-1)$-dimensional Fourier transform defined in Section~\ref{sec_not}~\eqref{it_Fourier}. Then, there holds for $f \in \mathcal{B}^0(\R^{\theta-1})$
		\begin{equation}\label{eq_Fourier_Transform_eps}
			\begin{aligned}
				&\mathcal{F}D_\varepsilon^{0,\kappa}(z)\mathcal{F}^{-1}f(t)(\xi') 
				= \int_{-1}^1 \bigg(\frac{\widetilde{\alpha}'\cdot \xi'   + m \beta + zI_N}{\sqrt{z^2-m^2-\abs{\xi'}^2}} + \widetilde{\alpha}_\theta\sign(t-s)\bigg) \\
				&\hspace{ 150 pt}\cdot\frac{ie^{\abs{\varepsilon(t-s)}i\sqrt{z^2-m^2-\abs{\xi'}^2 }}}{2}f(s)(\xi') \, ds
			\end{aligned}
		\end{equation}
		and 
		\begin{equation}\label{eq_Fourier_Transform_0}
			\begin{aligned} \mathcal{F}D_0^{0,\kappa}(z)\mathcal{F}^{-1}f(t)(\xi') &= \int_{-1}^1 \bigg(\frac{\widetilde{\alpha}'\cdot \xi'   + m \beta + zI_N}{\sqrt{z^2 - m^2 - \abs{\xi'}^2}} + 		\widetilde{\alpha}_\theta\sign(t-s)\bigg)\\
				&\hspace{150 pt}\cdot\frac{if(s)(\xi')}{2} \, ds
			\end{aligned}
		\end{equation}
		for a.e $(t,\xi') \in (-1,1) \times  \R^{\theta-1}$. In particular, the operators $D_\varepsilon^{0,\kappa}(z)$ and $D_0^{0,\kappa}(z)$ depend continuously on $\kappa \in \textup{SO}(\theta)$ with respect to the operator norm.
	\end{proposition}
	
	\begin{proof}
		We start with the case  $\varepsilon >0$. Equation \eqref{eq_D_eps_straight} shows 
		\begin{equation*}
			D_{\varepsilon}^{0,\kappa}(z)f(t) = \int_{-1}^1G_z\bigl(\kappa(\cdot,\varepsilon(t-s))\bigr)\ast f(s) \,ds
		\end{equation*} 
		for a.e. $t \in (-1,1)$ and all  $f \in \mathcal{B}^0(\R^{\theta-1})$. Thus, Lemma \ref{lem_Fourier_transform_G_z_eps} and \cite[Theorem~IX.4]{RS75}  prove the statement for $\varepsilon>0$. It remains to consider the operator $D_0^{0,\kappa}(z)$. We start by defining 
		\begin{equation*}
			\begin{aligned}
				\widetilde{D}_0^{0,\kappa}(z) &:  \mathcal{B}^0(\R^{\theta-1}) \to \mathcal{B}^0(\R^{\theta-1})\\
				\widetilde{D}_0^{0,\kappa}(z)f(t)(\xi') &:= \int_{-1}^1 \bigg(\frac{\widetilde{\alpha}'\cdot \xi'   + m \beta + zI_N}{\sqrt{z^2-m^2-\abs{\xi'}^2}} + 		\widetilde{\alpha}_\theta\sign(t-s)\bigg)\frac{i f(s)(\xi')}{2} \, ds.
			\end{aligned}
		\end{equation*}
		Next, let  $f \in \mathcal{B}^{0}(\R^{\theta-1})$. Using \eqref{eq_Fourier_Transform_eps}  and dominated convergence, see \cite[Proposition 1.2.5]{HNVW16}, one obtains that $\mathcal{F} D_\varepsilon^{0,\kappa}(z) \mathcal{F}^{-1}f$ converges for $\varepsilon \to 0$ to $\widetilde{D}_0^{0,\kappa}(z) f$ in $\mathcal{B}^0(\R^{\theta-1})$.  Thus, the boundedness of $\mathcal{F}$ and $\mathcal{F}^{-1}$ in $L^2(\R^{\theta-1};\C^N)$ (and therefore also in $\mathcal{B}^0(\R^{\theta-1})$; cf. Section~\ref{sec_not}~\eqref{it_Bochner}) implies that $D_\varepsilon^{0,\kappa}(z)f$ converges to $\mathcal{F}^{-1}\widetilde{D}_0^{0,\kappa}(z)\mathcal{F}f$ in $\mathcal{B}^0(\R^{\theta-1})$. Moreover, by \eqref{eq_D_strong_conv}  $D_\varepsilon^{0,\kappa}(z) f$ converges to $D_0^{0,\kappa}(z)f$ in $ \mathcal{B}^0(\R^{\theta-1})$. Hence, $D_0^{0,\kappa}(z) f = \mathcal{F}^{-1}\widetilde{D}_0^{0,\kappa}(z)\mathcal{F}f$ which proves \eqref{eq_Fourier_Transform_0}. 
		
		It remains to verify the continuous dependence on $\kappa$. Taking the definition of $\widetilde{\alpha}_1, \dots, \widetilde{\alpha}_\theta$, see the text before Lemma~\ref{lem_Fourier_transform_G_z_eps}, and $D_\varepsilon^{0,\kappa}(z)$ and $D_0^{0,\kappa}(z)$ in \eqref{eq_Fourier_Transform_eps} and \eqref{eq_Fourier_Transform_0}, respectively, into account, one sees that all terms that depend on $\kappa$ can be taken out of the integrals. This implies the claimed continuity.
	\end{proof}
	The structure of $\mathcal{F} D_\varepsilon^{0,\kappa}(z) \mathcal{F}^{-1}$ and $\mathcal{F} D_0^{0,\kappa}(z) \mathcal{F}^{-1}$ inspires us to change our viewpoint. Namely, instead of viewing theses operators in $\mathcal{B}^0(\R^{\theta-1})$ we consider them as operators in the isometrically isomorphic space $L^2(\R^{\theta-1};L^2((-1,1);\C^N))$. In \cite[Chapter~XIII.16]{RS77} and generally in the context of direct integrals the notation $\int_{\R^{\theta-1}}^\oplus L^2((-1,1);\C^N) \, d\xi'$ for this space is also common.
	Considered as operators acting in this space  $\mathcal{F} D_\varepsilon^{0,\kappa}(z) \mathcal{F}^{-1}$ and $\mathcal{F} D_0^{0,\kappa}(z) \mathcal{F}^{-1}$ are decomposable direct integral operators with fibers which are defined for 
	 $\varepsilon \geq  0$, $\xi' \in \R^{\theta-1}$, and $z \in \rho(H)$ by
		\begin{equation}\label{eq_def_frak_d}
			\begin{aligned}
				\mathfrak{D}_{\varepsilon,\xi'}({z})&: L^2((-1,1);\C^N) \to  L^2((-1,1);\C^N),\\
				\mathfrak{D}_{\varepsilon,\xi'}({z})f(t)&:=\int_{-1}^1 \bigg(\frac{\widetilde{\alpha}'\cdot \xi'   + m \beta + zI_N}{\sqrt{z^2-m^2-\abs{\xi'}^2}} + \widetilde{\alpha}_\theta\sign(t-s)\bigg)\\
				&\hspace{100 pt}\cdot\frac{ie^{\abs{\varepsilon(t-s)}i\sqrt{z^2-m^2-\abs{\xi'}^2}}}{2}f(s) \,ds.
			\end{aligned}
		\end{equation}
	\begin{remark}
		These operators still depend on the rotation matrix $\kappa \in \textup{SO}(\theta)$ since $\widetilde{\alpha}_1, \dots, \widetilde{\alpha}_\theta$ depend on $\kappa$. However, since we use these operators only as auxiliary operators in this section, we omit $\kappa$.
	\end{remark}
	
	Next, we explain the ideas described before \eqref{eq_def_frak_d} in a more rigorous way. Using 
	\begin{equation*}
		\begin{aligned}
			\mathcal{B}^0(\R^{\theta-1}) &= L^2((-1,1);L^2(\R^{\theta-1};\C^N)) \\
			&\simeq L^2((-1,1) \times \R^{\theta-1};\C^N) \simeq L^2(\R^{\theta-1};L^2((-1,1);\C^N)),
		\end{aligned}
	\end{equation*}
	see \cite[Corollary~1.2.23 and Proposition~1.2.24]{HNVW16}, allows us to define the isometric isomorphism
	\begin{equation*}
		\begin{aligned}
				\mathfrak{i}: \mathcal{B}^0(\R^{\theta-1}) &\to  L^2(\R^{\theta-1};L^2((-1,1);\C^N)),\\
			\mathfrak{i}  f(\xi')(t)&:= \mathcal{F}f(t)(\xi') \quad \text{for a.e. } (\xi',t) \in \R^{\theta-1}\times(-1,1). 
		\end{aligned}
	\end{equation*}
	Thus, by Proposition~\ref{prop_Fourier_D} and \eqref{eq_def_frak_d} we obtain for $\varepsilon\geq0$ that
	\begin{equation*}
			\mathfrak{i}  D^{0,\kappa}_\varepsilon({z}) \mathfrak{i}^{-1} :  L^2(\R^{\theta-1};L^2((-1,1);\C^N)) \to  L^2(\R^{\theta-1};L^2((-1,1);\C^N))
	\end{equation*}
	acts for $f \in L^2(\R^{\theta-1};L^2((-1,1);\C^N))$ and fixed $\xi' \in \R^{\theta-1}$  as 	$ \mathfrak{D}_{\varepsilon,\xi'}({z}) f(\xi')$,
	i.e.  $\mathfrak{i} D^{0,\kappa}_\varepsilon({z}) \mathfrak{i}^{-1}$ is a direct integral operator which can be decomposed in the fiber operators $\mathfrak{D}_{\varepsilon,\xi'}({z})$; cf. \cite[Chapter~XIII.16]{RS77}.

	Now, we  use the  theory of direct integrals in order to transfer results regarding $\mathfrak{D}_{\varepsilon,\xi'}({z})$ to ${D}^{0,\kappa}_\varepsilon({z})$ and vice versa.
	We formulate this in the upcoming lemma which follows from \cite[Theorem XIII.83 and Theorem XIII.84]{RS77}. There, we denote by $\mathcal{L}(L^2((-1,1);\C^N))$ the space of all bounded operators in $L^2((-1,1);\C^N)$.
	\begin{lemma}\label{lem_decomp_op}
		Let $\mathfrak{M} \in L^\infty(\R^{\theta-1};\mathcal{L}(L^2((-1,1);\C^N)))$ and $M$ be defined by
		\begin{equation*}
			M : \mathcal{B}^0(\R^{\theta-1}) \to \mathcal{B}^0(\R^{\theta-1}), \qquad
			\mathfrak{i}  M \mathfrak{i}^{-1} f(\xi') := \mathfrak{M}(\xi') f(\xi'). 
		\end{equation*}
		Then, 
		\begin{equation*}
			\begin{aligned}
			\| M \|_{0 \to 0} &=  \|\mathfrak{M}\|_{L^\infty(\R^{\theta-1};\mathcal{L}(L^2((-1,1);\C^N)))} \\
			&= \textup{ess sup}_{\xi' \in \R^{\theta -1}} \| \mathfrak{M}(\xi')\|_{L^2((-1,1);\C^N) \to L^2((-1,1);\C^N)}
			\end{aligned}
		\end{equation*}
		 and the operator $M$ is continuously invertible if and only if $\mathfrak{M}$ is continuously invertible a.e. and the inequality ${\|\mathfrak{M}^{-1}\|_{L^\infty(\R^{\theta-1};\mathcal{L}(L^2((-1,1);\C^N)))} < \infty}$ holds. Furthermore, in this case $\|M^{-1}\|_{0\to0} = \|\mathfrak{M}^{-1}\|_{L^\infty(\R^{\theta-1};\mathcal{L}(L^2((-1,1);\C^N)))} $.
	\end{lemma}

	Next, we study the operator $\mathfrak{D}_{\varepsilon,\xi'}(z)$ in detail. For this purpose, we introduce for $\rho \in (0,\infty)$ and $w' \in \R^{\theta-1}$ with $\abs{w'} = 1$ the auxiliary operator
	\begin{equation*}
		\begin{aligned}
			\mathfrak{H}_{\rho,w'} :L^2((-1,1);\C^N) &\to L^2((-1,1);\C^N) \\
			(\mathfrak{H}_{\rho,w'} f) (t) &:= \int_{-1}^{1} \big( \widetilde{\alpha}'\cdot w' +  i  \widetilde{\alpha}_\theta \sign(t-s) ) \frac{e^{-\rho\abs{t-s}}}{2}f(s)\,ds.
		\end{aligned}
	\end{equation*}
	It is not difficult to check that $\mathfrak{H}_{\rho,w'}$ is a self-adjoint Hilbert-Schmidt operator.
	
	\begin{lemma}\label{lem_diff_frak_operators} 
		Let  $\varepsilon > 0$, $ \xi' \in \R^{\theta-1} \setminus \{0\}$, and $z \in \rho(H)$. Then, there exists a constant $C_1>0$ which only depends on $m$ and $z$ such that
		\begin{equation*}
			\begin{aligned}
				\norm{ \mathfrak{D}_{\varepsilon,\xi'}({z})-\mathfrak{D}_{0,\xi'}({z}) }_{L^2( 	(-1,1); \C^N) \to L^2((-1,1) ; \C^N) }  &\leq  C_1 \varepsilon (1+\abs{\xi'}),\\
				\norm{ 	\mathfrak{D}_{\varepsilon,\xi'}({z})-\mathfrak{H}_{\abs{\xi'}\varepsilon,\xi'/\abs{\xi'}}}_{L^2( (-1,1); \C^N) \to L^2((-1,1) ; \C^N) }  &\leq \frac{C_1}{1+\abs{\xi'}}.
			\end{aligned}
		\end{equation*}
	\end{lemma}
	\begin{proof}
	 	In this proof $C>0$ denotes a constant which may change in-between lines, but only depends on $m$ and $z$.	We start by estimating the kernel of the integral operator $\mathfrak{D}_{\varepsilon,\xi'}({z})-\mathfrak{D}_{0,\xi'}({z})$. It is easy to see that we can estimate this kernel by
		\begin{equation*}
			\begin{aligned}
				C\Bigl|1- e^{\varepsilon \abs{t-s}i\sqrt{z^2-m^2-\abs{\xi'}^2}} \Bigr| &\leq C\varepsilon \abs{t-s} \sqrt{ \abs{ z^2-m^2-\abs{\xi'}^2}} \\
				&\leq C \varepsilon  (1+\abs{\xi'}).
			\end{aligned}
		\end{equation*}
		With this estimate one finds a constant $C_1$ which only depends on $m$ and $z$ such that 
		\begin{equation*}
			\norm{\mathfrak{D}_{\varepsilon,\xi'}({z})-\mathfrak{D}_{0,\xi'}({z})}_{L^2((-1,1); \C^N) \to L^2((-1,1);\C^N)} \leq  C_1 \varepsilon  (1+\abs{\xi'}).
		\end{equation*}
		Next, we estimate the kernel of $\mathfrak{D}_{\varepsilon,\xi'}({z})-\mathfrak{H}_{\abs{\xi'}\varepsilon,\xi'/\abs{\xi'}} $  by
		\begin{equation}\label{eq_diff_frak_op_kernel}
			\begin{aligned}
				&\frac{1}{2}\biggl|\left( e^{-\varepsilon\abs{t-s}|\xi'|} -e^{\varepsilon\abs{t-s}i{\sqrt{m^2-z^2-\abs{\xi'}^2}}} \right)\left( \widetilde{\alpha}'\cdot \frac{\xi'}{|\xi'|} +  \widetilde{\alpha}_\theta i \sign(t-s)\right)\biggr| \\
				&\hspace{ 30 pt}+  \frac{1}{2}\Biggl|e^{\varepsilon\abs{t-s} i{\sqrt{z^2-m^2-\abs{\xi'}^2 }}} \Big(\widetilde{\alpha}'\cdot \frac{\xi'}{|\xi'|}\Big)\Bigg(1-\frac{i\abs{\xi'}}{\sqrt{z^2-m^2-\abs{\xi'}^2}}\Bigg)\Biggr| \\
				&\hspace{ 30 pt}+\frac{1}{2}  \Biggl|e^{\varepsilon\abs{t-s}i{\sqrt{z^2-m^2-\abs{\xi'}^2}}} \Bigg(\frac{m \beta + {z} I_N}{\sqrt{z^2-m^2-\abs{\xi'}^2 }} \Bigg)\Biggr|.
			\end{aligned}
		\end{equation}
		The first term in \eqref{eq_diff_frak_op_kernel} can be estimated for $\varepsilon \abs{t-s} \leq 1$ by
		\begin{equation*}
			\begin{aligned}
				C \Bigl|e^{-\varepsilon\abs{t-s}|\xi'|} -e^{\varepsilon\abs{t-s}i{\sqrt{z^2-m^2-\abs{\xi'}^2}}}\Bigr| &\leq C \varepsilon |t-s| \Big||\xi'|+i\sqrt{z^2-m^2-\abs{\xi'}^2} \Big|\\
				&\leq C \frac{|z^2-m^2|}{\Big||\xi'|-i\sqrt{z^2-m^2-\abs{\xi'}^2} \Big|}\\
				&\leq  \frac{ C}{1+\abs{\xi'}},
			\end{aligned}
		\end{equation*}
		where we used  $z \in \rho(H) = \C \setminus \big((-\infty,-|m|]\cup[|m|,\infty) \big)$ and   $\textup{Im } \sqrt{w} >0$ for $w \in \C \setminus [0,\infty)$. For $\varepsilon |t-s|> 1$ we get
		\begin{equation*}
			C \Bigl| e^{-\varepsilon\abs{t-s}|\xi'|} -e^{\varepsilon\abs{t-s}i{\sqrt{z^2-m^2-\abs{\xi'}^2}}} \Bigr| \leq C (e^{-\abs{\xi'}} + e^{- \textup{Im} \sqrt{z^2 - m^2 -\abs{\xi'}^2}})  
			\leq  \frac{ C}{1+\abs{\xi'}}.
		\end{equation*}
		Similarly as we estimated the first term in the case $\varepsilon \abs{t-s} \leq 1$, the second term in \eqref{eq_diff_frak_op_kernel} can be estimated by
		\begin{equation*}
			\begin{aligned}
				C \bigg| 1-\frac{i\abs{\xi'}}{\sqrt{z^2-m^2-\abs{\xi'}^2}} \bigg| &= C \frac{ |z^2-m^2|}{\big|\sqrt{z^2-m^2-\abs{\xi'}^2}+i|\xi'|\big| \big|\sqrt{z^2-m^2-\abs{\xi'}^2}\big|}\\
				&\leq \frac{C}{(1+|\xi'|)^2} \leq \frac{C}{1+|\xi'|}.
			\end{aligned}
		\end{equation*}
		One also sees that third term in \eqref{eq_diff_frak_op_kernel} is smaller than $\tfrac{C}{1+|\xi'|}$ for sufficiently large $C$. Summing up, we have that the kernel of $\mathfrak{D}_{\varepsilon,\xi'}({z})-\mathfrak{H}_{\abs{\xi'}\varepsilon,\xi'/\abs{\xi'}} $ can be bounded by $\tfrac{C}{1+|\xi'|}$ and therefore if $C_1$  is chosen sufficiently large, then
		\begin{equation*}
			\norm{ \mathfrak{D}_{\varepsilon,\xi'}({z})-\mathfrak{H}_{\abs{\xi'}\varepsilon,\xi'/\abs{\xi'}}}_{L^2( (-1,1); \C^N) \to L^2((-1,1) ; \C^N) }  \leq \frac{C_1}{1+\abs{\xi'}}.
		\end{equation*}
	\end{proof}
	
	\begin{lemma}\label{lem_volterra}
		Let  $\rho>0$, $w' \in \R^{\theta-1}$ with $\abs{w'} =1$ and $q$ be as in \eqref{eq_q2}. Then, $\sigma(\sqrt{q}\mathfrak{H}_{\rho,w'}\sqrt{q}) \subset [-\tfrac{2}{\pi},\tfrac{2}{\pi}]$.
	\end{lemma}
	\begin{proof}
		To shorten notation we set $\widetilde{\alpha}_\pm := \widetilde{\alpha}' \cdot w' \pm i\widetilde{\alpha}_\theta$. Then, one has for $f \in L^2((-1,1);\C^N)$
		\begin{equation*}
			\mathfrak{H}_{\rho,w'}f(t) = \frac{1}{2}\int_t^1  e^{-\rho|t-s|}\widetilde{\alpha}_- f(s) \,ds   + \frac{1}{2}\int_{-1}^t   e^{-\rho|t-s|}\widetilde{\alpha}_+ f(s) \,ds.
		\end{equation*}
		Using the anti-commutation rules for $\widetilde{\alpha}' \cdot w'$ and $\widetilde{\alpha}_\theta$, and $(\widetilde{\alpha} \cdot w')^2 = \widetilde{\alpha}_\theta^2 = I_N$, see the definition of $\widetilde{\alpha}_1, \dots, \widetilde{\alpha}_\theta$ before Lemma~\ref{lem_Fourier_transform_G_z_eps} and \eqref{it_Dirac_matrices} in Section~\ref{sec_not}, shows that $\ran \widetilde{\alpha}_+ \perp \ran \widetilde{\alpha}_-$.  Hence, as $q \geq 0$,  we have for  $f \in L^2((-1,1); \C^N)$ 
		\begin{equation}\label{eq_volterra}
			\begin{aligned}
				\|\sqrt{q}\mathfrak{H}_{\rho,w'}\sqrt{q}f&\|_{L^2((-1,1);\C^N)}^2 \\
				=& \frac{1}{4} \int_{-1}^1 q(t) \bigg| \int_t^1 e^{-\rho|t-s|}\widetilde{\alpha}_-  \sqrt{q(s)} f(s)  \,ds  \bigg|^2 \,dt \\
				&+ \frac{1}{4} \int_{-1}^1 q(t) \bigg| \int_{-1}^t   e^{-\rho|t-s|}\widetilde{\alpha}_+ \sqrt{q(s)}f(s) \,ds  \bigg|^2 \,dt.
			\end{aligned}
		\end{equation}
		We start by estimating the first term on the right-hand side. To do so, we assume that $Q$ is the primitive function of $q$ such that $Q(1) =0$ and therefore by $\int_{-1}^1 q(s) \, ds =1$, $Q(-1) =-1$. Applying the Cauchy-Schwarz inequality and Fubini's theorem yields
		\begin{equation*}
			\begin{aligned}
				\frac{1}{4}\int_{-1}^1 q(t) &\biggl| \int_t^1  e^{-\rho|t-s|}\widetilde{\alpha}_-  \sqrt{q(s)} f(s)  \,ds  \biggr|^2 \,dt\\
				&= \frac{1}{4} \int_{-1}^1 q(t) \biggl| \int_t^1  \frac{\sqrt{\cos\bigl(\tfrac{\pi}{2}Q(s)\bigr)}}{\sqrt{\cos\bigl(\tfrac{\pi}{2}Q(s)\bigr)}}  e^{-\rho|t-s|}\widetilde{\alpha}_-  \sqrt{q(s)} f(s)  \,ds  \biggr|^2 \,dt\\
				&\leq \frac{1}{4} \int_{-1}^1 q(t) \bigg(\int_{t}^1 \cos\bigl(\tfrac{\pi}{2}Q(s)\bigr)q(s) \,ds \bigg) \bigg( \int_{t}^1 \frac{1}{\cos(\tfrac{\pi}{2}Q(s))}|\widetilde{\alpha}_-f(s)|^2\,ds \bigg) \, dt\\
				&= \frac{1}{2\pi} \int_{-1}^1  -\sin\bigl(\tfrac{\pi}{2}Q(t)\bigr) q(t)\bigg(\int_{t}^1 \frac{1}{\cos\bigl(\tfrac{\pi}{2}Q(s)\bigr)}|\widetilde{\alpha}_- f(s)|^2\,ds \bigg) \, dt\\
				&=  \frac{1}{2\pi} \int_{-1}^1  \bigg( \int_{-1}^s -\sin\bigl(\tfrac{\pi}{2}Q(t)\bigr)q(t) \,dt \bigg) \frac{1}{\cos\bigl(\tfrac{\pi}{2}Q(s)\bigr)}|\widetilde{\alpha}_- f(s)|^2\,ds \\
				&= \frac{1}{\pi^2} \int_{-1}^1 |\widetilde{\alpha}_- f(s)|^2 \, ds.
			\end{aligned}
		\end{equation*}
		The same trick with $Q +1$ instead of $Q$ yields that the second term of the right hand side of equation \eqref{eq_volterra} can be estimated by  $\tfrac{1}{\pi^2} \int_{-1}^1 |\widetilde{\alpha}_+ f(s)|^2 \, ds.$ Using these estimates and  $\ran \widetilde{\alpha}_+ \perp \ran \widetilde{\alpha}_-$ gives us
		\begin{equation*}
			\begin{aligned}
				\|\sqrt{q}\mathfrak{H}_{\rho,w'}\sqrt{q}f\|_{L^2((-1,1);\C^N)}^2 & \leq \frac{1}{\pi^2} \int_{-1}^{1} |\widetilde{\alpha}_- f(s)|^2 +|\widetilde{\alpha}_+ f(s) |^2 \,ds \\
				&= \frac{1}{\pi^2} \int_{-1}^{1} |(\widetilde{\alpha}_- + \widetilde{\alpha}_+) f(s) |^2 \,ds\\
				&= \frac{1}{\pi^2} \int_{-1}^{1} |2(\widetilde{\alpha}' \cdot w')f(s) |^2 \,ds \\
				& = \frac{4}{\pi^2} \norm{f}_{L^2((-1,1);\C^N)}^2.
			\end{aligned}
		\end{equation*}
		Since $\mathfrak{H}_{\rho,w'}$ is self-adjoint in $L^2((-1,1);\C^N)$, we obtain $\sigma(\sqrt{q}\mathfrak{H}_{\rho,w'}\sqrt{q}) \subset [-\tfrac{2}{\pi}, \tfrac{2}{\pi}]$. 
	\end{proof}
	
	Having studied the spectrum of $\sqrt{q}\mathfrak{H}_{\rho,w'}\sqrt{q}$, we employ this knowledge to study the bounded invertibility of $I + \mathfrak{H}_{\rho,w'}Q_{\eta,\tau}q$. Recall that $Q_{\eta,\tau} = \eta I_N + \tau \beta$ for $\eta,\tau \in \R$. 

	\begin{lemma}\label{lem_frak_h_inv_estimate}
		Let $\rho > 0$, $w' \in \R^{\theta-1}$ with $\abs{w'}=1$, $\eta,\tau \in \R$, $Q_{\eta,\tau}= \eta I_N + \tau \beta$, $d = \eta^2-\tau^2$,
		\begin{equation*}
			c(d):= 
			\begin{cases}
				\frac{4 d}{\pi^2}, &  d \geq  0, \\
				0,  & d<0, 
			\end{cases}
		\end{equation*} 
		and  $q$ be as in \eqref{eq_q2}.
		If $d < \frac{\pi^2}{4}$, then  $I + \mathfrak{H}_{\rho,w'}Q_{\eta,\tau} q$ is continuously invertible in $L^{2}((-1,1);\C^N)$ and the norm of the inverse is bounded by the constant 
		\begin{equation}\label{eq_def_C_2}
			C_{2} = C_{2}(\eta,\tau) := 4\norm{q}_{L^\infty((-1,1))} (\abs{\eta}+\abs{\tau})\frac{1 + (\abs{\eta}+\abs{\tau})\frac{2}{\pi}}{(1-c(d))\pi}  + 1.
		\end{equation}
	\end{lemma}
	\begin{proof}
		Using $\mathfrak{H}_{\rho,w'} Q_{\eta,\tau} = Q_{\eta,-\tau} \mathfrak{H}_{\rho,w'}$ and $Q_{\eta,-\tau} Q_{\eta,\tau} = d I_N$, which follows from the anti-commutation relations for $\widetilde{\alpha}_1, \dots, \widetilde{\alpha}_\theta$, discussed before Lemma~\ref{lem_Fourier_transform_G_z_eps}, and $\beta$, gives us
		\begin{equation*}
			I -	d (\sqrt{q}\mathfrak{H}_{\rho,w'}\sqrt{q})^2 = ( I	-\sqrt{q}\mathfrak{H}_{\rho,w'}Q_{\eta,\tau}\sqrt{q} ) ( I+	\sqrt{q}\mathfrak{H}_{\rho,w'} Q_{\eta,\tau} \sqrt{q}) .
		\end{equation*}
		If $d<\tfrac{\pi^2}{4}$, then Lemma~\ref{lem_volterra} implies $1 \in \rho(d( \sqrt{q}\mathfrak{H}_{\rho,w'}\sqrt{q})^2)$ and therefore  the operator $I + \sqrt{q}\mathfrak{H}_{\rho,w'}Q_{\eta,\tau}\sqrt{q} $ is also continuously invertible in $L^{2}((-1,1);\C^N)$ and 
		\begin{equation*}
			\begin{aligned}
				&\bigl\|(I+ \sqrt{q}\mathfrak{H}_{\rho,w'} Q_{\eta,\tau}\sqrt{q})^{-1}\bigr\|_{L^2((-1,1);\C^N) \to L^2((-1,1);\C^N) }   \\
				&\hspace{20 pt}= 	\bigl\|(I -	d (\sqrt{q}\mathfrak{H}_{\rho,w'}\sqrt{q})^2)^{-1}( I	-\sqrt{q}\mathfrak{H}_{\rho,w'}Q_{\eta,\tau}\sqrt{q} )\bigr\|_{L^2((-1,1);\C^N) \to L^2((-1,1);\C^N) }\\
				&\hspace{20 pt}\leq \frac{1 + (\abs{\eta}+\abs{\tau})\frac{2}{\pi}}{1-c(d)}. 
			\end{aligned}
		\end{equation*}
		Moreover,   $(I+\mathfrak{H}_{\rho,w'} Q_{\eta,\tau}q)^{-1} = I - \mathfrak{H}_{\rho,w'} Q_{\eta,\tau}\sqrt{q}(I+ \sqrt{q}\mathfrak{H}_{\rho,w'} Q_{\eta,\tau}\sqrt{q})^{-1} \sqrt{q} $ and  hence
		\begin{equation*}
			\begin{aligned}
				&\bigl\|(I+ \mathfrak{H}_{\rho,w'}  Q_{\eta,\tau}q)^{-1}\bigr\|_{L^2((-1,1);\C^N) \to L^2((-1,1);\C^N) }   \\
				&\hspace{15 pt}= 	\bigl\|\mathfrak{H}_{\rho,w'} Q_{\eta,\tau}\sqrt{q} (I+ \sqrt{q}\mathfrak{H}_{\rho,w'} Q_{\eta,\tau}\sqrt{q})^{-1}\sqrt{q} - I  \bigr\|_{L^2((-1,1);\C^N) \to L^2((-1,1);\C^N) }\\
				&\hspace{15 pt}\leq \norm{q}_{L^\infty((-1,1))} (\abs{\eta}+\abs{\tau}) \norm{\mathfrak{H}_{\rho,w'}   }_{L^2((-1,1);\C^N) \to L^2((-1,1);\C^N) } \frac{1 + (\abs{\eta}+\abs{\tau})\frac{2}{\pi}}{1-c(d)}  + 1\\
				&\hspace{15 pt}\leq 4\norm{q}_{L^\infty((-1,1))} (\abs{\eta}+\abs{\tau})\frac{1 + (\abs{\eta}+\abs{\tau})\frac{2}{\pi}}{(1-c(d))\pi}  + 1,
			\end{aligned}
		\end{equation*}
		where we used Lemma~\ref{lem_volterra} (for the constant function $q = \tfrac{1}{2}$) to estimate the term $\norm{\mathfrak{H}_{\rho,w'}   }_{L^2((-1,1);\C^N) \to L^2((-1,1);\C^N) }$  by $\tfrac{4}{\pi}$.
	\end{proof}

	In the last part of Section~\ref{sec_hyperplane} we use our findings to prove  norm estimates for the operator $(I+D_\varepsilon^{0,\kappa}(z)Q_{\eta,\tau}q)^{-1}$.
	\begin{proposition}\label{prop_uniform_bdd_D}
		Let $z \in \C \setminus \R$,  $\kappa \in \textup{SO}(\theta)$, $\eta,\tau \in \R$ such that $d = \eta^2 - \tau^2< \tfrac{\pi^2}{4}$, and $q$ be as in \eqref{eq_q2}. Moreover, let 
		\begin{equation}\label{eq_def_C_3}
			C_3 = C_3(\eta,\tau,\kappa) :=2\max\{C_2, \| (I+D_{0}^{0,\kappa}(z)Q_{\eta,\tau}q)^{-1} \|_{0\to0} \},
		\end{equation}
		with $C_2 = C_2(\eta,\tau)$ as in \eqref{eq_def_C_2}, and  
		\begin{equation}\label{eq_def_delta_1}
			\delta_1 =\delta_1(\eta,\tau,\kappa):= ( C_3 C_1(|\eta| + |\tau|) \norm{q}_{L^\infty((-1,1))})^{-2},
		\end{equation} 
		with $C_1$ from  Lemma~\ref{lem_diff_frak_operators}. Then,
		\begin{equation*}
			\sup_{\varepsilon \in (0,\delta_1) } \norm{(I+D^{{0,\kappa}}_\varepsilon(z)Q_{\eta,\tau}q)^{-1}}_{0 \to 0} \leq C_3 < \infty.
		\end{equation*} 
	\end{proposition}
	\begin{proof}
		We claim that $C_3 < \infty$. Since $d< \tfrac{\pi^2}{4}$, the constant $C_2$ from \eqref{eq_def_C_2} is finite. Thus, it suffices to show $\| (I+D_{0}^{0,\kappa}(z)Q_{\eta,\tau}q)^{-1} \|_{0\to0} < \infty$ in order to proof the claim $C_3 < \infty$.
		
		 Applying  Proposition~\ref{prop_inv_I+BVq} (for $V= Q_{\eta,\tau} = \eta I_2+\tau\beta =\textup{const.}$ and $r=0$) shows that $I + B_0(z)Q_{\eta,\tau}q$ is continuously invertible in $\mathcal{B}^0(\Sigma_{0,\kappa})$. Let us shortly explain why  Proposition~\ref{prop_inv_I+BVq} is indeed applicable. The inequality $d < \tfrac{\pi^2}{4}$ implies $d \not= (2k+1)^2\pi^2$, $k \in \N_0$, and therefore \eqref{eq_scaling_cond} is fulfilled. In the same way as in Proposition~\ref{prop_inv_I+BVq}  we set $(\widetilde{\eta},\widetilde{\tau})  = \textup{tanc}\bigl(\tfrac{\sqrt{d}}{2}\bigr) (\eta,\tau) $. Then, $\widetilde{d} = \widetilde{\eta}^2-\widetilde{\tau}^2 = 4\tan\bigl(\tfrac{\sqrt{d}}{2}\bigr) < 4$, i.e. $\widetilde{d}$ fulfils \eqref{eq_d_tilde_non_crit}. Hence, the assumptions of Proposition~\ref{prop_inv_I+BVq} are satisfied and its application is justified. Together with \eqref{eq_norm_iota} and \eqref{def_D_operators}, this implies that  the operator $I + D_0^{0,\kappa}(z)Q_{\eta,\tau}q = \iota_{0,\kappa} (I + B_0(z)Q_{\eta,\tau}q) \iota_{0,\kappa}^{-1}$ is continuously invertible in $\mathcal{B}^0(\R^{\theta-1})$, which proves $\| (I+D_{0}^{0,\kappa}(z)Q_{\eta,\tau}q)^{-1} \|_{0\to0} < \infty$. Hence, $C_3 < \infty$.
		 
		According to Lemma \ref{lem_diff_frak_operators} we have
		\begin{equation*}
			\begin{aligned}
				&\bigl\|\mathfrak{D}_{\varepsilon,\xi'}({z})Q_{\eta,\tau}q - \mathfrak{H}_{\abs{\xi'}\varepsilon,\xi'/\abs{\xi'}}Q_{\eta,\tau}q \bigr\|_{L^2((-1,1);\C^N) \to L^2((-1,1);\C^N) } \\
				&\hspace{50 pt}\leq C_1 \norm{q}_{L^\infty((-1,1))}\frac{\abs{\eta}+\abs{\tau}}{1+|\xi'|} = \frac{\delta_1^{-1/2}}{C_3(1+|\xi'|)}
			\end{aligned}
		\end{equation*}	
		for $\xi' \in \R^{\theta-1}\setminus\{0\}$ and all $\varepsilon>0$. Hence, if we choose $R:=\delta_1^{-1/2}-1$, then the choices of $C_3$, $\delta_1$ and $R$, and Lemma~\ref{lem_frak_h_inv_estimate} yield for $0\neq |\xi'| \geq R$  and $\varepsilon>0$
		\begin{equation*}
			\begin{aligned}
				\bigl\|(I+&\mathfrak{H}_{\abs{\xi'}\varepsilon,\xi'/\abs{\xi'}}Q_{\eta,\tau}q)^{-1}\\
				&\cdot(\mathfrak{D}_{\varepsilon,\xi'}({z})Q_{\eta,\tau}q - \mathfrak{H}_{\abs{\xi'}\varepsilon,\xi'/\abs{\xi'}}Q_{\eta,\tau}q)\bigr\|_{L^2((-1,1);\C^N) \to L^2((-1,1);\C^N) } \\
				&\leq C_2 \frac{\delta_1^{-1/2}}{C_3 (1+R)} \\
				& \leq \frac{C_3}{2} \cdot \frac{ \delta_1^{-1/2}}{ C_3 \delta_1^{-1/2}} \\
				&=  \frac{1}{2}.
			\end{aligned}
		\end{equation*}
 In particular,
		\begin{equation*}
			\mathfrak{P}_{\varepsilon,\xi'}:=	I +(I+\mathfrak{H}_{\abs{\xi'}\varepsilon,\xi'/\abs{\xi'}}Q_{\eta,\tau}q)^{-1}
			(\mathfrak{D}_{\varepsilon,\xi'}({z})Q_{\eta,\tau}q - \mathfrak{H}_{\abs{\xi'}\varepsilon,\xi'/\abs{\xi'}}Q_{\eta,\tau}q)
		\end{equation*}
		is continuously invertible in $L^2((-1,1);\C^N)$ and the norm of its inverse can be bounded by $2$ for $0\neq |\xi'| \geq R$ and $\varepsilon >0$. This implies that also 
		\begin{equation*}
			I+\mathfrak{D}_{\varepsilon,\xi'}({z})Q_{\eta,\tau}q=(I+\mathfrak{H}_{\abs{\xi'}\varepsilon,\xi'/\abs{\xi'}}Q_{\eta,\tau}q) \mathfrak{P}_{\varepsilon,\xi'}
		\end{equation*}
		is continuously invertible in $L^2((-1,1);\C^N)$  and by Lemma~\ref{lem_frak_h_inv_estimate}  the corresponding norm estimate
		\begin{equation}\label{eq_uniform_bdd_D_1}
			\begin{aligned}
				&\norm{ (I+\mathfrak{D}_{\varepsilon,\xi'}({z})Q_{\eta,\tau}q)^{-1} }_{L^2((-1,1);\C^N) \to L^2((-1,1);\C^N) }  \\
				&\hspace{50 pt}=	\big\|\mathfrak{P}_{\varepsilon,\xi'}^{-1}(I+\mathfrak{H}_{\abs{\xi'}\varepsilon,\xi'/\abs{\xi'}}Q_{\eta,\tau}q)^{-1} \big\|_{L^2((-1,1);\C^N) \to L^2((-1,1);\C^N) }  \\
				&\hspace{50 pt} \leq 2 C_2 \\
				& \hspace{50 pt} \leq C_3
			\end{aligned}
		\end{equation}	
		is valid	for $0\neq |\xi'| \geq R$ and $\varepsilon >0$. 
		
		Having found an estimate for $0 \neq |\xi'| \geq R$, we aim to find a similar estimate for $0 \neq |\xi'| \leq R$.
		Again, according to  Lemma \ref{lem_diff_frak_operators}  we have for $\xi' \in \R^{\theta-1}\setminus\{0\}$ and $\varepsilon>0$
		\begin{equation*}
			\begin{aligned}
				&\norm{\mathfrak{D}_{\varepsilon,\xi'}({z})Q_{\eta,\tau}q - \mathfrak{D}_{0,\xi'}({z})Q_{\eta,\tau}q }_{L^2((-1,1);\C^N) \to L^2((-1,1);\C^N) }\\
				&\hspace{30 pt}\leq C_1 \norm{q}_{L^\infty((-1,1))}(\abs{\eta}+\abs{\tau})\varepsilon(1+|\xi'|)
				=\varepsilon \frac{\delta_1^{-1/2}(1+|\xi'|)}{C_3}.
			\end{aligned}
		\end{equation*}
		
		Moreover, Lemma \ref{lem_decomp_op} (with $\mathfrak{M}(\xi') =  I+\mathfrak{D}_{0,\xi'}({z})Q_{\eta,\tau}q$ for $\xi' \in \R^{\theta-1}$)  and \eqref{eq_def_C_3} imply 
		\begin{equation*}
			\begin{aligned}
				\textup{ess sup}_{\xi'\in\R^{\theta-1}} &\| (I+\mathfrak{D}_{0,\xi'}({z})Q_{\eta,\tau}q)^{-1} \|_{L^2((-1,1);\C^N) \to L^2((-1,1);\C^N) } \\
				=&\|(I+D^{{0,\kappa}}_0(z)Q_{\eta,\tau}q)^{-1}\|_{0 \to 0} 
				\leq \frac{C_3}{2}.
			\end{aligned}
		\end{equation*}
		Hence, as $1+R = \delta_1^{-1/2}$, we can estimate similarly as in the first part of the proof for $\varepsilon \in (0,\delta_1)$
		\begin{equation}\label{eq_uniform_bdd_D_2}
			\begin{aligned}
				&\textup{ess sup}_{\abs{\xi'} \leq R} \norm{ (I+\mathfrak{D}_{\varepsilon,\xi'}({z})Q_{\eta,\tau}q)^{-1} }_{L^2((-1,1);\C^N) \to L^2((-1,1);\C^N) } \\
				&= \textup{ess sup}_{\abs{\xi'} \leq R} \big\| \big[I+(I+\mathfrak{D}_{0,\xi'}({z})Q_{\eta,\tau}q)^{-1}\bigl(\mathfrak{D}_{\varepsilon,\xi'}({z})Q_{\eta,\tau}q  -\mathfrak{D}_{0,\xi'}({z})Q_{\eta,\tau}q \bigr)\big]^{-1}\\
				&	\hspace{100 pt}\cdot(I+\mathfrak{D}_{0,\xi'}({z})Q_{\eta,\tau}q)^{-1} \big\|_{L^2((-1,1);\C^N) \to L^2((-1,1);\C^N) }\\
				&\leq \frac{1}{1-\frac{C_3}{2} \cdot \frac{\varepsilon\delta_1^{-1/2}(1+R)}{C_3}} \cdot \frac{C_3}{2}\\
				&= \frac{1}{1-\frac{C_3}{2} \cdot \frac{\varepsilon\delta_1^{-1}}{C_3}} \cdot \frac{C_3}{2}  \leq C_3.
			\end{aligned}
		\end{equation}	
		Combining \eqref{eq_uniform_bdd_D_1} and  \eqref{eq_uniform_bdd_D_2}, and applying  Lemma~\ref{lem_decomp_op} gives us
		\begin{equation*}
			\begin{aligned}
				&\hspace{-10 pt}\norm{(I+D_{\varepsilon}^{0,\kappa}({z})Q_{\eta,\tau}q)^{-1} }_{0\to0}\\
				&= \max\big\{ \textup{ess sup}_{|\xi'| \geq R
				}\norm{(I+\mathfrak{D}_{\varepsilon,\xi'}({z})Q_{\eta,\tau}q)^{-1}}_{L^2((-1,1);\C^N) \to L^2((-1,1);\C^N) } , \\ 
				&\hspace{47 pt}\textup{ess sup}_{|\xi'| < R}\norm{(I+\mathfrak{D}_{\varepsilon,\xi'}({z})Q_{\eta,\tau}q)^{-1}}_{L^2((-1,1);\C^N) \to L^2((-1,1);\C^N) } \big\}  \\
				&\leq C_3
			\end{aligned} 
		\end{equation*}	
		for $\varepsilon \in (0,\delta_1)$.
	\end{proof}

	\begin{cor}\label{cor_uniform_bdd_D}
		Let $z \in \C \setminus \R$,   $q$ be as in \eqref{eq_q2}, $S \subset \R^2$ be compact and $\max_{(\eta,\tau) \in S} \eta^2-\tau^2 < \tfrac{\pi^2}{4}$. Then, there exists a $\delta_2= \delta_2(S) > 0$  such that 
		\begin{equation*}
			\sup_{(\varepsilon, y_0, (\eta,\tau), \kappa ) \in (0,\delta_2) \times \R \times S \times  \textup{SO}(\theta)} \norm{(I+D^{{y_0,\kappa}}_\varepsilon(z)Q_{\eta,\tau}q)^{-1}}_{0 \to 0} < \infty.
		\end{equation*} 
	\end{cor}
	\begin{proof}
		Since $D_\varepsilon^{y_0,\kappa}(z) = D_\varepsilon^{0,\kappa}(z)$, see the text below \eqref{eq_D_eps_straight}, the assertion follows directly from Proposition~\ref{prop_uniform_bdd_D} if we can show
		\begin{equation*}
			\sup_{((\eta,\tau),\kappa) \in S \times \textup{SO}(\theta)} C_3(\eta,\tau,\kappa) < \infty \quad \textup{and} \quad \inf_{((\eta,\tau),\kappa) \in S \times \textup{SO}(\theta)} \delta_1(\eta,\tau,\kappa) >0,
		\end{equation*}
		with $C_3$ and $\delta_1$ as in Proposition \ref{prop_uniform_bdd_D}. Note also that as $S$ is bounded the first inequality  and \eqref{eq_def_delta_1} imply the second inequality in the above displayed formula. Moreover, the assumption $\max_{(\eta,\tau)\in S} \eta^2-\tau^2< \tfrac{\pi^2}{4}$ implies $\max_{(\eta,\tau) \in S}C_2(\eta,\tau)< \infty$, where $C_2$ was defined in \eqref{eq_def_C_2}. Hence, it follows from \eqref{eq_def_C_3} that the inequality  $\sup_{((\eta,\tau),\kappa) \in S \times \textup{SO}(\theta)} C_3(\eta,\tau,\kappa) < \infty $  is fulfilled if 
		\begin{equation}\label{eq_cor_D_uniform_bdd}
			\sup_{((\eta,\tau),\kappa)\in S\times\textup{SO}(\theta)} \|(I+D^{{0,\kappa}}_0(z)Q_{\eta,\tau}q)^{-1}\|_{0 \to 0} < \infty.
		\end{equation}
		Obviously, the operator $D^{{0,\kappa}}_0(z)Q_{\eta,\tau}q$ depends continuously on  $\eta$, $\tau$ with respect to the operator norm in $\mathcal{B}^0(\R^{\theta-1})$. According to    Proposition~\ref{prop_Fourier_D} $D^{{0,\kappa}}_0(z)Q_{\eta,\tau}q$ depends also continuously on $\kappa$ with respect to the operator norm. Moreover, Proposition~\ref{prop_uniform_bdd_D} gives us that the operator $I+ D^{{0,\kappa}}_0(z)Q_{\eta,\tau}q  $ is continuously invertible in $\mathcal{B}^0(\R^{\theta-1})$ for $((\eta,\tau), \kappa) \in S \times \textup{SO}(\theta)$. Thus, as $S \times \textup{SO}(\theta)$ is compact,  \eqref{eq_cor_D_uniform_bdd} is indeed true.
	\end{proof} 
	
	\subsubsection{$\Sigma$ is a rotated $C^2_b$-graph}\label{sec_loc_part}
	After treating the case of affine hyperplanes, we turn to the case where $\Sigma$ is a rotated $C^2_b$-graph as in \eqref{eq_rotated_graph}, i.e.  there exist ${\zeta \in C^2_b(\R^{\theta-1};\R)}$ and $\kappa \in \textup{SO}(\R^{\theta})$ such   that 
	\begin{equation*}
		\Sigma = \Sigma_{\zeta,\kappa} = \{\kappa(x',\zeta(x')): x' \in \R^{\theta -1} \}.
	\end{equation*}
	Moreover, recall that $x_{\zeta,\kappa}(x') = \kappa(x',\zeta(x'))$ and $\nu_{\zeta,\kappa}(x') = \nu(x_{\zeta,\kappa}(x'))$, $x' \in \R^{\theta-1}$, see \eqref{eqref_x_nu_zeta_kappa}. According to \cite[Proposition A.2 (i), first line of the proof of Proposition 2.4, and eq. (3.11)]{BHS23} there exists a $C_4 = C_4(\Sigma)>0$ such that for all $x',y' \in \R^{\theta-1}$ and $\widetilde{\varepsilon} \in (-2\varepsilon_2,2\varepsilon_2)$
	\begin{equation}\label{eq_est_x_Sigma}
		C_4^{-1}\big(|x'-y' | + |\widetilde{\varepsilon}|\big)  \leq|x_{\zeta,\kappa}(x') - x_{\zeta,\kappa}(y') + \widetilde{\varepsilon} \nu_{\zeta,\kappa}(x')| \leq C_4 \big(|x'-y' | + |\widetilde{\varepsilon}|\big),  
	\end{equation}
	with $\varepsilon_2>0$ chosen as in Proposition~\ref{prop_conv_res}.
	We are going to prove the uniform boundedness of $(I + B_\varepsilon(z)Vq)^{-1}$ in $\mathcal{B}^0(\Sigma)$ with respect to $\varepsilon \in (0,\varepsilon_3)$, for a suitable $\varepsilon_3 \in (0,\varepsilon_2)$.  It follows from \eqref{eq_B_diff}, \eqref{eq_norm_iota}, \eqref{def_D_operators}, and  \eqref{eq_Q_matrix} that this is equivalent to proving the uniform boundedness of  $(I + D_\varepsilon^{\zeta,\kappa}(z)Q^{\zeta,\kappa}_{\eta,\tau}q)^{-1} = \iota_{\zeta,\kappa}(I + \overline{B}_\varepsilon(z)Vq)^{-1} \iota_{\zeta,\kappa}^{-1}$.
	
	We start by analysing $D_\varepsilon^{\zeta,\kappa}(z)$ locally. To do so we need to introduce further notations. For $x_0' \in \R^{\theta-1}$ we define 
	\begin{equation}\label{eq_zeta_x_0'}
		\zeta_{x_0'}(x') := \zeta(x_0') + \langle\nabla \zeta(x_0'), x'-x_0' \rangle, \qquad  x'\in \R^{\theta-1}.
	\end{equation}
	Moreover, we define the localisation parameter $a_\varepsilon := \varepsilon^{1/6}$ for $\varepsilon \in (0,\varepsilon_2)$.
	Next, we introduce a family of auxiliary operators. For this, we choose a $C^\infty$-function $\omega$ with $0 \leq \omega \leq 1$, $\omega = 1$ on $\R^{\theta-1} \setminus B(0,1)$ and $\omega =0$ on $B(0,\frac{1}{2})$. We use this function to cut out the singular part of the integral kernel of $D_{\varepsilon}^{\zeta,\kappa}(z)$; cf. \eqref{eq_D_eps_int_rep_0}. More precisely, in analogy with  \eqref{eq_D_eps_int_rep} and  \eqref{eq_d_eps}, we define for $\varepsilon \in (0,\varepsilon_2)$ and $\widetilde{\varepsilon} \in (-2\varepsilon_2,2\varepsilon_2)\setminus \{0\}$ the operators
	\begin{equation}\label{def_e_eps}
		\begin{aligned}
			e^{a_\varepsilon}_{\widetilde{\varepsilon}}(z) &: L^2(\R^{\theta-1};\C^N) \to L^2(\R^{\theta-1};\C^N),\\
			e^{a_\varepsilon}_{\widetilde{\varepsilon}}(z)  g (x') &:= \int_{\R^{\theta-1}} G_z(x_{\zeta,\kappa}(x') -x_{\zeta,\kappa}(y') + \widetilde{\varepsilon}\nu_{\zeta,\kappa}(x')) \omega\bigl(\tfrac{x' - y'}{a_\varepsilon}\bigr)\\
			&\hspace{150 pt}\cdot\sqrt{1+ \abs{\nabla \zeta(y')}^2} g(y') \,dy',
		\end{aligned}
	\end{equation}
	and 
	\begin{equation}\label{def_E_eps}
		E_{{\varepsilon}}(z) :\mathcal{B}^0(\R^{\theta-1}) \to \mathcal{B}^0(\R^{\theta-1}), \qquad 
		E_{{\varepsilon}}(z)  f(t):= \int_{-1}^1 e^{a_\varepsilon}_{\varepsilon(t-s)}(z)f(s) \,ds.
	\end{equation}
	
	We start by proving preliminary results for  $e_{\widetilde{\varepsilon}}^{a_\varepsilon}(z)$ and $d^{\zeta,\kappa}_{\widetilde{\varepsilon}}(z)$. Afterwards, we transfer these results to $E_\varepsilon(z)$ and $D^{\zeta,\kappa}_\varepsilon(z)$ in Proposition \ref{prop_local_prop}.
	For the estimates regarding   $e_{\widetilde{\varepsilon}}^{a_\varepsilon}(z)$ and $d^{\zeta,\kappa}_{\widetilde{\varepsilon}}(z)$ the following lemma turns out to be useful.
	\begin{lemma}\label{lem_est_G_z}
		Let $z \in \rho(H)$  and $G_z$ be defined by \eqref{eq_G_z_2D}--\eqref{eq_G_z_3D}. Then, there exist $C_5=C_5(z,m)>0$ and $C_6=C_6(z,m)>0$ such that for all $x \in \R^{\theta} \setminus \{0\}$ and $j \in \{1,\dots,\theta\}$
		\begin{equation}\label{eq_G_z_est_R_theta}
			\begin{aligned}
				|G_z(x)| &\leq  C_5|x|^{1-\theta}e^{-C_6 |x|}, \\ 	|\partial_j G_z(x)| &\leq C_5|x|^{-\theta}e^{-C_6 |x|}.
			\end{aligned}
		\end{equation}
		In particular, there holds for all $x',y' \in \R^{\theta-1}$, $j \in \{1,\dots,\theta\}$ and $\widetilde{\varepsilon} \in (-2\varepsilon_2,2\varepsilon_2)$
		\begin{equation}\label{eq_G_z_est_R_theta_prime}
			\begin{aligned}
				|G_z(x_{\zeta,\kappa}(x') - x_{\zeta,\kappa}(y') + \widetilde{\varepsilon}\nu_{\zeta,\kappa}(x'))| &\leq  C_5C_4^{\theta-1}(|x'-y'| +|\widetilde{\varepsilon}|)^{1-\theta}e^{-\frac{C_6}{C_4} |x'-y'| }, \\	|\partial_j G_z(x_{\zeta,\kappa}(x') - x_{\zeta,\kappa}(y') + \widetilde{\varepsilon}\nu_{\zeta,\kappa}(x'))| &\leq C_5C_4^\theta(|x'-y'| +|\widetilde{\varepsilon}|)^{-\theta}e^{-\frac{C_6}{C_4} |x'-y'| }.
			\end{aligned}
		\end{equation}
	\end{lemma}
	\begin{proof}
		Rough estimations and asymptotic expansions of the modified Bessel functions $K_1$ and $K_0$, see \cite[\S10.25~(ii) and \S10.30~(i)]{DLMF} and  \cite[eq. (B.8)]{BHS23}, lead to
		\begin{equation*}
				|G_z(x)| \leq  C(1+|x|^{1-\theta})e^{-\textup{Im}\sqrt{z^2-m^2} |x|} 
        \end{equation*}
        and
        \begin{equation*}
            |\partial_j G_z(x)| \leq C(1+|x|^{-\theta})e^{-\textup{Im}\sqrt{z^2-m^2} |x|}
		\end{equation*}
		for all $x \in \R^{\theta}\setminus \{0\}$ and $j \in \{1,\dots,\theta\}$, where $C = C(m,z)>0$ is a constant which only depends on $m$ and $z$. Thus, \eqref{eq_G_z_est_R_theta} is valid if one chooses $ C_6 \in (0,\textup{Im} \sqrt{z^2-m^2})$ and 
		\begin{equation*}
			C_5 = \sup_{x \in \R^{\theta}\setminus\{0\}, l \in \{\theta-1,\theta\}} C\frac{1+ |x|^{-l}}{|x|^{-l}} e^{-(\textup{Im }\sqrt{m^2-z^2} - C_6)|x|  } < \infty.
		\end{equation*}
		Furthermore,  combining these estimates with \eqref{eq_est_x_Sigma} implies \eqref{eq_G_z_est_R_theta_prime}.
	\end{proof}
	
	\begin{lemma}\label{lem_e_eps^r} 
		Let $z \in \rho(H)$, $\varepsilon \in (0,\varepsilon_2)$, $a_\varepsilon = \varepsilon^{1/6}$, and $\widetilde{\varepsilon} \in (-2\varepsilon_2,2\varepsilon_2)\setminus \{0\}$. Then, the operator $e^{a_\varepsilon}_{\widetilde{\varepsilon}}(z)$ acts as a bounded operator from $L^2(\R^{\theta-1};\C^N)$ to $H^1(\R^{\theta-1};\C^N)$ and 
		\begin{equation*}
			\norm{e^{a_\varepsilon}_{\widetilde{\varepsilon}}(z)}_{L^2(\R^{\theta-1};\C^N) \to H^1(\R^{\theta-1};\C^N)} \leq C\frac{1+\abs{\log(\varepsilon)}}{a_\varepsilon},
		\end{equation*}
		where $C>0$ does not depend on $\widetilde{\varepsilon}$ and  $\varepsilon$. Moreover, for $f \in L^2(\R^{\theta-1};\C^N)$ the mapping $(-2\varepsilon_2,2\varepsilon_2)\setminus\{0\}\ni \widetilde{\varepsilon} \mapsto e^{a_\varepsilon}_{\widetilde{\varepsilon}}(z)f \in H^1(\R^{\theta-1};\C^N)$ is continuous.
	\end{lemma}
	\begin{proof}
		We aim to prove the assertion by applying Lemma~\ref{lem_Schur_test}. To do so, it is necessary to find suitable estimates for the integral kernel of $e^{a_\varepsilon}_{\widetilde{\varepsilon}}(z)$ which is for $x',y' \in \R^{\theta-1}$ given by
		\begin{equation*}
			k(x',y') := G_z(x_{\zeta,\kappa}(x') -x_{\zeta,\kappa}(y') + \widetilde{\varepsilon}\nu_{\zeta,\kappa}(x')) \omega\bigl(\tfrac{x' - y'}{a_\varepsilon}\bigr)\sqrt{1+ \abs{\nabla \zeta(y')}^2}.
		\end{equation*} 
		We notice as $G_z \in C^\infty(\R^{\theta}\setminus\{0\};\C^{N \times N})$, $\zeta \in C_b^2(\R^{\theta-1};\R)$, and $\omega \in C_b^\infty(\R^{\theta-1};\R)$ and as $\omega$ cuts out the singularity of $G_z$, we have $k \in C_b^1(\R^{\theta-1} \times \R^{\theta-1};\C^{N \times N})$. Furthermore, using \eqref{eq_G_z_est_R_theta_prime}, $0 \leq \omega \leq 1$, $\supp \omega \subset \R^{\theta-1}\setminus B(0,\frac{1}{2})$ and $\zeta \in C^2_b(\R^{\theta-1};\R)$ immediately gives us for $x' \neq y' \in \R^{\theta-1}$
		\begin{equation*}
			\begin{aligned}
				|k(x',y')| &\leq C \chi_{\R^{\theta-1}\setminus B(0,1/2)}\bigl(\tfrac{x'-y'}{a_\varepsilon}\bigr) (|x'-y'| + |\widetilde{\varepsilon}|)^{1-\theta}e^{-c|x'-y'|}\\
				&\leq  C \chi_{\R^{\theta-1}\setminus B(0,1/2)}\bigl(\tfrac{x'-y'}{a_\varepsilon}\bigr) |x'-y'|^{1-\theta}e^{-c|x'-y'|},
			\end{aligned}
		\end{equation*}
		where $c = \tfrac{C_6}{C_4}$ with $C_4>0$ from \eqref{eq_est_x_Sigma} and $C_6$ from Lemma~\ref{lem_est_G_z}. Next, we estimate the derivatives with respect to $x'$ of $k$. The derivative with respect to $x_l'$, $l \in \{1,\dots\theta-1\}$, is given for for $x'\neq y' \in \R^{\theta-1}$ by
		\begin{equation*}
			\begin{aligned}
				\frac{d}{d x'_l}k(x',y')&= \sum_{j = 1}^{\theta} \Big((\partial_j G_z)(x_{\zeta,\kappa}(x') -x_{\zeta,\kappa}(y') + \widetilde{\varepsilon}\nu_{\zeta,\kappa}(x'))\\
				&\hspace{50 pt}\cdot\frac{d}{d x'_l}(x_{\zeta,\kappa}(x')[j] + \widetilde{\varepsilon}\nu_{\zeta,\kappa}(x')[j])
				\omega\bigl(\tfrac{x' - y'}{a_\varepsilon}\bigr)\sqrt{1+ \abs{\nabla \zeta(y')}^2} \Big)\\
				&+ G_z(x_{\zeta,\kappa}(x') -x_{\zeta,\kappa}(y') + \widetilde{\varepsilon}\nu_{\zeta,\kappa}(x'))\frac{1}{a_\varepsilon}(\partial_{l} \omega)\bigl(\tfrac{x'-y'}{a_\varepsilon}\bigr)\sqrt{1+ \abs{\nabla \zeta(y')}^2},
			\end{aligned}
		\end{equation*}
		where $v[j]$ denotes the $j$-th component of a vector $v$.
		Applying \eqref{eq_G_z_est_R_theta_prime}, the properties of $\omega$ and $\zeta$ again, we can estimate for $x' \neq y' \in \R^{\theta-1}$
		\begin{equation*}
			\begin{aligned}
				\Bigl|\frac{d}{d x'_l}k(x',y')\Bigr| &\leq  C \chi_{\R^{\theta-1}\setminus B(0,1/2)}\bigl(\tfrac{x'-y'}{a_\varepsilon}\bigr)\Bigl(  (|x'-y'| + |\widetilde{\varepsilon}|)^{-\theta}e^{-c|x'-y'|} \\
				&\hspace{110 pt} +  \frac{1}{a_\varepsilon}(|x'-y'| + |\widetilde{\varepsilon}|)^{1-\theta}e^{-c|x'-y'|} \Bigr)\\
				&\leq  C \chi_{\R^{\theta-1}\setminus B(0,1/2)}\bigl(\tfrac{x'-y'}{a_\varepsilon}\bigr)\Bigl(  |x'-y'|^{-\theta} +  \frac{1}{a_\varepsilon}|x'-y'|^{1-\theta} \Bigr)e^{-c|x'-y'|}.
			\end{aligned}
		\end{equation*}
		Thus, if we set $\widetilde{k}(z'):=  C \chi_{\R^{\theta-1}\setminus B(0,1/2)}\bigl(\tfrac{z'}{a_\varepsilon}\bigr)\Bigl(  |z'|^{-\theta} +  \tfrac{1}{a_\varepsilon}|z'|^{1-\theta} \Bigr)e^{-c|z'|}$ for $z' \in \R^{\theta-1}\setminus \{0\}$ we get
		\begin{equation}\label{eq_est_k_tilde}
			|k(x',y')|,\, \sum_{l=1}^{\theta-1} \Bigl|\frac{d}{d x'_l} k(x',y')\Bigr| \leq \widetilde{k}(x'-y'),  \quad x' \neq y' \in \R^{\theta-1}.
		\end{equation}
		Hence, by Lemma~\ref{lem_Schur_test} the map $e^{a_\varepsilon}_{\widetilde{\varepsilon}}(z)$ acts as a bounded operator from $L^2(\R^{\theta-1};\C^N)$ to $H^1(\R^{\theta-1};\C^N)$ and 
		\begin{equation*}
			\norm{e^{a_\varepsilon}_{\widetilde{\varepsilon}}(z)}_{L^2(\R^{\theta-1};\C^N) \to H^1(\R^{\theta-1};\C^N)} \leq C \|\widetilde{k}\|_{L^1(\R^{\theta-1})}. 
		\end{equation*}
		Now, the norm estimate in the assertion follows from
		\begin{equation*}
			\begin{aligned}
				\|\widetilde{k}\|_{L^1(\R^{\theta-1})} &= C \int_{\R^{\theta-1}}  \chi_{\R^{\theta-1}\setminus B(0,1/2)}\bigl(\tfrac{z'}{a_\varepsilon}\bigr)\Big(  |z'|^{-\theta} +  \frac{1}{a_\varepsilon}|z'|^{1-\theta} \Big)e^{-c|z'|} dz' \\
				&\leq C \int_{a_\varepsilon/2}^\infty \Big(r^{-\theta} + \frac{1}{a_\varepsilon}r^{1-\theta}\Big)e^{-cr} r^{\theta-2} \, dr\\
				&\leq C \bigg(\frac{1}{a_\varepsilon} + \frac{1+\abs{\log(a_\varepsilon)}}{a_\varepsilon}\bigg) \leq C \frac{1+\abs{\log(\varepsilon)}}{a_\varepsilon}.
			\end{aligned}
		\end{equation*}
	
		Finally, we prove the continuity. To do so, let $\widetilde{\varepsilon} \in (-2\varepsilon_2,2\varepsilon_2) \setminus \{0\}$ and $(\widetilde{\varepsilon}_n)_{n \in \N} $ be a sequence such that  $\widetilde{\varepsilon}_n \in (-2\varepsilon_2,2\varepsilon_2) \setminus \{0\}$ for all $n \in \N$ and  $\widetilde{\varepsilon}_n \overset{n \to \infty}\longrightarrow  \widetilde{\varepsilon}$. Using the dominated convergence theorem and \eqref{eq_est_k_tilde} shows that for $f \in L^2(\R^{\theta-1};\C^N)$ $e^{a_\varepsilon}_{\widetilde{\varepsilon}_n}f$ and $\partial_l e^{a_\varepsilon}_{\widetilde{\varepsilon}_n}f$, $l \in \{1,\dots,\theta-1\}$, converge pointwise to $e^{a_\varepsilon}_{\widetilde{\varepsilon}}f$ and $\partial_l e^{a_\varepsilon}_{\widetilde{\varepsilon}}f$, ${l \in \{1,\dots,\theta-1\}}$, respectively. Furthermore, \eqref{eq_est_k_tilde} shows that $|e^{a_\varepsilon}_{\widetilde{\varepsilon}_n}f|$ and $|\partial_l e^{a_\varepsilon}_{\widetilde{\varepsilon}_n}f|$, $l \in \{1,\dots,\theta-1\}$, are independently of $n \in \N$ pointwise  bounded by the function $|f| * \widetilde{k}$ which is by Young's inequality square integrable as $ f \in L^2(\R^{\theta-1};\C^N)$ and $\widetilde{k} \in L^1(\R^{\theta-1})$. Hence, applying the dominated converge theorem again shows that $e^{a_\varepsilon}_{\widetilde{\varepsilon}_n}f$ converges to $e^{a_\varepsilon}_{\widetilde{\varepsilon}}f$ in $H^1(\R^{\theta-1};\C^N)$.
	\end{proof}
	
	\begin{lemma}\label{lem_comuutator_d_eps}
		Let $z \in \rho(H)$, $\psi \in C^{1}_b(\R^{\theta-1})$, $x_0' \in \R^{\theta-1}$,  $\zeta_{x_0'} $ be as in \eqref{eq_zeta_x_0'}, and $\widetilde{\varepsilon}  \in (- 2\varepsilon_2,2\varepsilon_2)\setminus \{0\}$. Then,
		\begin{equation*}
			\big\|[d^{\zeta_{x_0'},\kappa}_{\widetilde{\varepsilon}}(z),\psi]\big\|_{L^2(\R^{\theta-1};\C^N) \to H^1(\R^{\theta-1};\C^N)} \leq C \norm{\psi}_{W^1_\infty(\R^{\theta-1})} (1+\abs{\log\abs{\widetilde{\varepsilon}}}),
		\end{equation*}
		where $C>0$ does not depend on $\widetilde{\varepsilon}\neq0$ and $x_0' \in \R^{\theta-1}$. Moreover, for $f \in L^2(\R^{\theta-1};\C^N)$ the mapping $(-2\varepsilon_2,2\varepsilon_2)\setminus\{0\}\ni \widetilde{\varepsilon} \mapsto [d^{\zeta_{x_0'},\kappa}_{\widetilde{\varepsilon}}(z),\psi]f \in H^1(\R^{\theta-1};\C^N)$ is continuous.
	\end{lemma}
	\begin{proof}
		We prove this result in the same vein as the previous lemma, i.e. we estimate the integral kernel of $[d^{\zeta_{x_0'},\kappa}_{\widetilde{\varepsilon}}(z),\psi]$ and its partial derivatives, and apply Lemma~\ref{lem_Schur_test}. The integral kernel of $[d^{\zeta_{x_0'},\kappa}_{\widetilde{\varepsilon}}(z),\psi]$ is given for $x',y' \in \R^{\theta-1}$ by 
		\begin{equation*}
			k(x',y'):= G_z(x_{\zeta_{x_0'},\kappa}(x') -x_{\zeta_{x_0'},\kappa}(y') + \widetilde{\varepsilon}\nu_{\zeta_{x_0'},\kappa}(x'))\sqrt{1+ \abs{\nabla \zeta_{x_0'}(y')}^2}(\psi(y')- \psi(x')).
		\end{equation*}
		Using 
		\begin{equation}\label{eq_lem_comm_d_eps_1}
			\begin{aligned}
				x_{\zeta_{x_0'},\kappa}(x') &= \kappa(x',\zeta(x_0')+\langle\nabla\zeta(x_0'),x'-x_0'\rangle), \\
				\nu_{\zeta_{x_0'},\kappa}(x') &= \frac{\kappa(-\nabla\zeta(x_0'),1)}{\sqrt{1+|\nabla \zeta(x_0')|^2}}  = \nu_{\zeta,\kappa}(x_0'),\\
				\nabla \zeta_{x_0'}(x') &= \nabla \zeta(x_0')
			\end{aligned}
		\end{equation}
		for $x' \in \R^{\theta-1}$ shows that $k$ can be simplified to
		\begin{equation*}
			k(x',y')= G_z\bigl(\kappa(x'-y', \langle \nabla \zeta(x_0'), x'-y' \rangle) + \widetilde{\varepsilon}\nu_{\zeta,\kappa}(x'_0)\bigr)\sqrt{1+ \abs{\nabla \zeta(x_0')}^2}(\psi(y')- \psi(x')). 
		\end{equation*}
		Moreover, with \eqref{eq_lem_comm_d_eps_1}, $\kappa \in \textup{SO}(\theta)$, and the Pythagorean theorem one gets
		\begin{equation*}
			\begin{aligned}
				|\kappa(x'-y', \langle \nabla \zeta(x_0'), x'-y' \rangle) &+ \widetilde{\varepsilon}\nu_{\zeta,\kappa}(x'_0)|^2 \\
				&= |x'-y'|^2 + \langle \nabla \zeta(x_0'), x'-y' \rangle^2 + \widetilde{\varepsilon}^2 \\
				&\leq |x'-y'|^2(1+\|\nabla \zeta\|_{L^\infty(\R^{\theta-1};\R^{\theta-1})}^2) + \widetilde{\varepsilon}^2. 
			\end{aligned}
		\end{equation*}
		In particular, we can choose $C_4'>0$  which does not depend on $x_0' $ and $\widetilde{\varepsilon}$ such that 
		\begin{equation}\label{eq_lem_comm_d_eps_2}
			\begin{aligned}
				(C_4')^{-1}(|x'-y'| + |\widetilde{\varepsilon}|) &\leq |\kappa(x'-y', \langle \nabla \zeta(x_0'), x'-y' \rangle) + \widetilde{\varepsilon}\nu_{\zeta,\kappa}(x'_0)| \\ 
				&\leq  	C_4'(|x'-y'| + |\widetilde{\varepsilon}|). 
			\end{aligned}
		\end{equation}
		Then, \eqref{eq_G_z_est_R_theta},  \eqref{eq_lem_comm_d_eps_2}, the Lipschitz continuity of $\psi \in C^1_b(\R^{\theta-1})$ and $\zeta \in C^2_b(\R^{\theta-1};\R)$ yield
		\begin{equation*}
			\begin{aligned}
				|k(x',y')| &\leq C (|x'-y'| + |\widetilde{\varepsilon}|)^{1-\theta}e^{-c'|x'-y'|} \norm{\psi}_{W^1_\infty(\R^{\theta-1})} |x'-y’|\\
				&\leq C\norm{\psi}_{W^1_\infty(\R^{\theta-1})}(|x'-y'| + |\widetilde{\varepsilon}|)^{2-\theta}e^{-c'|x'-y'|}, \quad  x',y' \in \R^{\theta-1},
			\end{aligned}
		\end{equation*} 
		where $c' = \tfrac{C_6}{C_4'}$ with $C_6$ from Lemma~\ref{lem_est_G_z}  and $C>0$ is independent of $x_0'$ and $\widetilde{\varepsilon}$.
		The derivative with respect to $x_l'$, $x' \in \R^{\theta-1}$, of $k$ is given by
		\begin{equation*}
			\begin{aligned}
				\frac{d}{d x'_l}k(x',y')&= \biggl(\sum_{j = 1}^{\theta} (\partial_j G_z)\bigl(\kappa(x'-y', \langle \nabla \zeta(x_0'), x'-y' \rangle) + \widetilde{\varepsilon}\nu_{\zeta,\kappa}(x'_0)\bigr)\\
				&\hspace{100 pt} \cdot \bigl(\kappa(e_l',\partial_l \zeta(x_0'))\bigr)[j](\psi(y') - \psi(x'))\\
				&\hspace{10 pt}-G_z\bigl(\kappa(x'-y', \langle \nabla \zeta(x_0'), x'-y' \rangle) + \widetilde{\varepsilon}\nu_{\zeta,\kappa}(x'_0)\bigr)(\partial_{l}\psi)(x') \biggr)\\
				&\hspace{180 pt}\cdot\sqrt{1+ \abs{\nabla \zeta(x_0')}^2}, 
			\end{aligned}
		\end{equation*}
		where $e_l'$ denotes the $l$-th Euclidean unit vector in $\R^{\theta-1}$ and $\bigl(\kappa(e_l',\partial_l \zeta(x_0'))\bigr)[j]$ denotes the $j$-th entry of the vector $\kappa(e_l',\partial_l \zeta(x_0'))$.
		Using  \eqref{eq_G_z_est_R_theta}, \eqref{eq_lem_comm_d_eps_2}, the Lipschitz continuity of $\psi \in C^1_b(\R^{\theta-1})$ and $\zeta \in C^2_b(\R^{\theta-1};\R)$ again gives us
		\begin{equation*}
			\begin{aligned}
				\Bigl|\frac{d}{d x'_l}k(x',y')\Bigr| &\leq C\norm{\psi}_{W^1_\infty(\R^{\theta-1})}\Big( (|x'-y'| + |\widetilde{\varepsilon}|)^{-\theta}e^{-c'|x'-y'|} |x'-y’|\\
				&\hspace{100 pt}+ (|x'-y'| + |\widetilde{\varepsilon}|)^{1-\theta}e^{-c'|x'-y'|}  \Big)\\
				&\leq C\norm{\psi}_{W^1_\infty(\R^{\theta-1})}(|x'-y'| + |\widetilde{\varepsilon}|)^{1-\theta}e^{-c'|x'-y'|} 
			\end{aligned}
		\end{equation*}
		for all $x',y' \in \R^{\theta-1}$, where $C>0$ can again be chosen independently of $x_0'$ and $\widetilde{\varepsilon}$.
		Setting for $z' \in \R^{\theta-1}$
		\begin{equation*}
			\widetilde{k}(z'):=  C \norm{\psi}_{W^1_\infty(\R^{\theta-1})}\big((|z'| + |\widetilde{\varepsilon}|)^{2-\theta}+ (|z'| + |\widetilde{\varepsilon}|)^{1-\theta}\big)e^{-c'|z'|}
		\end{equation*} 
		leads to
		\begin{equation*}
			|k(x',y')|, \, \sum_{l=1}^{\theta-1} \Bigl|\frac{d}{d x'_l}k(x',y')\Bigr| \leq \widetilde{k}(x'-y'),  \quad  x' , y' \in \R^{\theta-1}.
		\end{equation*}
		Now, Lemma~\ref{lem_Schur_test} shows
		
		\begin{equation*}
			\begin{aligned}
				\big\|[d^{\zeta_{x_0'},\kappa}_{\widetilde{\varepsilon}}(z)&,\psi]\big\|_{L^2(\R^{\theta-1};\C^N) \to H^1(\R^{\theta-1};\C^N)} \leq C \int_{\R^{\theta-1}} \widetilde{k}(z') \, dz'\\
				& \leq C\norm{\psi}_{W^1_\infty(\R^{\theta-1})} \int_0^\infty \big((r + \abs{\widetilde{\varepsilon}})^{2-\theta}+ (r+ \abs{\widetilde{\varepsilon}})^{1-\theta}\big)e^{-c'r} r^{\theta-2} \, dr\\
				&\leq C\norm{\psi}_{W^1_\infty(\R^{\theta-1})} \int_0^\infty \big(1+ (r+ |\widetilde{\varepsilon}|)^{-1}\big)e^{-c'r}  \, dr\\
				&\leq C\norm{\psi}_{W^1_\infty(\R^{\theta-1})}(1+ \abs{ \log\abs{\widetilde{\varepsilon}}}).
			\end{aligned}
		\end{equation*}
		The assertion regarding the continuity can be proven in a similar way as in Lemma~\ref{lem_e_eps^r}.
	\end{proof}
	
	\begin{lemma}\label{lem_diff_d}
		Let $z \in \rho(H)$, $x_0' \in \R^{\theta -1}$, $\zeta \in C^2_b(\R^{\theta-1};\R)$ be as described at the beginning of Section~\ref{sec_loc_part}, $\zeta_{x_0'} $ be as in \eqref{eq_zeta_x_0'}, $\varepsilon \in (0,\varepsilon_2)$, $a_\varepsilon = \varepsilon^{1/6}$, and $ \widetilde{\varepsilon}  \in (- 2\varepsilon_2,2\varepsilon_2)\setminus \{0\}$.  Then, there exists a $\delta_3 = \delta_3(\zeta) \in (0,\varepsilon_2)$ such that for all  $\varepsilon \in (0,\delta_3)$ the inequality
		\begin{equation*}
			\big\|\chi_{B(x_0',3a_\varepsilon)}\big(d^{\zeta,\kappa}_{\widetilde{\varepsilon}}(z)- d^{\zeta_{x_0'},\kappa}_{\widetilde{\varepsilon}}(z)  \big)\chi_{B(x_0',3a_\varepsilon)}\big\|_{L^2(\R^{\theta-1};\C^N) \to L^2(\R^{\theta-1};\C^N)} \leq C a_\varepsilon (1+\abs{\log\abs{\widetilde{\varepsilon}}})
		\end{equation*}
		holds,
		where $C>0$ does not depend on $\varepsilon$, $\widetilde{\varepsilon}$, and $x_0'$.
	\end{lemma}
	\begin{proof}
		We prove this statement by estimating the integral kernel of  the operator
		\begin{equation*}
			\chi_{B(x_0',a_\varepsilon)}\big( d^{\zeta,\kappa}_{\widetilde{\varepsilon}}(z) - d^{\zeta_{x_0'},\kappa}_{\widetilde{\varepsilon}}(z) \big)\chi_{B(x_0',a_\varepsilon)}
		\end{equation*} and applying Lemma \ref{lem_Schur_test}. The mentioned integral kernel is given for $x' , y' \in \R^{\theta-1}$ by
		\begin{equation*}
			\begin{aligned}
				k(x',y') :=& \chi_{B(x_0',3a_\varepsilon)}(x')\bigg(G_z(x_{\zeta,\kappa}(x') -x_{\zeta,\kappa}(y') + \widetilde{\varepsilon}\nu_{\zeta,\kappa}(x'))\sqrt{1+\abs{\nabla \zeta(y')}^2} \\
				&-G_z(\kappa(x'-y',\langle\nabla \zeta(x_0') , x'-y'\rangle) + \widetilde{\varepsilon}\nu_{\zeta,\kappa}(x_0'))\sqrt{1+\abs{\nabla \zeta(x_0')}^2}\bigg)\\
				&\hspace{200 pt} \cdot  \chi_{B(x_0',3a_\varepsilon)}(y').
			\end{aligned}
		\end{equation*}
		If $x' \not\in B(x_0',3a_\varepsilon)$ or $y' \not\in B(x_0',3a_\varepsilon)$, then $k(x',y') =0$. Thus, we assume from now on $x',y'\in B(x_0',3a_\varepsilon)$. 
		Using the mean value theorem for matrix-valued functions, cf. \cite[Lemma A.1]{BHS23}, \eqref{eq_G_z_est_R_theta}, and $\zeta \in C^2_b(\R^{\theta-1};\R)$ we find
		\begin{equation}\label{eq_d_diff_ker_1}
			\begin{aligned}
				|k(x',y')| &\leq \sqrt{\theta}\sup_{\upsilon \in [0,1],j\in \{1,\dots \theta\}} |\partial_j G_z(w_\upsilon)\sqrt{1+|\nabla \zeta(y')|^2} | |w_1 - w_0|\\
				&\hspace{50 pt} + \Bigl| G_z(w_0) \bigl(\sqrt{1+|\nabla \zeta(y')|^2} -\sqrt{1+|\nabla \zeta(x_0')|^2}\bigr)\Bigr| \\
				&\leq C 	\Big(\sup_{\upsilon \in [0,1]}|w_\upsilon|^{-\theta} |w_1 - w_0| + |w_0|^{1-\theta} |x_0'- y'|\Big)
			\end{aligned}
		\end{equation}
		with
		\begin{equation*}
			w_\upsilon =\upsilon\big(x_{\zeta,\kappa}(x') -x_{\zeta,\kappa}(y') 
			+ \widetilde{\varepsilon}\nu_{\zeta,\kappa}(x')\big) 
			+(1-\upsilon) \big( \kappa(x'-y',\langle\nabla \zeta(x_0'),x'-y'\rangle) + \widetilde{\varepsilon}\nu_{\zeta,\kappa}(x_0')\big)
		\end{equation*}
		for $\upsilon \in [0,1]$. Note that the mean value theorem is applicable since $w_\upsilon  \neq 0$ for $\upsilon \in [0,1]$, see \eqref{eq_d_diff_ker_3} below. 
		 Next, we want to estimate
		\begin{equation*}
			\begin{aligned}
				|w_1 - &w_0| \\
				&= |x_{\zeta,\kappa}(x') -x_{\zeta,\kappa}(y') 
				+ \widetilde{\varepsilon}\nu_{\zeta,\kappa}(x') -   \kappa(x'-y',\langle\nabla \zeta(x_0'),x'-y'\rangle) - \widetilde{\varepsilon}\nu_{\zeta,\kappa}(x_0') |\\
				&=|\kappa(x'-y',\zeta(x')-\zeta(y'))\\
				&\hspace{70 pt}-   \kappa(x'-y',\langle\nabla \zeta(x_0'),x'-y'\rangle) + \widetilde{\varepsilon}(\nu_{\zeta,\kappa}(x') - \nu_{\zeta,\kappa}(x_0'))  |\\
				&=|\kappa(0,\zeta(x')-\zeta(y') - \langle \nabla \zeta(x_0'),x'-y'\rangle) + \widetilde{\varepsilon}(\nu_{\zeta,\kappa}(x') - \nu_{\zeta,\kappa}(x_0') ) |.
			\end{aligned}
		\end{equation*}
		As $\zeta \in C^2_b(\R^{\theta-1};\R)$ and $\kappa \in \textup{SO}(\theta)$, there exists a $K=K(\zeta)>0$ such that
		\begin{equation*}
			\abs{\nu_{\zeta,\kappa}(x') - \nu_{\zeta,\kappa}(x_0')}  = \bigg| \frac{\kappa(-\nabla \zeta(x'),1)}{\sqrt{1+ |\nabla\zeta(x')|^2}}  - \frac{\kappa(-\nabla \zeta(x_0'),1)}{\sqrt{1+ |\nabla\zeta(x_0')|^2}} \bigg|\leq K \abs{x' - x_0'} \leq 3K a_\varepsilon
		\end{equation*} 
		and 
		\begin{equation*}
			\begin{aligned}
				|\kappa(0,\zeta(x') -\zeta(y') - \langle&\nabla \zeta(x_0'),x'-y' \rangle ) | \\
				&= | \zeta(x') -\zeta(y') - \langle \nabla \zeta(x_0'),x'-y'\rangle| \\
				&= \bigg|  \int_{0}^1 \langle \nabla\zeta(y' + t(x'-y')) - \nabla\zeta(x_0'), x'-y' \rangle \,dt\bigg|\\
				&\leq K\bigg|  \int_{0}^1 |t (x'- x_0') + (1-t)(y'-x_0')| |x'-y'|  \,dt\bigg| \\
				& \leq 3K a_\varepsilon |x'-y'|,
			\end{aligned}
		\end{equation*}
		where we used $x',y' \in B(x_0',3a_\varepsilon)$. Hence, if $\delta_3=\delta_3(\zeta) \in (0,\varepsilon_2)$ is chosen sufficiently small, then  for all $a_\varepsilon \in (0,\delta_3^{1/6})$ the inequality
		\begin{equation}\label{eq_d_diff_ker_2}
			|w_1 - w_0| \leq K3 a_\varepsilon (|x'-y'|+|\widetilde{\varepsilon}|) \leq \frac{1}{2 C_4}(|x'-y'|+|\widetilde{\varepsilon}|)
		\end{equation}
		holds with $C_4>0$ from \eqref{eq_est_x_Sigma}.
		Therefore, we can  use \eqref{eq_est_x_Sigma} to estimate $|w_\upsilon|$, $\upsilon \in [0,1]$,  from below by
		\begin{equation}\label{eq_d_diff_ker_3}
			\begin{aligned}
				|w_\upsilon| &= |\upsilon w_1 + (1-\upsilon)w_0| = |w_1+ (1-\upsilon) (w_0-w_1)| \geq |w_1| - |w_1-w_0| \\
				&\geq  \frac{1}{C_4}(|x'-y'| + |\widetilde{\varepsilon}|) - \frac{1}{2 C_4}(|x'-y'|+ |\widetilde{\varepsilon}|) = \frac{1}{2 C_4}(|x'-y'|+|\widetilde{\varepsilon}|).
			\end{aligned}
		\end{equation}
		Thus, plugging \eqref{eq_d_diff_ker_2} and \eqref{eq_d_diff_ker_3} into \eqref{eq_d_diff_ker_1} yields for $a_\varepsilon = \varepsilon^{1/6}$ with  $\varepsilon \in (0,\delta_3)$ 
		\begin{equation*}
			\abs{k(x',y')} \leq \begin{cases} 0 & \textup{if } x' \not \in B(x_0',3 a_\varepsilon) \textup{ or } y' \not \in B(x_0', 3a_\varepsilon),\\
				C a_\varepsilon (\abs{x'-y'} + \abs{\widetilde{\varepsilon}})^{1-\theta} &\textup{else}.
			\end{cases}
		\end{equation*}
		Applying Lemma \ref{lem_Schur_test} yields
		\begin{equation*}
			\begin{aligned}
				\big\|\chi_{B(x_0',3a_\varepsilon)}\big(d^{\zeta,\kappa}_{\widetilde{\varepsilon}}(z)-	 d^{\zeta_{x_0'},\kappa}_{\widetilde{\varepsilon}}(z)  \big)&\chi_{B(x_0',3a_\varepsilon)}\big\|_{L^2(\R^{\theta-1};\C^N) \to L^2(\R^{\theta-1};\C^N)} \\
				&\leq C a_\varepsilon \int_{B(0,6a_\varepsilon)}  (\abs{z'} + \abs{\widetilde{\varepsilon}})^{1-\theta} \, dz'\\
				& \leq C a_\varepsilon \int_0^{6a_\varepsilon} (r + \abs{\widetilde{\varepsilon}})^{1-\theta} r^{\theta-2} \, dr\\
				&\leq C a_\varepsilon \int_0^{6a_\varepsilon}  (r + \abs{\widetilde{\varepsilon}})^{-1} \,dr\\
				&\leq C a_\varepsilon(1+ \abs{\log\abs{\widetilde{\varepsilon}}}).
			\end{aligned}
		\end{equation*}
	\end{proof}

	\begin{cor}\label{cor_diff_d}
		Let $z \in \rho(H)$, $x_0' \in \R^{\theta -1}$, $\zeta \in C^2_b(\R^{\theta-1};\R)$ be as described at the beginning of Section~\ref{sec_loc_part}, $\zeta_{x_0'} $ be as in \eqref{eq_zeta_x_0'}, $\varepsilon \in (0,\delta_3)$ with $\delta_3$  chosen as in Lemma~\ref{lem_diff_d}, $a_\varepsilon = \varepsilon^{1/6} $, $ \widetilde{\varepsilon}  \in (- 2\varepsilon_2,2\varepsilon_2)\setminus \{0\}$, and $Q_{\eta,\tau}^{\zeta,\kappa}$ be as in \eqref{eq_Q_matrix}. Then,
		\begin{equation*}
			\begin{aligned}
				\big\|&\chi_{B(x_0',3a_\varepsilon)}\big( 
				d^{\zeta, \kappa}_{\widetilde{\varepsilon}}(z)Q_{\eta,\tau}^{\zeta,\kappa} \\
				&\hspace{7 pt}-d_{\widetilde{\varepsilon}}^{\zeta_{x_0'},\kappa}(z)Q_{\eta,\tau}^{\zeta,\kappa} (x_0')\big)\chi_{B(x_0',3a_\varepsilon)}\big\|_{L^2(\R^{\theta-1};\C^N) \to 	L^2(\R^{\theta-1};\C^N)} \leq C a_\varepsilon(1+\abs{\log\abs{\widetilde{\varepsilon}}}),
			\end{aligned}
		\end{equation*}
		where $C$ does not depend on $\varepsilon$, $\widetilde{\varepsilon}$, and $x_0'$. 
	\end{cor}
	\begin{proof}
		Lemma~\ref{lem_diff_d} and  $Q_{\eta,\tau}^{\zeta,\kappa}  \in C^1_b(\R^{\theta-1};\C^{N\times N})$ yield
		\begin{equation*}
			\begin{aligned}
				\big\|&\chi_{B(x_0',3a_\varepsilon)}\big( 
				d^{\zeta, \kappa}_{\widetilde{\varepsilon}}(z)Q_{\eta,\tau}^{\zeta,\kappa} \\
				&\hspace{80 pt}-d_{\widetilde{\varepsilon}}^{\zeta_{x_0'},\kappa}(z)Q_{\eta,\tau}^{\zeta,\kappa} (x_0')\big)\chi_{B(x_0',3a_\varepsilon)}\big\|_{L^2(\R^{\theta-1};\C^N) \to 	L^2(\R^{\theta-1};\C^N)}\\
				&\leq 	\big\|\chi_{B(x_0',3a_\varepsilon)} 
				d^{\zeta, \kappa}_{\widetilde{\varepsilon}}(z)\big(Q_{\eta,\tau}^{\zeta,\kappa} - Q_{\eta,\tau}^{\zeta,\kappa}(x_0') \big)\chi_{B(x_0',3a_\varepsilon)} \big\|_{L^2(\R^{\theta-1};\C^N) \to 	L^2(\R^{\theta-1};\C^N)} \\
				&\hspace{10pt}+\big\|\chi_{B(x_0',3a_\varepsilon)}\big(	d^{\zeta, \kappa}_{\widetilde{\varepsilon}}(z)\\
				&\hspace{80 pt}-  d_{\widetilde{\varepsilon}}^{\zeta_{x_0'},\kappa}(z)\big)\chi_{B(x_0',3a_\varepsilon)} Q_{\eta,\tau}^{\zeta,\kappa} (x_0')\big\|_{L^2(\R^{\theta-1};\C^N) \to 	L^2(\R^{\theta-1};\C^N)}\\
				& \leq C \bigl(\|d^{\zeta, \kappa}_{\widetilde{\varepsilon}}(z) \|_{L^2(\R^{\theta-1};\C^N) \to L^2(\R^{\theta-1};\C^N)} a_\varepsilon + a_\varepsilon(1+\abs{\log\abs{\widetilde{\varepsilon}}})\bigr),
			\end{aligned}
		\end{equation*}
		where $C$ does not depend on $\varepsilon$, $\widetilde{\varepsilon}$, and $x_0'$. Moreover,   \eqref{eq_d_eps}, \eqref{eq_G_z_est_R_theta_prime},  and $\zeta \in C^2_b(\R^{\theta-1};\R)$ imply for the integral kernel $k$ of $d_{\widetilde{\varepsilon}}^{\zeta,\kappa}(z)$ 
		\begin{equation*}
			|k(x',y')| \leq C (|x'-y'|+|\widetilde{\varepsilon}|)^{1-\theta}e^{-c|x'-y'|}, \quad x',y' \in \R^{\theta-1},
		\end{equation*}
		where $c = \tfrac{C_6}{C_4}>0$ with $C_4$ from \eqref{eq_est_x_Sigma} and $C_6$ from \eqref{eq_G_z_est_R_theta_prime}, 
		and therefore Lemma~\ref{lem_Schur_test} implies
		\begin{equation*}
			\begin{aligned}
				\|d^{\zeta, \kappa}_{\widetilde{\varepsilon}}(z) \|_{L^2(\R^{\theta-1};\C^N) \to L^2(\R^{\theta-1};\C^N)} &\leq C \int_{\R^{\theta-1}} (|z'| + |\widetilde{\varepsilon}|)^{1-\theta} e^{-c|z'|} \, dz'\\
				&\leq C \int_{0}^\infty (r + |\widetilde{\varepsilon}|)^{1-\theta} e^{-cr} r^{\theta-2} \,dr\\
				&\leq C (1+\abs{\log\abs{\widetilde{\varepsilon}}}),
			\end{aligned}
		\end{equation*}
		which completes the proof.
	\end{proof}
	
	Now, we transfer in Proposition~\ref{prop_local_prop} the results regarding  $e_{\widetilde{\varepsilon}}^{a_\varepsilon}(z)$ and $d^{\zeta,\kappa}_{\widetilde{\varepsilon}}(z)$  to $E_\varepsilon(z)$ and $D^{\zeta,\kappa}_\varepsilon(z)$. The first, second, and third estimate in Proposition~\ref{prop_local_prop} are consequences of Lemma~\ref{lem_e_eps^r}, Lemma~\ref{lem_comuutator_d_eps}, and Corollary~\ref{cor_diff_d}, respectively.
	Recall that $\zeta \in C^2_b(\R^{\theta-1};\R)$ is described in the beginning of Section~\ref{sec_loc_part}, $\zeta_{x_0'} $ is defined by~\eqref{eq_zeta_x_0'}, and $\delta_3$ is chosen as in  Lemma~\ref{lem_diff_d}.
	
	\begin{proposition}\label{prop_local_prop}
		Let $z \in \rho(H)$, $x_0' \in \R^{\theta -1}$,  $\varepsilon \in (0,\delta_3)$, $a_\varepsilon = \varepsilon^{1/6} $, $Q_{\eta,\tau}^{\zeta,\kappa}$ be as in \eqref{eq_Q_matrix}, and $\psi \in C^1_b(\R^{\theta-1})$. Then, the operators  $E_{\varepsilon}(z)$  and $[D_{\varepsilon}^{ \zeta, \kappa}(z),\psi]$ act as bounded operators from $\mathcal{B}^0(\R^{\theta-1})$ to $\mathcal{B}^1(\R^{\theta-1})$ and 
		\begin{equation*}
			\begin{aligned}
				&\|E_{\varepsilon}(z)\|_{0\to1} \leq C\frac{1+\abs{\log(\varepsilon)}}{a_\varepsilon},\\
				&\|[D_{\varepsilon}^{ \zeta_{x_0'}, \kappa}(z),\psi]\|_{0\to1}\leq C \norm{\psi}_{W^1_\infty(\R^{\theta-1})} (1+\abs{\log(\varepsilon)}),\\
				&\big\|\chi_{B(x_0',3a_\varepsilon)}\big(D_{\varepsilon}^{ \zeta, \kappa}(z)Q_{\eta,\tau}^{\zeta,\kappa}- D_{\varepsilon}^{ \zeta_{x_0'}, \kappa}(z)Q_{\eta,\tau}^{\zeta,\kappa}(x_0')\big)\chi_{B(x_0',3a_\varepsilon)}\big\|_{0 \to 0} \leq C a_\varepsilon (1+\abs{\log(\varepsilon)}),
			\end{aligned}
		\end{equation*}
		where $C>0$ does not depend on $x_0'$ and $\varepsilon$.
	\end{proposition}
	\begin{proof}
		We start by showing that $[D_{\varepsilon}^{ \zeta_{x_0'}, \kappa}(z),\psi]$ is well-defined as an operator from $\mathcal{B}^0(\R^{\theta-1})$ to $\mathcal{B}^1(\R^{\theta-1})$. Let $f \in \mathcal{B}^0(\R^{\theta-1})$ and set $g = [D_{\varepsilon}^{ \zeta_{x_0'}, \kappa}(z),\psi]f$. It follows from \eqref{eq_D_eps_int_rep} that  $g$ has for $ t \in (-1,1)$ the representation 
		\begin{equation}\label{eq_g}
			g(t)= \int_{-1}^1 [d_{\varepsilon(t-s)}^{\zeta_{x_0'},\kappa}(z),\psi]f(s)\,ds.
		\end{equation}
		Since $[d_{\varepsilon(t-s)}^{\zeta_{x_0'},\kappa}(z),\psi]$ has the continuity property from Lemma~\ref{lem_comuutator_d_eps} and $f \in \mathcal{B}^0(\R^{\theta-1})$, \cite[Proposition~1.1.28]{HNVW16}  implies that the function
		\begin{equation*}
			(-1,1)\times(-1,1) \ni (t,s) \mapsto [d_{\varepsilon(t-s)}^{\zeta_{x_0'},\kappa}(z),\psi]f(s) \in H^1(\R^{\theta-1};\C^N)
		\end{equation*}
		is measurable. According to Lemma~\ref{lem_comuutator_d_eps}  we have
		\begin{equation}\label{eq_small2big_1}
			\begin{aligned}
				&\int_{-1}^1\bigg(\int_{-1}^1\|[d_{\varepsilon(t-s)}^{\zeta_{x_0'},\kappa}(z),\psi]f(s)\|_{H^1(\R^{\theta-1};\C^N)} \,ds\bigg)^2 \, dt \\
				&\hspace{5 pt}\leq C \|\psi \|_{W^1_\infty(\R^{\theta-1})}^2 \int_{-1}^1 \bigg(\int_{-1}^1(1+\abs{\log\abs{\varepsilon(t-s)}}) \| f(s) \|_{L^2(\R^{\theta-1};\C^N)}\,ds \bigg)^2 \,dt.
			\end{aligned}
		\end{equation}
		This expression can be estimated with the Cauchy-Schwarz inequality and Fubini's theorem by
		\begin{equation}\label{eq_small2big_2}
			\begin{aligned}
				&C\norm{\psi}_{W^1_\infty(\R^{\theta-1})}^2 \int_{-1}^{1} \biggl(\int_{-1}^1 (1+\abs{\log\abs{\varepsilon(t-s)}}) \,ds  \\
				&\hspace{100 pt}\cdot\int_{-1}^1(1+\abs{\log\abs{\varepsilon(t-s)}} ) \norm{f(s)}_{L^2(\R^{\theta-1};\C^N)}^2 \,ds \biggr)  \,dt\\
				& \leq C \norm{\psi}_{W^1_\infty(\R^{\theta-1})}^2 \bigg(\int_{-2}^2 (1+\abs{\log\abs{\varepsilon s}}) \,ds\bigg)^2 \int_{-1}^1 \norm{f(s)}_{L^2(\R^{\theta-1};\C^N)}^2 \,ds\\
				& \leq C \big(\norm{\psi}_{W^1_\infty(\R^{\theta-1})}(1+\abs{\log(\varepsilon)})\norm{f}_0\big)^2.
			\end{aligned}
		\end{equation}
		In particular, applying the Cauchy-Schwarz inequality again gives us
		\begin{equation*}
			\begin{aligned}
				\int_{-1}^1 \int_{-1}^1& \|[d_{\varepsilon(t-s)}^{\zeta_{x_0'},\kappa}(z),\psi]f(s) \|_{H^1(\R^{\theta-1};\C^N)} \,ds \,dt \\
				&\leq \sqrt{2} 	\sqrt{\int_{-1}^1 \bigg(\int_{-1}^1 \| [d_{\varepsilon(t-s)}^{\zeta_{x_0'},\kappa}(z),\psi]f(s) \|_{H^1(\R^{\theta-1};\C^N)} \,ds\bigg)^2 \,dt} < \infty.
			\end{aligned}
		\end{equation*}
		Thus, Fubini's theorem for Bochner integrals, see \cite[Proposition~1.2.7]{HNVW16}, shows that $g(t)$ in~\eqref{eq_g} is  well-defined  and measurable as a function from $(-1,1)$ to $H^1(\R^{\theta-1};\C^N)$. Moreover, \eqref{eq_g}, \eqref{eq_small2big_1}, and \eqref{eq_small2big_2} also give us the norm estimate as
		\begin{equation*}
			\begin{aligned}
				\Vert[D_{\varepsilon}^{ \zeta_{x_0'}, \kappa}(z),\psi]f\Vert_{1}^2 
				&=\norm{g}_{1}^2\\
				 &=\int_{-1}^1 \|g(t)\|^2_{H^1(\R^{\theta-1};\C^N)} \, dt\\
				&\leq \int_{-1}^1\bigg(\int_{-1}^1\|[d_{\varepsilon(t-s)}^{\zeta_{x_0'},\kappa}(z),\psi]f(s)\|_{H^1(\R^{\theta-1};\C^N)} \,ds\bigg)^2 \, dt\\
				& \leq C \big(\norm{\psi}_{W^1_\infty(\R^{\theta-1})}(1+\abs{\log(\varepsilon)}) \norm{f}_0\big)^2.
			\end{aligned}
		\end{equation*}
			The proof of the assertion regarding $E_\varepsilon(z)$ follows along the same lines if one applies Lemma~\ref{lem_e_eps^r} instead of Lemma~\ref{lem_comuutator_d_eps}. Moreover, since    
			\begin{equation*}
				\chi_{B(x_0',3a_\varepsilon)}\big(D_{\varepsilon}^{ \zeta, \kappa}(z)Q_{\eta,\tau}^{\zeta,\kappa}- D_{\varepsilon}^{ \zeta_{x_0'}, \kappa}(z)Q_{\eta,\tau}^{\zeta,\kappa}(x_0')\big)\chi_{B(x_0',3a_\varepsilon)}
			\end{equation*}
		 is only considered as an operator in $\mathcal{B}^0(\R^{\theta-1})$, 
		 one only has to prove the norm estimate in this topology. The norm can be estimated in the same way as we estimated the norm of $E_\varepsilon(z)$ by applying Corollary~\ref{cor_diff_d} instead of Lemma~\ref{lem_comuutator_d_eps}.
	\end{proof}
	
	As the last part of our local analysis we state  a result concerning the inverse of $I+ D_{\varepsilon}^{\zeta_{x_0'},\kappa}(z)Q_{\eta,\tau}^{\zeta,\kappa}(x_0')q$ for $x_0' \in \R^{\theta-1}$. This is an important result as these operators are going to play an essential role when constructing  the inverse of $I+ D_{\varepsilon}^{\zeta,\kappa}(z)Q_{\eta,\tau}^{\zeta,\kappa}q$. Recall that $\zeta \in C^2_b(\R^{\theta-1};\R)$ and $\Sigma$ are as described in the beginning of Section~\ref{sec_loc_part}, $\zeta_{x_0'} $ is given by \eqref{eq_zeta_x_0'}, and $q$ is as in 	\eqref{eq_q2}.
	
	\begin{proposition}\label{prop_uniform_bdd_D_inv}
		Let $z \in \C \setminus \R$, $\eta,\tau \in C^1_b(\Sigma;\R)$, $d = \eta^2- \tau^2$ such that
		\begin{equation*}
			\sup_{x_\Sigma \in \Sigma} d(x_\Sigma) < \frac{\pi^2}{4},
		\end{equation*} 
		and $Q_{\eta,\tau}^{\zeta,\kappa}$ be as in \eqref{eq_Q_matrix}. Then, there exists a $\delta_4 > 0$  such that the operators $(I+D_{\varepsilon}^{ \zeta_{x_0'}, \kappa}(z)Q_{\eta,\tau}^{\zeta,\kappa}(x_0')q)^{-1}$ are uniformly bounded in $\mathcal{B}^0(\R^{\theta-1})$ with respect to $\varepsilon \in (0,\delta_4)$ and $x_0'\in\R^{\theta-1}$.
	\end{proposition}
	\begin{proof}
		Note that
		\begin{equation*}
				Q_{\eta,\tau}^{\zeta,\kappa} (x_0') 
				= \eta(x_{\zeta,\kappa}(x_0')) I_N + \tau(x_{\zeta,\kappa}(x_0')) \beta  =Q_{\eta(x_{\zeta,\kappa}(x_0')),\tau(x_{\zeta,\kappa}(x_0'))}
		\end{equation*}  
		for $x_0' \in \R^{\theta-1}$; cf. \eqref{eq_Q_matrix} and \eqref{eq_Q_const}. Moreover, for every $x_0' \in \R^{\theta-1}$ the set $\Sigma_{\zeta_{x_0'},\kappa}$ is an affine hyperplane in $\R^{\theta-1}$ and therefore  there exists a $y_0(x_0') \in \R$ and a $\widetilde{\kappa}(x_0')\in \textup{SO}(\theta)$ such that
		\begin{equation*}
			\begin{aligned}
				\Sigma_{\zeta_{x_0'},\kappa} &= \{\kappa(x',\zeta_{x_0'}(x')): x' \in \R^{\theta-1}\}  \\
				&= \widetilde{\kappa}(x_0') \big(\R^{\theta -1} \times \{y_0(x_0')\}\big) = \Sigma_{y_0(x_0'),\widetilde{\kappa}(x_0')}.
			\end{aligned}
		\end{equation*}
		Now, let $\overline{B}_{\varepsilon}^{\Sigma_{\zeta_{x_0'},\kappa}}(z) = \overline{B}_{\varepsilon}^{\Sigma_{y_0(x_0'),\widetilde{\kappa}(x_0')}}(z)$ be defined as $\overline{B}_\varepsilon(z)$ in \eqref{def_B_eps_bar} with $\Sigma$ substituted by $\Sigma_{\zeta_{x_0'},\kappa} = \Sigma_{y_0(x_0'),\widetilde{\kappa}(x_0')}$. According to Proposition~\ref{prop_uniform_bdd_D} the operator $I+D_{\varepsilon}^{y_0(x_0'), \widetilde{\kappa}(x_0')}(z)Q_{\eta(x_{\zeta,\kappa}(x_0')),\tau(x_{\zeta,\kappa}(x_0'))}q$ is continuously invertible in $\mathcal{B}^0(\R^{\theta-1})$. Moreover, we get from \eqref{def_D_operators}
		\begin{equation*}
			\begin{aligned}
			I&+D_{\varepsilon}^{ \zeta_{x_0'}, \kappa}(z)Q_{\eta,\tau}^{\zeta,\kappa} (x_0')q = \iota_{ \zeta_{x_0'}, \kappa} \bigl(I+\overline{B}_{\varepsilon}^{\Sigma_{\zeta_{x_0'},\kappa}}(z)Q_{\eta(x_{\zeta,\kappa}(x_0')),\tau(x_{\zeta,\kappa}(x_0'))}q\bigr)\iota_{ \zeta_{x_0'}, \kappa }^{-1}\\
			&= \iota_{ \zeta_{x_0'}, \kappa} \bigl(I+\overline{B}_{\varepsilon}^{\Sigma_{y_0(x_0'),\widetilde{\kappa}(x_0')}}(z)Q_{\eta(x_{\zeta,\kappa}(x_0')),\tau(x_{\zeta,\kappa}(x_0'))}q\bigr)\iota_{ \zeta_{x_0'}, \kappa }^{-1}\\
			&=\iota_{ \zeta_{x_0'}, \kappa} \iota_{ y_0(x_0'),\widetilde{\kappa}(x_0')}^{-1} \bigl(I+D_{\varepsilon}^{y_0(x_0'), \widetilde{\kappa}(x_0')}(z)Q_{\eta(x_{\zeta,\kappa}(x_0')),\tau(x_{\zeta,\kappa}(x_0'))}q\bigr)\iota_{y_0(x_0'),\widetilde{\kappa}(x_0') } \iota_{ \zeta_{x_0'}, \kappa}^{-1}
			\end{aligned}
		\end{equation*}
		and from \eqref{eq_norm_iota} 
		\begin{equation*}
			\begin{aligned}
			&\norm{\iota_{\zeta_{x_0'},\kappa}}_{L^2(\Sigma_{\zeta_{x_0'},\kappa};\C^N) \to L^2(\R^{\theta-1};\C^N)} =  \norm{\iota_{\zeta_{x_0'},\kappa}^{-1}}_{L^2(\R^{\theta-1};\C^N) \to L^2(\Sigma_{\zeta_{x_0'},\kappa};\C^N)}^{-1}\\
			 &\hspace{143 pt}= (1+\abs{\nabla\zeta(x_0')}^2)^{-1/4},\\
		&\norm{\iota_{y_0(x_0'),\widetilde{\kappa}(x_0') }}_{L^2( \Sigma_{y_0(x_0'),\widetilde{\kappa}(x_0')};\C^N) \to L^2(\R^{\theta-1};\C^N)}	 	\\
		&\hspace{100 pt}=\norm{\iota_{y_0(x_0'),\widetilde{\kappa}(x_0') }^{-1}}_{L^2( \Sigma_{y_0(x_0'),\widetilde{\kappa}(x_0')};\C^N) \to L^2(\R^{\theta-1};\C^N)}^{-1} \\
		&\hspace{100 pt}= 1.
			\end{aligned}
		\end{equation*}
		These considerations show that $I+D_{\varepsilon}^{ \zeta_{x_0'}, \kappa}(z)Q_{\eta,\tau}^{\zeta,\kappa} (x_0')q$ is also continuously invertible in $\mathcal{B}^0(\R^{\theta-1})$ and  
		\begin{equation*} 
			\begin{aligned}
				&\big\|(I+D_{\varepsilon}^{ \zeta_{x_0'}, \kappa}(z)Q_{\eta,\tau}^{\zeta,\kappa} (x_0')q)^{-1}\big\|_{0 \to 0} \\
				&\hspace{70 pt}\leq\big\|(I+D_{\varepsilon}^{y_0(x_0'), \widetilde{\kappa}(x_0')}(z)Q_{\eta(x_{\zeta,\kappa}(x_0')),\tau(x_{\zeta,\kappa}(x_0'))}q)^{-1}\big\|_{0 \to 0}.
			\end{aligned}
		\end{equation*}
		Now, the result follows from applying Corollary~\ref{cor_uniform_bdd_D} (for $S = \overline{\ran(\eta,\tau)}$) if one chooses $\delta_4 = \delta_2(\overline{\ran(\eta,\tau)}) >0$, where $\delta_2$ was introduced in Corollary~\ref{cor_uniform_bdd_D}.
	\end{proof}

	Inspired by the local principle in \cite[Proposition~5]{R22a}, see also \cite{R20,R21}, we are going to construct partitions of unity which allow us to globalize the established local  results.  We start by choosing a partition of unity $(\phi_{n'})_{n' \in \Z^{\theta-1}}$ for $\R^{\theta-1}$ with uniformly bounded derivatives which satisfies  $\supp \phi_{n'} \subset B(n',1)$ for $n'\in\Z^{\theta-1}$.  Moreover, let $(\vartheta_{n'})_{n' \in \Z^{\theta-1}}$ be a sequence of functions with uniformly bounded derivatives which fulfils  $0\leq\vartheta_{n'} \leq 1$, $\vartheta_{n'} =1$ on $B(n',2)$ and $\supp \vartheta_{n'} \subset B(n',3)$ for $n' \in \Z^{\theta-1}$. According to Proposition~\ref{prop_part_unity} such sequences exist.
	By defining for $a \in (0,\varepsilon_2^{1/6})$ and $n' \in \Z^{\theta-1}$ the functions $\phi_{n'}^a(\cdot) = \phi_{n'}(\cdot/a)$ and $\vartheta_{n'}^a(\cdot) = \vartheta_{n'}(\cdot/a)$ we obtain similar sequences with scaled supports; in particular $(\phi_{n'}^a)_{n'\in \Z^{\theta-1}}$ is a partition of unity for $\R^{\theta-1}$. Furthermore, there exists a $C >0$ which does not depend on $a$ such that
	\begin{equation}\label{eq_part_of_unity_uniform_est}
		\sup_{n' \in \Z^{\theta-1}} \max\{\|\phi_{n'}^a\|_{W^1_\infty(\R^{\theta-1})},\|\vartheta_{n'}^a\|_{W^1_\infty(\R^{\theta-1})}\} < \frac{C}{a}.
	\end{equation}
	
	Before we can construct the inverse of $I + D_\varepsilon(z)Q_{\eta,\tau}^{\zeta,\kappa} q$ we have to deal with series of operators. Let $(A_{n'})_{n' \in \Z^{\theta-1}}$ be a family of bounded operators mapping from a Hilbert space $\mathcal{H}$ to a Hilbert space $\mathcal{G}$. 
	If the sequence of partial sums $S_n = \sum_{n ' \in \Z^{\theta-1}, \abs{n'} \leq n} A_{n'}$ converges in the strong sense to an operator, then we define
	\begin{equation*}
		\sum_{n' \in \Z^{\theta-1}}^{\textup{st.}} A_{n'} := \underset{n \to \infty}{\textup{s-lim}} \, S_n.
	\end{equation*}
	According to the Banach-Steinhaus theorem, $\sum_{n' \in \Z^{\theta-1}}^{\textup{st.}} A_{n'} $ is again a bounded operator mapping from $\mathcal{H}$ to $\mathcal{G}$.
	Moreover, the definition of the series implies that if $\mathcal{H}'$ and $\mathcal{G}'$ are Hilbert spaces, and $U:\mathcal{H}' \to \mathcal{H}$ and $V: \mathcal{G} \to \mathcal{G}'$ are bounded operators, then 
	\begin{equation*}
		\sum_{n' \in \Z^{\theta-1}}^{\textup{st.}} VA_{n'}U : \mathcal{H}' \to \mathcal{G}'
	\end{equation*}
	is a well-defined bounded operator and 
	\begin{equation*}
		V \Big(\sum_{n' \in \Z^{\theta-1}}^{\textup{st.}} A_{n'}\Big)U  =	\sum_{n' \in \Z^{\theta-1}}^{\textup{st.}} V A_{n'} U
	\end{equation*} 
	holds.
	If $(\mathcal{A}_{n'})_{n' \in \Z^{\theta-1}}$ is a uniformly bounded sequence of operators in $\mathcal{B}^0(\R^{\theta-1})$, then according to Proposition~\ref{prop_part_series}  the series $\sum_{n' \in \Z^{\theta-1}}^\textup{st.} \vartheta^a_{n'} \mathcal{A}_{n'}\vartheta^a_{n'}$
	converges in the strong sense in $\mathcal{B}^0(\R^{\theta-1})$ and we have
	\begin{equation}\label{eq_series_est_020}
		\Big\| \sum_{n' \in \Z^{\theta-1}}^\textup{st.} \vartheta^a_{n'} \mathcal{A}_{n'}\vartheta^a_{n'} \Big\|_{0\to0} \leq 11^{\theta-1}\sup_{n' \in \Z^{\theta-1}} \| \mathcal{A}_{n'}\|_{0 \to 0}.
	\end{equation}  
	Moreover, if $(\mathcal{A}_{n'})_{n' \in \mathbb{Z}^{\theta-1}}$ is also uniformly bounded as a sequence of operators from $\mathcal{B}^0(\R^{\theta-1})$ to $\mathcal{B}^1(\R^{\theta-1})$, then $\sum_{n' \in \Z^{\theta-1}}^\textup{st.} \vartheta^a_{n'} \mathcal{A}_{n'}\vartheta^a_{n'}$ acts also as  a bounded operator from $\mathcal{B}^0(\R^{\theta-1})$ to $\mathcal{B}^1(\R^{\theta-1})$ and
	\begin{equation}\label{eq_series_est_021}
		\Big\| \sum_{n' \in \Z^{\theta-1}}^\textup{st.} \vartheta^a_{n'} \mathcal{A}_{n'}\vartheta^a_{n'} \Big\|_{0\to1} \leq \frac{C}{a} \sup_{n' \in \Z^{\theta-1}}\norm{\mathcal{A}_{n'}}_{0 \to 1}
	\end{equation}
	holds, with $C>0$ independent of $a$, see also Proposition~\ref{prop_part_series}. Before we use these essential observations in the proof of Proposition~\ref{prop_right_inverse}, we state a helpful preliminary lemma.
	
	\begin{lemma}\label{lem_right_inverse}
		Let $z \in \rho(H)$, $\varepsilon \in (0,\varepsilon_2)$, $a_\varepsilon = \varepsilon^{1/6} $, $D_{\varepsilon}^{\zeta,\kappa}(z)$ be as in \eqref{def_D_operators}, and $E_{\varepsilon}(z)$ be as in \eqref{def_E_eps}. Then, for any fixed $n' \in \mathbb{Z}^{\theta-1}$ 
		\begin{equation*}
			(1-\vartheta_{n'}^{a_\varepsilon})E_\varepsilon(z) \phi_{n'}^{a_\varepsilon} = (1-\vartheta_{n'}^{a_\varepsilon}) D_{\varepsilon}^{ \zeta, \kappa}(z) \phi_{n'}^{a_\varepsilon}.
		\end{equation*}
	\end{lemma}
	\begin{proof}
		We prove this by showing that the difference of the integral kernels of the operators $(1-\vartheta_{n'}^{a_\varepsilon})E_\varepsilon(z) \phi_{n'}^{a_\varepsilon}$ and $ (1-\vartheta_{n'}^{a_\varepsilon}) D_{\varepsilon}^{ \zeta, \kappa}(z) \phi_{n'}^{a_\varepsilon}$ is zero. By \eqref{eq_D_eps_int_rep_0}, \eqref{def_e_eps}, and \eqref{def_E_eps}  this  difference is given by
		\begin{equation}\label{eq_lem_right_inv_kernel}
			\begin{aligned}
				&(1-\vartheta_{n'}^{a_\varepsilon}(x'))\bigl(\omega\bigl(\tfrac{x'-y'}{a_\varepsilon}\bigr) -1\bigr)\phi_{n'}^{a_\varepsilon}(y') \\
				&\hspace{50 pt}\cdot G_z(x_{\zeta,\kappa}(x') -x_{\zeta,\kappa}(y') +\varepsilon(t-s)\nu_{\zeta,\kappa}(x')) 
				\sqrt{1+ \abs{\nabla \zeta(y')}^2}
			\end{aligned}
		\end{equation}
		for all $x',y' \in \R^{\theta-1}$ and $t,s \in (-1,1)$. If $y' \notin B(a_\varepsilon n',a_\varepsilon)$, then $\tfrac{y'}{a_\varepsilon} \notin B(n',1) \supset \supp \phi_{n'}$ and therefore 
		\begin{equation*}
			\phi_{n'}^{a_\varepsilon}(y') = \phi_{n'}\bigl(\tfrac{y'}{a_\varepsilon}\bigr)=0.
		\end{equation*}
		Furthermore, if $x' \in B(a_\varepsilon n', 2a_\varepsilon)$, then $\tfrac{x'}{a_\varepsilon} \in B(n',2)$ and  hence as $\vartheta_{n'} =1 $ on $B(n',2)$ we have 
		\begin{equation*}
			1-\vartheta_{n'}^{a_\varepsilon}(x') = 1-\vartheta_{n'}\bigl(\tfrac{x'}{a_\varepsilon}\bigr)=0.
		\end{equation*}
		These two observations show that if $x' \in  B(a_\varepsilon n', 2a_\varepsilon)$ or $y' \notin B(a_\varepsilon n',a_\varepsilon)$, then \eqref{eq_lem_right_inv_kernel} vanishes. Thus, it remains to consider the case $x' \notin B(a_\varepsilon n', 2 a_\varepsilon)$ and $y' \in B(a_\varepsilon n', a_\varepsilon)$. However, this implies $|x'-y'| > a_\varepsilon$. In this case we use $\omega = 1$ on $\R^{\theta-1} \setminus B(0,1)$, see the text above \eqref{def_E_eps}, to obtain 
		\begin{equation*}
			\omega\bigl(\tfrac{x'-y'}{a_\varepsilon}\bigr) -1 =0. 
		\end{equation*}
		This shows that \eqref{eq_lem_right_inv_kernel} vanishes for all $x',y'\in \R^{\theta-1}$ and $t,s \in(-1,1)$.
	\end{proof}
	
	\begin{proposition}\label{prop_right_inverse}
		Let $z \in \C \setminus \R$, $\zeta \in C^2_b(\R^{\theta-1};\R)$ and $\Sigma$ be as described in the beginning of Section~\ref{sec_loc_part}, $\eta,\tau \in C^1_b(\Sigma;\R)$,  $d = \eta^2- \tau^2$ satisfy
		\begin{equation*}
			\sup_{x_\Sigma \in \Sigma} d(x_\Sigma) < \frac{\pi^2}{4},
		\end{equation*}
		$Q_{\eta,\tau}^{\zeta,\kappa}$ be as in \eqref{eq_Q_matrix} and $q$ be as in \eqref{eq_q2}. Then, there exists a $\delta_5 \in (0,\varepsilon_2)$, with $\varepsilon_2>0$ from Proposition~\ref{prop_conv_res},  such   that  $I+D^{\zeta,\kappa}_{{\varepsilon}}(z)Q_{\eta,\tau}^{\zeta,\kappa}q$ has a bounded right inverse which is uniformly bounded in $\mathcal{B}^0(\R^{\theta-1})$ with respect to $\varepsilon \in (0,\delta_5)$. 
	\end{proposition}
	\begin{proof}
		The proof is split into four steps. In \textit{Step~1} we define $R_\varepsilon$ which will turn out to be a first approximation for the right inverse of $I+D^{\zeta,\kappa}_{{\varepsilon}}(z)Q_{\eta,\tau}^{\zeta,\kappa}q$. Moreover, in this step we also show that $R_\varepsilon$ is uniformly bounded in $\mathcal{B}^0(\R^{\theta-1})$  with respect to $\varepsilon$. Then, in \textit{Step~2} we calculate  $(I+D^{\zeta,\kappa}_{{\varepsilon}}(z)Q_{\eta,\tau}^{\zeta,\kappa}q) R_\varepsilon$. Afterwards, we find in \textit{Step~3} that  $(I+D^{\zeta,\kappa}_{{\varepsilon}}(z)Q_{\eta,\tau}^{\zeta,\kappa}q) R_\varepsilon $ equals $I + K_\varepsilon + L_\varepsilon$, where $K_\varepsilon$ and $L_\varepsilon$ fulfil the inequalities
		\begin{equation}\label{eq_est_K_eps_L_eps}
			\|K_{\varepsilon}\|_{0 \to 1}\leq  C \frac{  1+\abs{\log(\varepsilon)}}{a_\varepsilon^2} \\
			\quad \textup{and} \quad  \|L_{\varepsilon}\|_{0\to0} \leq C a_\varepsilon (1+\abs{\log (\varepsilon)}).
		\end{equation}
		Based on these observations we define in \textit{Step~4} an operator $\widetilde{R}_\varepsilon$  which is uniformly bounded in $\mathcal{B}^0(\R^{\theta-1})$ with respect to $\varepsilon$ and fulfils $(I+D^{\zeta,\kappa}_{{\varepsilon}}(z)Q_{\eta,\tau}^{\zeta,\kappa}q) \widetilde{R}_\varepsilon = I+ \widetilde{L}_\varepsilon$, where $\|\widetilde{L}_\varepsilon\|_{0\to0}$ can be estimated by $C(1+|\log(\varepsilon)|) \varepsilon^{1/6-r}$ for an $r \in (0, \tfrac{1}{6})$. In particular, this shows that for sufficiently small $\varepsilon>0$  the right inverse of $I+D^{\zeta,\kappa}_{{\varepsilon}}(z)Q_{\eta,\tau}^{\zeta,\kappa}q$ is given by the operator $\widetilde{R}_\varepsilon(I+\widetilde{L}_\varepsilon)^{-1}$ and  uniformly bounded in $\mathcal{B}^0(\R^{\theta-1})$ with respect to $\varepsilon$.
		
		\textit{Step~1.}
		We define for $\varepsilon \in (0,\min\{\varepsilon_2,\delta_4\})$, where  $\varepsilon_2$ and $\delta_4$  are chosen as in Proposition~\ref{prop_conv_res} and  Proposition~\ref{prop_uniform_bdd_D_inv}, respectively,
		\begin{equation*}
			\begin{aligned}
				R_\varepsilon &: \mathcal{B}^0(\R^{\theta-1}) \to \mathcal{B}^0(\R^{\theta-1}), \\
				R_\varepsilon &:= \sum^{\textup{st.}}_{n' \in \Z^{\theta-1}} \phi_{n'}^{a_\varepsilon} R_{n',\varepsilon}\vartheta_{n'}^{a_\varepsilon}. 
			\end{aligned}
		\end{equation*}
		Here, $a_\varepsilon = \varepsilon^{1/6}$ and $R_{n',\varepsilon}:=(I+D^{\zeta_{a_\varepsilon n'},\kappa}_\varepsilon(z)  Q_{\eta,\tau}^{\zeta,\kappa}(a_\varepsilon n') q)^{-1}$ with $\zeta_{a_\varepsilon n'} = \zeta_{x_0'}$ as in \eqref{eq_zeta_x_0'} for $x_0' = a_\varepsilon n'$. The equality $\vartheta_{n'}^{a_\varepsilon} \phi_{n'}^{a_\varepsilon}=\phi_{n'}^{a_\varepsilon}$ implies
		\begin{equation*}
			R_\varepsilon = \sum^{\textup{st.}}_{n' \in \Z^{\theta-1}} \vartheta_{n'}^{a_\varepsilon} \phi_n^{a_\varepsilon} R_{n',\varepsilon}\vartheta_{n'}^{a_\varepsilon} 
		\end{equation*} 
		and therefore Proposition~\ref{prop_uniform_bdd_D_inv} and \eqref{eq_series_est_020} show that $R_\varepsilon$ is well-defined and uniformly bounded by
		\begin{equation}\label{eq_R_eps_uniform_bdd}
			\norm{R_\varepsilon}_{0 \to 0} \leq 11^{\theta-1}\sup_{n' \in \Z^{\theta-1}} \norm{ (I+D^{\zeta_{a_\varepsilon n'},\kappa}_\varepsilon(z) Q_{\eta,\tau}^{\zeta,\kappa}(a_\varepsilon n')  q)^{-1}}_{0\to0}  \leq  C ,
		\end{equation}
		where $C>0$ does not depend on  $\varepsilon $.
		
		\textit{Step~2.}
		Applying $I+D_{{\varepsilon}}^{\zeta,\kappa}(z)Q^{\zeta,\kappa}_{\eta,\tau}q$ to $R_\varepsilon$ yields
		\begin{equation*}
			\begin{aligned}
				(I+D_{{\varepsilon}}^{\zeta,\kappa}(z)Q^{\zeta,\kappa}_{\eta,\tau}q)R_\varepsilon 
				&= \sum^{\textup{st.}}_{n' \in \Z^{\theta-1}}(I+D_{{\varepsilon}}^{\zeta,\kappa}(z)Q^{\zeta,\kappa}_{\eta,\tau}q)\phi_{n'}^{a_\varepsilon} R_{n',\varepsilon}  \vartheta_{n'}^{a_\varepsilon}  \\
				&=\sum^{\textup{st.}}_{n' \in \Z^{\theta-1}}  \vartheta_{n'}^{a_\varepsilon}  (I+D_{{\varepsilon}}^{\zeta,\kappa}(z)Q^{\zeta,\kappa}_{\eta,\tau}q)\phi_{n'}^{a_\varepsilon} R_{n',\varepsilon}  \vartheta_{n'}^{a_\varepsilon}\\
				&\hspace{10 pt}+ \sum^{\textup{st.}}_{n' \in \Z^{\theta-1}}(1-\vartheta_{n'}^{a_\varepsilon})D_{{\varepsilon}}^{\zeta,\kappa}(z)Q^{\zeta,\kappa}_{\eta,\tau}q\phi_{n'}^{a_\varepsilon} R_{n',\varepsilon}  \vartheta_{n'}^{a_\varepsilon}. 
			\end{aligned}
		\end{equation*}
		Moreover, using Lemma~\ref{lem_right_inverse} gives us
		\begin{equation*}
			\begin{aligned}
				(I+&D_{{\varepsilon}}^{\zeta,\kappa}(z)Q^{\zeta,\kappa}_{\eta,\tau}q)	R_\varepsilon 
				=\sum^{\textup{st.}}_{n' \in \Z^{\theta-1}}  \vartheta_{n'}^{a_\varepsilon}  (I+D_{{\varepsilon}}^{\zeta,\kappa}(z)Q^{\zeta,\kappa}_{\eta,\tau}q)\phi_{n'}^{a_\varepsilon} R_{n',\varepsilon}  \vartheta_{n'}^{a_\varepsilon}\\
				&\hspace{83 pt} + \sum^{\textup{st.}}_{n' \in \Z^{\theta-1}}(1-\vartheta_{n'}^{a_\varepsilon})E_\varepsilon(z)Q^{\zeta,\kappa}_{\eta,\tau}q\phi_{n'}^{a_\varepsilon} R_{n',\varepsilon}  \vartheta_{n'}^{a_\varepsilon}  \\
				&=\sum^{\textup{st.}}_{n' \in \Z^{\theta-1}}  \vartheta_{n'}^{a_\varepsilon}  (I+D_{{\varepsilon}}^{\zeta,\kappa}(z)Q^{\zeta,\kappa}_{\eta,\tau}q - E_{{\varepsilon}}(z)Q^{\zeta,\kappa}_{\eta,\tau}q)\phi_{n'}^{a_\varepsilon} R_{n',\varepsilon}  \vartheta_{n'}^{a_\varepsilon}  \\
				&\hspace{10 pt}+  E_{{\varepsilon}}(z)Q^{\zeta,\kappa}_{\eta,\tau}qR_\varepsilon.
			\end{aligned}
		\end{equation*}
		Writing $D_{{\varepsilon}}^{\zeta,\kappa}(z)Q^{\zeta,\kappa}_{\eta,\tau}q \phi_{n'}^{a_\varepsilon}$ as 
		\begin{equation*}
			\begin{aligned}
				&	 D_{{\varepsilon}}^{\zeta_{a_\varepsilon n'},\kappa}(z)Q^{\zeta,\kappa}_{\eta,\tau}(a_\varepsilon n')q \phi_{n'}^{a_\varepsilon} +  \big(D_{{\varepsilon}}^{\zeta,\kappa}(z)Q^{\zeta,\kappa}_{\eta,\tau} -  D_{{\varepsilon}}^{\zeta_{a_\varepsilon n'},\kappa}(z)Q^{\zeta,\kappa}_{\eta,\tau}(a_\varepsilon n')\big)q\phi_{n'}^{a_\varepsilon} \\	
				& \qquad = 	 \phi_{n'}^{a_\varepsilon}D_{{\varepsilon}}^{\zeta_{a_\varepsilon n'},\kappa}(z)Q^{\zeta,\kappa}_{\eta,\tau}(a_\varepsilon n')q +  [D_{{\varepsilon}}^{\zeta_{a_\varepsilon n'},\kappa}(z), \phi_{n'}^{a_\varepsilon}]Q^{\zeta,\kappa}_{\eta,\tau}(a_\varepsilon n')q\\
				&\qquad \hphantom{=}+ \big(D_{{\varepsilon}}^{\zeta,\kappa}(z)Q^{\zeta,\kappa}_{\eta,\tau} -  D_{{\varepsilon}}^{\zeta_{a_\varepsilon n'},\kappa}(z)Q^{\zeta,\kappa}_{\eta,\tau}(a_\varepsilon n')\big)q\phi_{n'}^{a_\varepsilon} \
			\end{aligned}
		\end{equation*}
		and introducing the operators $L_{n',\varepsilon}:= \big(D_{{\varepsilon}}^{\zeta,\kappa}(z)Q^{\zeta,\kappa}_{\eta,\tau} -  D_{{\varepsilon}}^{\zeta_{a_\varepsilon n'},\kappa}(z)Q^{\zeta,\kappa}_{\eta,\tau}(a_\varepsilon n')\big)q\phi_{n'}^{a_\varepsilon}$ and  $K_{n',\varepsilon} := [D_{{\varepsilon}}^{\zeta_{a_\varepsilon n'},\kappa}(z), \phi_{n'}^{a_\varepsilon}]Q^{\zeta,\kappa}_{\eta,\tau}(a_\varepsilon n')q - E_{\varepsilon}(z)Q^{\zeta,\kappa}_{\eta,\tau}q\phi_{n'}^{a_\varepsilon}$  yields
		\begin{equation}\label{eq_prop_right_inv_step1}
			\begin{aligned}
				(I+D_{{\varepsilon}}^{\zeta,\kappa}(z)&Q^{\zeta,\kappa}_{\eta,\tau}q)R_\varepsilon \\
				&=\sum^{\textup{st.}}_{n' \in \Z^{\theta-1}}  \vartheta_{n'}^{a_\varepsilon} \phi_{n'}^{a_\varepsilon} (I+D_{{\varepsilon}}^{\zeta_{a_\varepsilon n'},\kappa}(z)Q_{\eta,\tau}^{\zeta,\kappa}(a_\varepsilon n')q) R_{n',\varepsilon}  \vartheta_{n'}^{a_\varepsilon} \\
				&\hphantom{=}+ \sum^{\textup{st.}}_{n' \in \Z^{\theta-1}}  \vartheta_{n'}^{a_\varepsilon}(K_{n',\varepsilon}  +L_{n',\varepsilon} ) R_{n',\varepsilon}  \vartheta_{n'}^{a_\varepsilon} 
				+E_{{\varepsilon}}(z)Q^{\zeta,\kappa}_{\eta,\tau}qR_\varepsilon \\
				&=I +  \sum^{\textup{st.}}_{n' \in \Z^{\theta-1}}  \vartheta_{n'}^{a_\varepsilon}(K_{n',\varepsilon}  +L_{n',\varepsilon} ) R_{n',\varepsilon}  \vartheta_{n'}^{a_\varepsilon} 
				+E_{{\varepsilon}}(z)Q^{\zeta,\kappa}_{\eta,\tau}qR_\varepsilon,
			\end{aligned}
		\end{equation}
		where 
		\begin{equation*}
			(I+D_{{\varepsilon}}^{\zeta_{a_\varepsilon n'},\kappa}(z)Q_{\eta,\tau}^{\zeta,\kappa}(a_\varepsilon n')q) R_{n',\varepsilon} = I  \quad \textup{and} \quad \sum_{n' \in \Z^{\theta-1}} \vartheta_{n'}^{a_\varepsilon} \phi_{n'}^{a_\varepsilon} \vartheta_{n'}^{a_\varepsilon} = \sum_{n' \in \Z^{\theta-1}} \phi_{n'}^{a_\varepsilon} =1
		\end{equation*}
		were used. 
		
		\textit{Step~3.}
		We start this step by setting
		\begin{equation*}
				K_{\varepsilon} := \sum^{\textup{st.}}_{n' \in \Z^{\theta-1}}\vartheta_{n'}^{a_\varepsilon} K_{n',\varepsilon} R_{n',\varepsilon}  \vartheta_{n'}^{a_\varepsilon} +E_{{\varepsilon}}(z)Q^{\zeta,\kappa}_{\eta,\tau}qR_\varepsilon
        \end{equation*}
        and 
        \begin{equation*}
            L_{\varepsilon} 	:=  \sum^{\textup{st.}}_{n' \in \Z^{\theta-1}} \vartheta_{n'}^{a_\varepsilon} L_{n',\varepsilon}  R_{n',\varepsilon} \vartheta_{n'}^{a_\varepsilon}.
		\end{equation*}
		Then, \eqref{eq_prop_right_inv_step1} shows 
		\begin{equation}\label{eq_I+D_approx_inv}
			(I+D_{{\varepsilon}}^{\zeta,\kappa}(z)Q_{\eta,\tau}^{\zeta,\kappa}q)R_\varepsilon  = I + K_{\varepsilon} +L_{\varepsilon}.
		\end{equation} 
	 Since $R_\varepsilon$ and $D_{{\varepsilon}}^{\zeta,\kappa}(z)$ are uniformly bounded in $\mathcal{B}^0(\R^{\theta-1})$, see \textit{Step~1} and the text above \eqref{eq_D_eps_conv}, respectively, this implies that also $K_{\varepsilon} +L_{\varepsilon}$ is uniformly bounded in $\mathcal{B}^0(\R^{\theta-1})$. Moreover,  $Q_{\eta,\tau}^{\zeta,\kappa} \in C^1_b(\R^{\theta-1};\C^{N \times N})$, Proposition~\ref{prop_local_prop}, and~\eqref{eq_part_of_unity_uniform_est} imply
	 	\begin{equation*}
	 	\begin{aligned}
	 		\norm{K_{n',\varepsilon}}_{0 \to 1} &\leq  \norm{[D_{{\varepsilon}}^{\zeta_{a_\varepsilon n'},\kappa}(z), \phi_{n'}^{a_\varepsilon}]Q^{\zeta,\kappa}_{\eta,\tau}(a_\varepsilon n')q}_{0 \to 1} + \norm{E_{\varepsilon}(z)Q^{\zeta,\kappa}_{\eta,\tau}q\phi_{n'}^{a_\varepsilon}}_{0 \to 1}\\
	 		& \leq C  \bigl(\norm{[D_{{\varepsilon}}^{\zeta_{a_\varepsilon n'},\kappa}(z), \phi_{n'}^{a_\varepsilon}]}_{0 \to 1} + \norm{E_{\varepsilon}(z)}_{0 \to 1} \bigr)\\
	 		&\leq C \Bigl( \norm{\phi_{n'}^{a_\varepsilon}}_{W^1_\infty(\R^{\theta-1})}(1+ \abs{\log (\varepsilon)}) + \frac{1+ \abs{\log (\varepsilon)}}{a_\varepsilon}\Bigr)\\
	 		&\leq C \frac{1+ \abs{\log (\varepsilon)}}{a_\varepsilon}, \qquad  n' \in \Z^{\theta-1}.
	 	\end{aligned}
	 \end{equation*}
  In turn, with \eqref{eq_series_est_021}, \eqref{eq_R_eps_uniform_bdd}, $Q_{\eta,\tau}^{\zeta,\kappa} \in C^1_b(\R^{\theta-1};\C^{N \times N})$,  Proposition~\ref{prop_local_prop}, and Proposition \ref{prop_uniform_bdd_D_inv} we get 
		\begin{equation*}
			\begin{aligned}
				\norm{K_{\varepsilon}}_{0 \to 1} &\leq \Bigl\| \sum^{\textup{st.}}_{n' \in \Z^{\theta-1}}\vartheta_{n'}^{a_\varepsilon} K_{n',\varepsilon} R_{n',\varepsilon}  \vartheta_{n'}^{a_\varepsilon}\Bigr\|_{0 \to 1} + \|E_{{\varepsilon}}(z)Q^{\zeta,\kappa}_{\eta,\tau}qR_\varepsilon\|_{0 \to 1}\\
				& \leq \Bigl(\frac{C}{a_\varepsilon}\sup_{n' \in \Z^{\theta-1}} \|K_{n',\varepsilon}R_{n',\varepsilon}\|_{0\to1} +  \|E_{{\varepsilon}}(z)Q^{\zeta,\kappa}_{\eta,\tau}qR_\varepsilon\|_{0 \to 1}\Bigr)\\
				& \leq \Bigl(\frac{C}{a_\varepsilon}\sup_{n' \in \Z^{\theta-1}} \|K_{n',\varepsilon}\|_{0\to1} +  \|E_{{\varepsilon}}(z)\|_{0 \to 1}\Big)\\
				& \leq C \Bigl(\frac{1}{a_\varepsilon}\cdot\frac{1+ \abs{\log (\varepsilon)}}{a_\varepsilon} + \frac{1+ \abs{\log (\varepsilon)}}{a_\varepsilon}\Bigr)\\
				&\leq C \frac{1+\abs{\log(\varepsilon)}}{a_\varepsilon^2}.
			\end{aligned}
		\end{equation*}
		Similar we estimate $L_\varepsilon$ with \eqref{eq_series_est_020}, Proposition~\ref{prop_uniform_bdd_D_inv}, and Proposition~\ref{prop_local_prop}  by
		\begin{equation*}
			\begin{aligned}
				\|L_{\varepsilon}\|_{0 \to 0} &= \Big\| \sum^{\textup{st.}}_{n' \in \Z^{\theta-1}}\vartheta_{n'}^{a_\varepsilon} L_{n',\varepsilon} R_{n',\varepsilon}  \vartheta_{n'}^{a_\varepsilon}\Big\|_{0 \to 0} \\
				&= \Big\| \sum^{\textup{st.}}_{n' \in \Z^{\theta-1}}\vartheta_{n'}^{a_\varepsilon} \chi_{B(a_\varepsilon n',3 a_\varepsilon)} L_{n',\varepsilon} R_{n',\varepsilon}  \vartheta_{n'}^{a_\varepsilon}\Big\|_{0 \to 0}\\ 
				& \leq C \sup_{n' \in \Z^{\theta-1}} \|\chi_{B(a_\varepsilon n',3 a_\varepsilon)}L_{n',\varepsilon}\|_{0\to 0}\\
				&= C \sup_{n' \in \Z^{\theta-1}} \|\chi_{B(a_\varepsilon n',3 a_\varepsilon)}\big(D_{{\varepsilon}}^{\zeta,\kappa}(z)Q^{\zeta,\kappa}_{\eta,\tau} -  D_{{\varepsilon}}^{\zeta_{a_\varepsilon n'},\kappa}(z)Q^{\zeta,\kappa}_{\eta,\tau}(a_\varepsilon n')\big)q\phi_{n'}^{a_\varepsilon}\|_{0\to 0} \\
				&= C \sup_{n' \in \Z^{\theta-1}} \|\chi_{B(a_\varepsilon n',3 a_\varepsilon)} \\
				&\hspace{30 pt}\cdot\big(D_{{\varepsilon}}^{\zeta,\kappa}(z)Q^{\zeta,\kappa}_{\eta,\tau} -  D_{{\varepsilon}}^{\zeta_{a_\varepsilon n'},\kappa}(z)Q^{\zeta,\kappa}_{\eta,\tau}(a_\varepsilon n')\big)q\chi_{B(a_\varepsilon n',3 a_\varepsilon)}\phi_{n'}^{a_\varepsilon}\|_{0\to 0} \\
				&\leq C a_\varepsilon (1+\abs{\log(\varepsilon)}).
			\end{aligned}
		\end{equation*}
		This shows that \eqref{eq_est_K_eps_L_eps} is valid and hence completes \textit{Step~3}.
		
		\textit{Step~4.}
		We note that Proposition~\ref{prop_inv_I+BVq}, \eqref{def_iota},  \eqref{def_D_operators} and \eqref{eq_Q_matrix} imply that $I+D_0^{\zeta,\kappa}(z)Q^{\zeta,\kappa}_{\eta,\tau}q$ is continuously invertible in $\mathcal{B}^0(\R^{\theta-1})$ and $\mathcal{B}^{1/2}(\R^{\theta-1})$.
		Furthermore, since $R_\varepsilon$ and $K_\varepsilon + L_\varepsilon$ are  uniformly bounded operators in $\mathcal{B}^0(\R^{\theta-1})$, the operator 
		\begin{equation*}
			\widetilde{R}_\varepsilon := R_\varepsilon -(I+D_0^{\zeta,\kappa}(z)Q^{\zeta,\kappa}_{\eta,\tau}q)^{-1}( K_{\varepsilon} +L_{\varepsilon})
		\end{equation*}
		is also uniformly bounded in $\mathcal{B}^0(\R^{\theta-1})$. Moreover,
		applying $I+D_\varepsilon^{\zeta,\kappa}(z)Q^{\zeta,\kappa}_{\eta,\tau}q$ to $\widetilde{R}_\varepsilon$ and \eqref{eq_I+D_approx_inv} give us 
		\begin{equation}\label{eq_right_inverse_R_eps_tilde}
			\begin{aligned}
				(I&+D_{{\varepsilon}}^{\zeta,\kappa}(z)Q^{\zeta,\kappa}_{\eta,\tau}q) \widetilde{R}_\varepsilon  \\
				&= I + K_\varepsilon + L_\varepsilon  - (I+D_\varepsilon^{\zeta,\kappa}(z)Q^{\zeta,\kappa}_{\eta,\tau}q)(I+D_0^{\zeta,\kappa}(z)Q^{\zeta,\kappa}_{\eta,\tau}q)^{-1}( K_{\varepsilon} +L_{\varepsilon})\\
				& = I +  (D_0^{\zeta,\kappa}(z)-D_\varepsilon^{\zeta,\kappa}(z))Q^{\zeta,\kappa}_{\eta,\tau}q(I+D_0^{\zeta,\kappa}(z)Q^{\zeta,\kappa}_{\eta,\tau}q)^{-1}( K_{\varepsilon} +L_{\varepsilon}) \\
				&= I + \widetilde{L}_\varepsilon
			\end{aligned}
		\end{equation}
		with
		\begin{equation*}
			\widetilde{L}_\varepsilon := (D_0^{\zeta,\kappa}(z)-D_\varepsilon^{\zeta,\kappa}(z))Q^{\zeta,\kappa}_{\eta,\tau}q(I+D_0^{\zeta,\kappa}(z)Q^{\zeta,\kappa}_{\eta,\tau}q)^{-1}( K_{\varepsilon} +L_{\varepsilon}).
		\end{equation*}
		Thus, by using the estimates for $L_\varepsilon$ and $K_\varepsilon$ from \textit{Step~3} and \eqref{eq_D_eps_conv} for a fixed $r \in (0,\tfrac{1}{6})$ we obtain
		\begin{equation*}
			\begin{aligned}
				\| \widetilde{L}_\varepsilon \|_{0\to0}  
				&\leq  \|(D_{{0}}^{\zeta,\kappa}(z) -D_{\varepsilon}^{\zeta,\kappa}(z)) Q^{\zeta,\kappa}_{\eta,\tau}q(I+D_{{0}}^{\zeta,\kappa}(z)Q^{\zeta,\kappa}_{\eta,\tau}q)^{-1}K_{\varepsilon}\|_{0\to0} \\
				&\hspace{10 pt}+\|(D_{{0}}^{\zeta,\kappa}(z) -D_{\varepsilon}^{\zeta,\kappa}(z)) Q^{\zeta,\kappa}_{\eta,\tau}q(I+D_{{0}}^{\zeta,\kappa}(z)Q^{\zeta,\kappa}_{\eta,\tau}q)^{-1}L_{\varepsilon}\|_{0\to0} \\
				&\leq   \|(D_{{0}}^{\zeta,\kappa}(z) -D_{\varepsilon}^{\zeta,\kappa}(z)) Q^{\zeta,\kappa}_{\eta,\tau}q \|_{1/2 \to 0} \\
				&\hspace{70 pt} \cdot\|(I+D_{{0}}^{\zeta,\kappa}(z)Q^{\zeta,\kappa}_{\eta,\tau}q)^{-1}\|_{1/2 \to 1/2}\|K_\varepsilon\|_{0\to 1/2} \\
				&\hspace{10 pt}+\|(D_{{0}}^{\zeta,\kappa}(z) -D_{\varepsilon}^{\zeta,\kappa}(z)) Q^{\zeta,\kappa}_{\eta,\tau}q \|_{0\to0}\\
				&\hspace{70 pt} \cdot \|(I+D_{{0}}^{\zeta,\kappa}(z)Q^{\zeta,\kappa}_{\eta,\tau}q)^{-1}\|_{0\to0} \|L_\varepsilon\|_{0\to0}\\
				&\leq   \|(D_{{0}}^{\zeta,\kappa}(z) -D_{\varepsilon}^{\zeta,\kappa}(z)) Q^{\zeta,\kappa}_{\eta,\tau}q \|_{1/2 \to 0} \\
				&\hspace{70 pt} \cdot\|(I+D_{{0}}^{\zeta,\kappa}(z)Q^{\zeta,\kappa}_{\eta,\tau}q)^{-1}\|_{1/2 \to 1/2}\|K_\varepsilon\|_{0\to 1} \\
				&\hspace{10 pt}+\|(D_{{0}}^{\zeta,\kappa}(z) -D_{\varepsilon}^{\zeta,\kappa}(z)) Q^{\zeta,\kappa}_{\eta,\tau}q \|_{0\to0} \\
				&\hspace{70 pt}\cdot \|(I+D_{{0}}^{\zeta,\kappa}(z)Q^{\zeta,\kappa}_{\eta,\tau}q)^{-1}\|_{0\to0} \|L_\varepsilon\|_{0\to0}\\
				& \leq C \bigg( \frac{\varepsilon^{1/2-r} (1+\abs{\log(\varepsilon)})}{a_\varepsilon^{2}}  + a_\varepsilon (1+\abs{\log(\varepsilon)}) \bigg) \\
				&= C (1+\abs{\log(\varepsilon)}) (\varepsilon^{1/6-r} + \varepsilon^{1/6}) \\
				&\leq C (1+\abs{\log(\varepsilon)}) \varepsilon^{1/6-r}.
			\end{aligned}
		\end{equation*}
		This shows that if we choose $\delta_5>0$ sufficiently small, then $\| \widetilde{L}_\varepsilon\|_{0 \to 0} < \tfrac{1}{2}$ for all $\varepsilon \in (0,\delta_5)$ and $(I + \widetilde{L}_\varepsilon)^{-1}$ is uniformly bounded with respect to $\varepsilon \in (0,\delta_5)$. Thus,  as $\widetilde{R}_\varepsilon$ is also uniformly bounded in $\mathcal{B}^0(\R^{\theta-1})$,  $\widetilde{R}_\varepsilon(I+\widetilde{L}_\varepsilon)^{-1}$ is uniformly bounded in $\mathcal{B}^0(\R^{\theta-1})$ with respect to $\varepsilon \in (0,\delta_5)$ and by \eqref{eq_right_inverse_R_eps_tilde} it is also the right inverse of $I+D_{{\varepsilon}}^{\zeta,\kappa}(z)Q^{\zeta,\kappa}_{\eta,\tau}q$.
	\end{proof}
	
	\begin{proposition}\label{prop_uniform_bdd_inv_B_eps}
		Let $\Sigma$ be a rotated graph as described in the beginning of Section~\ref{sec_loc_part}, $z \in \C\setminus\R$, $q$  be as in \eqref{eq_q2}, $V = \eta I_N + \tau \beta$ with $\eta,\tau \in C^1_b(\Sigma;\R)$, and $ d= \eta^2-\tau^2$ such that
		\begin{equation*}
			\sup_{x_\Sigma \in \Sigma} d(x_\Sigma) < \frac{\pi^2}{4}.
		\end{equation*} 
		 Then, there exists $\varepsilon_3 \in (0,\varepsilon_2) $, with $\varepsilon_2>0$ from Proposition~\ref{prop_conv_res}, such that  $I+B_\varepsilon(z)Vq$ has a bounded  inverse which is uniformly bounded in $\mathcal{B}^0(\Sigma)$ with respect to $\varepsilon \in (0,\varepsilon_3)$.
	\end{proposition}
	\begin{proof}
		We  directly get from Proposition~\ref{prop_right_inverse},  \eqref{eq_norm_iota}, \eqref{def_D_operators}, and \eqref{eq_Q_matrix}  that $I+\overline{B}_{{\varepsilon}}(z)Vq$ has a right inverse  which is uniformly bounded with respect to $\varepsilon \in (0,\delta_5)$ with $\delta_5$ from the previous proposition. Using \eqref{eq_B_diff} shows that then $I+B_{{\varepsilon}}(z)Vq$ has a right inverse which is uniformly bounded for $\varepsilon \in (0,\varepsilon_3)$ if $\varepsilon_3>0$ is chosen small enough. We denote this right inverse  $\mathcal{R}_\varepsilon(z)$. Then, the right inverse of $I+VqB_{{\varepsilon}}(\overline{z})$ is given by $I - Vq \mathcal{R}_\varepsilon(\overline{z}) B_\varepsilon(\overline{z})$. Hence, $(I+VqB_{{\varepsilon}}(\overline{z}))^* = I + (B_{{\varepsilon}}(\overline{z}))^*V q$ has the uniformly bounded left-inverse $\mathcal{L}_\varepsilon(z) := I - (Vq \mathcal{R}_\varepsilon(\overline{z}) B_\varepsilon(\overline{z}))^*$. Moreover, by \cite[Remark~3.11]{BHS23} the estimate
		\begin{equation*}
			\big\| B_\varepsilon(z) - (B_\varepsilon(\overline{z}))^*\big\|_{0 \to 0} \leq C \varepsilon, \quad  \varepsilon \in (0,\varepsilon_2)
		\end{equation*}
		holds.
		Thus,
		\begin{equation}\label{eq_prop_uniform_bdd_inv_B_eps}
			\mathcal{L}_\varepsilon(z) (I +B_\varepsilon(z)Vq) = I + \mathcal{L}_\varepsilon(z)(B_\varepsilon(z)- (B_\varepsilon(\overline{z}))^*)Vq
		\end{equation}
		converges in $\mathcal{B}^0(\R^{\theta-1})$ for $\varepsilon \to 0$ to  the identity operator $I$. In particular, if $\varepsilon_3>0$ is chosen small enough, then the right hand side of \eqref{eq_prop_uniform_bdd_inv_B_eps} is  invertible in $\mathcal{B}^0(\R^{\theta-1})$ for all $\varepsilon \in (0,\varepsilon_3)$, showing that $I + B_\varepsilon(z) Vq$ is  invertible in $\mathcal{B}^0(\R^{\theta-1})$, i.e.  
		\begin{equation*}
			\mathcal{R}_\varepsilon(z) = (I + B_\varepsilon(z)Vq)^{-1}.
		\end{equation*}
		This concludes the proof since we already know that $\mathcal{R}_\varepsilon(z)$ is uniformly bounded in $\mathcal{B}^0(\R^{\theta-1})$. 
	\end{proof}

	\appendix
	\section{Additional results for Section~\ref{sec_loc_part}}
	
	In this section we provide results which are used in Section~\ref{sec_loc_part}. We begin by stating a convenient version of the Schur test.
	\begin{lemma}\label{lem_Schur_test}
		Let $k $ be a measurable function in $\R^{\theta-1}\times \R^{\theta-1} $ with values in $\C^{N \times N}$ and $\widetilde{k} \in L^1(\R^{\theta-1})$ such that 
		\begin{equation*}
			\abs{k(x',y')} \leq \widetilde{k}(x'-y') \quad \textup{for a.e. } x',y' \in \R^{\theta-1}.
		\end{equation*}
		Then, the operator $K : L^2(\R^{\theta-1};\C^N) \to L^2(\R^{\theta-1};\C^N)$ acting as
		\begin{equation*}
		  Kf(x') = \int_{\mathbb{R}^{\theta-1}} k(x',y') f(y') dy'
		\end{equation*}
		is well-defined and bounded and $\| K \|_{L^2(\R^{\theta-1};\C^N) \to L^2(\R^{\theta-1};\C^N)} \leq \|\widetilde{k}\|_{L^1(\R^{\theta-1})}$. Moreover, if additionally $k \in C^1(\R^{\theta-1}\times \R^{\theta-1} ;\C^{N \times N})$  and
		\begin{equation*}
			\sum_{l=1}^{\theta-1} \Bigl|\frac{d}{d x'_l}k(x',y')\Bigr| \leq \widetilde{k}(x'-y') \quad \textup{for a.e. } x',y' \in \R^{\theta-1},
		\end{equation*}
		then $K$ also acts as bounded operator from  $L^2(\R^{\theta-1};\C^N)$ to $H^1(\R^{\theta-1};\C^N)$ and $\| K \|_{L^2(\R^{\theta-1};\C^N) \to H^1(\R^{\theta-1};\C^N)} \leq C\|\widetilde{k}\|_{L^1(\R^{\theta-1})}$.
	\end{lemma}
	\begin{proof}
		The first assertion is an immediate consequence of the Schur test, see for instance \cite[Chapter~III, Example~2.4]{kato}. Next, let us prove the second assertion. We start by choosing $g \in \mathcal{D}(\R^{\theta-1};\C^N)$. Since $k \in C^1(\R^{\theta-1}\times \R^{\theta-1} ;\C^{N \times N})$ and $g$ is compactly supported, dominated convergence shows that  $Kg$  is differentiable and 
		\begin{equation*}
			\partial_{l} (K g)(x') = \int_{\R^{\theta-1}} \frac{d}{d x'_l} k(x',y') g(y') \, dy'.
		\end{equation*}
		Hence, applying the Schur test shows 
		\begin{equation*}
			\norm{K g}_{H^1(\R^{\theta-1};\C^N)} \leq C \|\widetilde{k}\|_{L^1(\R^{\theta-1})} \norm{g}_{L^2(\R^{\theta-1};\C^N)}.
		\end{equation*}
		The rest follows from the fact that $\mathcal{D}(\R^{\theta-1};\C^N) $ is dense in $L^2(\R^{\theta-1};\C^N)$, the completeness of $H^1(\R^{\theta-1};\C^N)$,  and the continuity of $K$ in $L^2(\R^{\theta-1};\C^N)$.
	\end{proof}
	
	In the upcoming proposition we construct a partition of unity for $\R^{\theta-1}$ such that the functions have uniformly bounded derivatives.	
	\begin{proposition}\label{prop_part_unity}
		There exists a partition of unity $(\phi_{n'})_{n' \in \Z^{\theta-1}} $ for $\R^{\theta-1}$ subordinate to $(B(n',1))_{n' \in \Z^{\theta-1}}$  such that $\max_{n' \in\Z^{\theta-1}} \norm{\phi_{n'}}_{W^1_\infty(\R^{\theta-1})} < \infty$. Moreover, there exists a sequence of functions $(\vartheta_{n'})_{n' \in \Z^{\theta-1}}$ with $\supp \vartheta_{n'} \subset B(n',3)$, $0 \leq \vartheta_{n'}  \leq 1$,  $\vartheta_{n'} = 1 $ on $B(n',2)$ for all $n' \in \Z^{\theta-1}$ and $\max_{n'\in\Z^{\theta-1}} \norm{\vartheta_{n'}}_{W^1_\infty(\R^{\theta-1})} < \infty$. 
	\end{proposition}
	\begin{proof}
		Note that since $\theta \in \{2,3\}$, the family $(B(n',3/4))_{n'\in \Z^{\theta-1}}$ is also an open cover of $\R^{\theta-1}$. We start by choosing a function $\phi \in C^\infty(\R^{\theta-1})$ such that $0\leq \phi \leq 1$, $\phi = 1$ on $B(0,3/4)$ and $\supp \phi \subset B(0,1)$. Furthermore, we set ${\widetilde \phi}_{n'}: = \phi(\cdot - n')$	for $n' \in \Z^{\theta-1}$. Then, $0\leq{\widetilde \phi}_{n'} \leq 1$, ${\widetilde \phi}_{n'}= 1$ on $B(n',3/4)$ and $\supp {\widetilde \phi}_{n'}  \subset B(n',1)$.  Next, we fix a bijection $\mathcal{Z} : \N \to \Z^{\theta-1}$ and  set $\phi_{\mathcal{Z}(1)} := \widetilde{\phi}_{\mathcal{Z}(1)}$ and 
		\begin{equation*}
		  \phi_{\mathcal{Z}(j)}= (1-\widetilde\phi_{\mathcal{Z}(1)})\cdot\dots \cdot (1- \widetilde\phi_{\mathcal{Z}(j-1)})\widetilde\phi_{\mathcal{Z}(j)}, \quad j \in \N\setminus\{1\}.
        \end{equation*}
        Then, $\supp \phi_{n'} \subset \supp \widetilde{\phi}_{n'}$, $0 \leq \phi_{n'} \leq 1$ for $n' \in \Z^{\theta-1}$ and on gets via induction for $j \in \N$
		\begin{equation*}
			\sum_{k =1}^{j} \phi_{\mathcal{Z}(k)} = 1- \prod_{k=1}^{j}(1-\widetilde{\phi}_{\mathcal{Z}(k)}).
		\end{equation*}
		This implies $\sum_{n' \in \Z^{\theta-1}} \phi_{n'}(x') = \sum_{j=1}^{\infty} \phi_{\mathcal{Z}(j)}(x')=1$ for $x' \in \R^{\theta-1}$. Furthermore, let $j \in \N$, $l \in \{1,\dots,\theta-1\}$, and $x' \in \R^{\theta-1}$. We estimate
		\begin{equation*}
			\begin{aligned}
				|\partial_l\phi_{\mathcal{Z}(j)}(x')|=& \Big| \partial_l (\widetilde{\phi}_{\mathcal{Z}(j)}(x'))  \prod_{k=1}^{j-1}(1-\widetilde{\phi}_{\mathcal{Z}(k)}(x'))  \\
				&- \sum_{k=1}^{j-1} \widetilde{\phi}_{\mathcal{Z}(j)}(x')  (\partial_l \widetilde{\phi}_{\mathcal{Z}(k)}(x')) \prod_{r=1, r\neq k}^{j-1}(1-\widetilde{\phi}_{\mathcal{Z}(r)}(x')) \Big| \\
				\leq& \sum_{k =1}^j \big| \partial_l \widetilde{\phi}_{\mathcal{Z}(k)}(x')\big| = \sum_{k=1, x' \in B(\mathcal{Z}(k),1)}^j\big|\partial_l\widetilde{\phi}_{\mathcal{Z}(k)}(x')\big|  \\
				\leq& 2^{\theta-1}\norm{\partial _l\phi}_{L^\infty(\R^{\theta-1})},
			\end{aligned}
		\end{equation*}
		where we used that $x' \in \R^{\theta-1}$ can be in at most $2^{\theta-1}$ balls of the form $B(n',1)$ with $n' \in \Z^{\theta-1}$.  This shows that the derivatives of the $\phi_{n'}$'s are uniformly bounded by $2^{\theta-1}\norm{\phi}_{W^1_\infty(\R^{\theta-1};\R)}$. Next we construct the sequence $(\vartheta_{n'})_{n' \in \Z^{\theta-1}}$. To do so, we choose $\vartheta  \in C^\infty(\R^{\theta-1})$ such that $0\leq \vartheta \leq 1$, $\theta = 1$ on $B(0,2)$ and $\supp \vartheta \subset B(0,3)$. Then, we define $\vartheta_{n'}:= \vartheta(\cdot -n')$. The constructed sequence has the claimed properties.
	\end{proof}
	
	Our next goal is to use the functions $\vartheta_{n'}$, $n' \in \mathbb{Z}$, from Proposition~\ref{prop_part_unity} to construct operators based on a uniformly bounded sequence of operators. We start by providing an useful variant of the Cotlar-Stein lemma.
	
	\begin{lemma}\label{lem_cotlar_stein}
		Let $\mathcal{H}$ and $\mathcal{G}$ be Hilbert spaces, let $(\mathcal{A}_{n'})_{n' \in \Z^{\theta-1}}$ be a family of uniformly bounded operators acting from $\mathcal{H}$ to $\mathcal{G}$. Moreover, assume that there exists a number $N \in \N$ such that  for every $n' \in \Z^{\theta-1}$ exist at most $N$ indices $m'\in \Z^{\theta-1}$ such $\mathcal{A}_{n'}^* \mathcal{A}_{m'} $ and $\mathcal{A}_{n'} \mathcal{A}_{m'}^*$ are nonzero operators. Then, the sum $\sum_{n' \in \Z^{\theta-1}, |n'|<n} \mathcal{A}_{n'}$ converges for $n \to \infty$ in the strong sense to a bounded  operator $\mathcal{A}$ (which is also denoted by $\sum_{n' \in \Z^{\theta-1}}^{\textup{st.}} \mathcal{A}_{n'}$). Moreover, its norm can be estimated by
		\begin{equation*}
			\| \mathcal{A}\|_{\mathcal{H} \to \mathcal{G}} \leq N \sup_{n' \in \Z^{\theta-1}} \| \mathcal{A}_{n'}\|_{\mathcal{H} \to \mathcal{G}}.
		\end{equation*}
	\end{lemma}
	\begin{proof}
		Our assumptions guarantee 
		\begin{equation*}
				\sup_{n' \in \Z^{\theta-1}}\sum_{m' \in \Z^{\theta-1}} \| \mathcal{A}_{n'} \mathcal{A}_{m'}^*\|_{\mathcal{G} \to \mathcal{G}}^{1/2} \leq N \sup_{n' \in \Z^{\theta-1}} \|\mathcal{A}_{n'}\|_{\mathcal{H} \to \mathcal{G}}
        \end{equation*}
        and
        \begin{equation*}
				\sup_{n' \in \Z^{\theta-1}}\sum_{m' \in \Z^{\theta-1}} \| \mathcal{A}_{n'}^* \mathcal{A}_{m'}\|_{\mathcal{H} \to \mathcal{H}}^{1/2} \leq N \sup_{n' \in \Z^{\theta-1}} \|\mathcal{A}_{n'}\|_{\mathcal{H} \to \mathcal{G}}.
		\end{equation*}
		Hence, the assertions follow  form the Cotlar-Stein lemma, see \cite[Lemma~18.6.5]{H94}.
	\end{proof}
	\begin{proposition}\label{prop_part_series}
		Let $a \in (0,b)$ for a $b>0$, $(\vartheta_{n'})_{n'\in\Z^{\theta-1}}$ be the sequence from Proposition~\ref{prop_part_unity}, $\vartheta_{n'}^a := \vartheta_{n'}\bigl(\tfrac{\cdot}{a}\bigr)$ for $n'\in \Z^{\theta-1}$, and $(A_{n'})_{n' \in \Z^{\theta-1}}$ be a sequence of uniformly bounded operators in  $\mathcal{B}^0(\R^{\theta-1})$. Then, 
		\begin{equation*}
			A = \sum_{n' \in \Z^{\theta-1}}^{\textup{st.}} \vartheta_{n'}^a A_{n'} \vartheta_{n'}^a   
		\end{equation*}
		is  a well-defined operator in $\mathcal{B}^0(\R^{\theta-1})$ with
        $\| A \|_{0 \rightarrow 0} \leq 11^{\theta-1} \sup_{n' \in \Z^{\theta-1}}\norm{A_{n'}}_{0 \to 0}$.
        Moreover,  if $(A_{n'})_{n' \in \Z^{\theta-1}}$ is also a family of uniformly bounded operators acting from $\mathcal{B}^0(\R^{\theta-1})$ to $\mathcal{B}^1(\R^{\theta-1})$, then $A$ acts also as a bounded operator from  $\mathcal{B}^0(\R^{\theta-1})$ to $\mathcal{B}^1(\R^{\theta-1})$ and $\norm{A}_{0\to1} \leq \tfrac{C}{a} \sup_{n' \in \Z^{\theta-1}}\norm{A_{n'}}_{0 \to 1}$, where $C>0$ does not depend on $a \in (0,b)$.
	\end{proposition}
	\begin{proof}
		Let us start by proving the assertion where we consider $A$ and $A_{n'}$, $n' \in \Z^{\theta-1}$, as operators acting from $\mathcal{B}^0(\R^{\theta-1})$ to $\mathcal{B}^0(\R^{\theta-1})$.  We set $\mathcal{A}_{n'} := \vartheta_{n'}^a A_{n'} \vartheta_{n'}^a$.  Since a fixed ball $B(an',3a)$ overlaps with at most $11^{\theta-1}$ balls of the type $B(am',3a)$, $m'\in \Z^{\theta-1}$, there exist for every $n' \in \Z^{\theta-1}$ at most $N= 11^{\theta-1}$ indices $m' \in \Z^{\theta-1}$ such that $ \mathcal{A}_{n'} \mathcal{A}_{m'}^* \neq 0$ and $\mathcal{A}_{n'}^* \mathcal{A}_{m'} \neq 0$. Moreover, we have for all $n' \in \Z^{\theta-1}$ 
		\begin{equation*}
			\|\mathcal{A}_{n'}\|_{0\to 0} \leq \|\vartheta_{n'}^a\|_{0\to 0}  \|A_{n'}\|_{0\to 0} \|\vartheta_{n'}^a\|_{0\to 0} \leq \|A_{n'}\|_{0\to 0}.
		\end{equation*}
		Thus, by  Lemma~\ref{lem_cotlar_stein} we conclude $\| A \|_{0 \rightarrow 0} \leq 11^{\theta-1} \sup_{n' \in \Z^{\theta-1}}\norm{A_{n'}}_{0 \to 0}$.
		
		Next, we assume that $A_{n'}$, $n' \in \Z^{\theta-1},$ act as  uniformly bounded operators from $\mathcal{B}^0(\R^{\theta-1})$ to $\mathcal{B}^1(\R^{\theta-1})$. Using again the fact that a fixed ball $B(an',3a)$ overlaps with at most $11^{\theta-1}$ balls of the type $B(am',3a)$, $m'\in \Z^{\theta-1}$, shows that  there exist for every $n' \in \Z^{\theta-1}$ at most $N= 11^{\theta-1}$ indices $m' \in \Z^{\theta-1}$ such that $ \mathcal{A}_{n'} \mathcal{A}_{m'}^{0*1} \neq 0$ and $\mathcal{A}_{n'}^{0*1}\mathcal{A}_{m'} \neq 0$, where the  expressions $ {\mathcal{A}_{n'}}^{0*1}$   and  ${\mathcal{A}_{m'}}^{0*1}$ denote the adjoint  operators of $\mathcal{A}_{n'}$ and $\mathcal{A}_{m'}$, respectively, considered as operators mapping from $\mathcal{B}^0(\R^{\theta-1})$ to $\mathcal{B}^1(\R^{\theta-1})$.  Furthermore, for all $n' \in \Z^{\theta-1}$ the inequality
		\begin{equation*}
			\begin{aligned}
				\|\mathcal{A}_{n'}\|_{0\to 1} &\leq \|\vartheta_{n'}^a\|_{1\to 1}  \|A_{n'}\|_{0\to 1} \|\vartheta_{n'}^a\|_{0\to 0} \\
				&\leq \|\vartheta_{n'}^a\|_{1\to 1} \|A_{n'}\|_{0\to1} \\
				&\leq C \|\vartheta_{n'}^a\|_{W^1_\infty(\R^{\theta-1})} \|A_{n'}\|_{0\to 1} \\
				&\leq \frac{C}{a} \|\vartheta_{n'}\|_{W^1_\infty(\R^{\theta-1})} \|A_{n'}\|_{0\to 1}\\
				&\leq \frac{C}{a} \|A_{n'}\|_{0\to 1}
			\end{aligned}
		\end{equation*}
		is valid.
		Thus, Lemma~\ref{lem_cotlar_stein} yields $\norm{A}_{0\to1} \leq \tfrac{C}{a} \sup_{n' \in \Z^{\theta-1}}\norm{A_{n'}}_{0 \to 1}$.
	\end{proof}
	
	\bibliographystyle{abbrv}

\begin{thebibliography}{99}
		 
			\bibitem{AB08}
			A.~R.~Akhmerov and C.~W.~J.~Beenakker.
			Boundary conditions for Dirac fermions on a terminated honeycomb lattice.
			{\em Phys. Rev. B} 77, 085423, 2008.
			
		\bibitem{AMV14}
		N.~Arrizabalaga, A.~Mas, and L.~Vega.
		\newblock Shell interactions for {D}irac operators.
		\newblock {\em J. Math. Pures Appl. (9)} 102(4): 617--639, 2014.
		%
		\bibitem{AMV15}
		N.~Arrizabalaga, A.~Mas, and L.~Vega.
		\newblock Shell interactions for {D}irac operators: on the point spectrum and
		the confinement.
		\newblock {\em SIAM J. Math. Anal.} 47(2): 1044--1069, 2015.
		
		\bibitem{BEHL17}
		J.~Behrndt, P.~Exner, M. Holzmann, and V.~Lotoreichik.
		\newblock Approximation of Schrödinger operators with $\delta$-interactions supported on hypersurfaces.
		\newblock {\em Math. Nachr.} 290(8-9): 1215--1248, 2017.
		%
		\bibitem{BEHL18}
		J.~Behrndt, P.~Exner, M. Holzmann, and V.~Lotoreichik.
		\newblock On the spectral properties of {D}irac operators with electrostatic
		$\delta$-shell interactions.
		\newblock {\em J. Math. Pures Appl.} 111: 47--78, 2018.
		%
		\bibitem{BEHL19}
		J.~Behrndt, P.~Exner, M. Holzmann, and V.~Lotoreichik.
		On Dirac operators in $\R^3$ with electrostatic and Lorentz scalar $\delta$-shell interactions.
		{\em Quantum Studies} 6(3): 295--314, 2019.

		
		\bibitem{BHOP20}
		J.~Behrndt, M.~Holzmann, T.~Ourmieres-Bonafos, and K.~Pankrashkin.
		Two-dimensional Dirac operators with singular interactions supported on closed curves.
		{\em J. Funct. Anal.} 279(8): 108700 (47 pages), 2020.
		
		
		
		
		\bibitem{BHS23}
		J.~Behrndt, M.~Holzmann, and C.~Stelzer-Landauer.
		\newblock Approximation of Dirac operators with $\delta$-shell potentials in the norm resolvent sense, I. Qualitative results.
		\newblock To appear in Math. Nachr., DOI: 10.1002/mana.70004.
		
			\bibitem{BHSS22}
		J.~Behrndt, M.~Holzmann, C. Stelzer-Landauer, and G. Stenzel.
		\newblock Boundary triples and Weyl functions for Dirac operators with singular interactions.
		\newblock {\em Rev. Math. Phys.} 36(2): 2350036 (65 pages), 2024.
		
		
		\bibitem{BHT23} J. Behrndt, M. Holzmann, and M. Tu\v sek.
		Two-dimensional Dirac operators with general $\delta$-shell interactions supported on a straight line.
		{\em J. Phys. A} 56(4): 045201 (29 pages), 2023.
		
		
		
		\bibitem{Ben21}
		B. Benhellal.
		\newblock Spectral properties of the Dirac operator coupled with $\delta$-shell interactions.
		\newblock {\em Lett. Math. Phys.} 112(3): 52 (52 pages), 2022.
		
		\bibitem{BP24}
		B. Benhellal and K. Pankrashkin.
		\newblock Curvature contribution to the essential spectrum of Dirac operators with critical shell interactions.
		\newblock {\em Pure Appl. Anal.} 6(1): 237--252, 2024.
		
		
		
		\bibitem{CLMT21}
		B.~Cassano, V.~Lotoreichik, A.~Mas, and M.~Tu\v{s}ek.
		\newblock General $\delta$-shell interactions for the two-dimensional Dirac operator: self-adjointness and approximation.
		\newblock {\em Rev. Mat. Iberoam.} 39(4): 1443--1492, 2023.
		
		
			
		\bibitem{HT22} 
		L. Heriban and M. Tu\v{s}ek.
		\newblock Non-self-adjoint relativistic point interaction in one dimension.
		{\em J. Math. Anal. Appl.} 516(2) (28 pages), 2022.
		
		\bibitem{HT23} 
		L. Heriban and M. Tu\v{s}ek.
		\newblock Non-local relativistic $\delta$-shell interactions.
		\newblock \emph{ Lett. Math. Phys.} 114(3): 79 (18 pages), 2024.
		
		\bibitem{H94}
		L.~Hörmander.
		\newblock \textit{Analysis of Partial Differential Operators III: Pseudodifferential Operators.}
		\newblock Classics in Mathematics. Springer-Verlag, Berlin, 2007.
		\newblock Reprint of the 1994 edition.
		
		\bibitem{Hu95}
		R.\,J.~Hughes. 
		\newblock{Renormalization of the relativistic delta potential in one dimension.} 
		\newblock{\em  Lett. Math. Phys.} 34(4): 395--406, 1995. 
		
		
		\bibitem{H97}
		R.\,J.~Hughes. 
		\newblock{Relativistic point interactions: approximation by smooth potentials.} 
		\newblock{\em  Rep. Math. Phys.} 39(3): 425--432, 1997. 
		
		\bibitem{H99}
		R.\,J. ~Hughes.
		\newblock Finite-rank perturbations of the Dirac operator.
		\newblock {\em J. Math. Anal. Appl.} 238(1): 67--81, 1999.
		%
		\bibitem{HNVW16}
		T.~Hytönen, J.~v.~Neerven,  M.~Veraar, and L.~Weis.
		\newblock{ \em Analysis in Banach Spaces: Volume I: Martingales and Littlewood-Paley Theory}, volume~48 of {\em A Series of Modern Surveys in Mathematics}.
		\newblock Springer International Publishing, Cham, 2016.
		
		\bibitem{kato}
		T.~Kato.
		{\em Perturbation Theory for Linear Operators.}
		\newblock Classics in Mathematics. Springer-Verlag, Berlin, 1995.
		\newblock Reprint of the 1980 edition.
		
			\bibitem{M17}
			A. Mas.
			Dirac operators, shell interactions and discontinuous gauge functions across the boundary.
			\newblock {\em J. Math. Phys.} 58(22): 022301 (14 pages), 2017.
		
		
		\bibitem{MP18}
		A. Mas and F. Pizzichillo.
		\newblock Klein's Paradox and the relativistic $\delta$-shell interaction in $\mathbb{R}^3$.
		{\em Anal. \& PDE}  11(3): 705--744, 2018. 
		%
		\bibitem{M87}
		J.~Marschall.
		\newblock The trace of Sobolev-Slobodeckij spaces on Lipschitz domains.
		\newblock{\em Manuscripta Math.} 58(1-2): 47--65, 1987.
		
		\bibitem{M00}
		W.~McLean.
		\newblock{\em Strongly Elliptic Systems and Boundary Integral Equations}.
		\newblock Cambridge University Press, Cambridge, 2000.
		
			\bibitem{NGM04}
			K.S.~Novoselov, A.K.~Geim, S.V.~Morozov, D.~Jiang, Y.~Zhang, S.V.~Dubonos, I.V.~Grigorieva, A.A.~Firsov.
			Electric field effect in atomically thin carbon films.
			\emph{Science} 306: 666--669, 2004.
		
		\bibitem{DLMF}
		F.\,W.\,J. Olver, A.\,B. {Olde Daalhuis}, D.\,W. Lozier, B.\,I. Schneider,
		R.\,F. Boisvert, C.\,W. Clark, B.\,R. Miller, B.\,V. Saunders,
		H.\,S. Cohl, and M.\,A. McClain, eds.
		\newblock{\em NIST Digital Library of Mathematical Functions.}
		\newblock https://dlmf.nist.gov, Release 1.2.2 of 2024-09-15.
		
		
		\bibitem{P94}
		T.~Palmer.
		\newblock {\em Banach Algebras and the General Theory of *-Algebras, Vol. 1.}
		\newblock Cambridge University Press, Cambridge, 1994.
		
		\bibitem{R20}
		V.~Rabinovich.
		Boundary problems for three-dimensional Dirac operators and generalized MIT bag models for unbounded domains.
		{\em Russ. J. Math. Phys.} 27(4): 500--516, 2021.
		
		\bibitem{R21}
		V.~Rabinovich.
		Two-dimensional Dirac operators with interactions on unbounded smooth curves.
		{\em Russ. J. Math. Phys.} 28(4): 524--542, 2021.
		%
		%
		\bibitem{R22a}
		V.~Rabinovich.
		Dirac operators with delta-interactions on smooth hypersurfaces in $\mathbb{R}^n$.
		{\em J. Fourier Anal. Appl.} 28(2): 20 (26 pages), 2022.
		
		
		\bibitem{RS72}
		M.~Reed and B.~Simon.
		\textit{Methods of Modern Mathematical Physics I. Functional Analysis}.
		Academic Press, 1972.
		
		\bibitem{RS75}
		M.~Reed and B.~Simon.
		\textit{Methods of Modern Mathematical Physics II. Fourier Analysis, Self-adjointness}.
		Academic Press, 1975.
		
		\bibitem{RS77}
		M.~Reed and B.~Simon.
		\textit{Methods of Modern Mathematical Physics IV. Analysis of Operators}.
		Academic Press, 1977.
		
		\bibitem{S89}
		P. \v{S}eba.
		Klein's paradox and the relativistic point interaction.
		{\em Lett. Math. Phys.} 18(1): 77--86, 1989.
		
		\bibitem{S92}
		M.\,A.~Shubin.
		\newblock  Spectral theory of elliptic operators on non-compact manifolds.
		\newblock \emph{Ast\'erisque} 207: 35--108, 1992.
		
		%
		\bibitem{S70}
		E.~Stein.
		\newblock {\em Singular Integrals and Differentiability Properties of Functions}.
		\newblock Monographs in Harmonic Analysis.
		\newblock Princeton University Press, 1970.
		
		
\bibitem{SL24}
  C.~Stelzer-Landauer.
   \newblock  \emph{Approximation of Dirac Operators with Delta-Shell 
Potentials in the Norm Resolvent Sense.}
   \newblock Monographic Series TU Graz / Computation in Engineering and 
Science. Verlag der TU Graz, Graz, 2025.
		
		\bibitem{T92}
		B.~Thaller.
		\newblock {\em The {D}irac {E}quation}.
		\newblock Texts and Monographs in Physics. Springer-Verlag, Berlin, 1992.
		
		\bibitem{T20} 
		M.~Tu\v{s}ek.
		\newblock Approximation of one-dimensional relativistic point interactions by regular potentials revised.
		\newblock {\em Lett. Math. Phys.} 110(10): 2585--2601, 2020.
		%
		\bibitem{W00}
		J.~Weidmann.
		\newblock {\em Lineare {O}peratoren in {H}ilbertr\"aumen. {T}eil I}.
		\newblock Mathematische Leitf\"aden. B. G. Teubner, Stuttgart, 2000.
		
		\bibitem{Z23}
		M.~Zreik.
		\newblock On the approximation of the $\delta$-shell interaction for the 3-D Dirac operator.
		\newblock \emph{Cubo} 26(3): 489--505, 2024.
		
	\end{thebibliography}

\end{document}